\documentclass[11pt,a4paper,fleqn]{article}

\usepackage{amsmath,amsthm}
\usepackage{framed}
\usepackage[latin1]{inputenc}
\usepackage[dvipsnames]{xcolor}
\usepackage{xspace}
\usepackage{wasysym}
\usepackage{dsfont}
\usepackage{graphicx}
\usepackage{mystyle}
\usepackage{svg}
\usepackage{xcolor,cancel}
\usepackage{url}

\graphicspath{
  {images/},{prebuiltimages/}
}

\usepackage[colorlinks]{hyperref}
\hypersetup{
linkcolor=blue,
citecolor=blue,
}

\newcommand{\E}{{\mathbb E}}
\newcommand{\R}{{\mathbb R}}

\renewcommand{\P}{{\mathbb P}}
\def\argmin{\mathop{\rm arg \; min}\limits}%

\def\KLA{{\rm KLA}}%
\def\MKLA{{\rm MKLA}}%
\def\MUKLA{{\rm UKLA}}%
\def\CV{{\rm CV}}%

\newcommand{\bX}{\mbox{$\boldsymbol{X}$}}
\newcommand{\bY}{\mbox{$\boldsymbol{Y}$}}
\newcommand{\bZ}{\ensuremath{\boldsymbol{Z}}}

\newcommand{\bV}{\boldsymbol{V}}

\newcommand{\bW}{\boldsymbol{W}}
\newcommand{\bP}{\boldsymbol{P}}

\newcommand{\bq}{\boldsymbol{q}}
\newcommand{\bj}{\boldsymbol{j}}
\newcommand{\bE}{\boldsymbol{E}}

\newcommand{\bQ}{\boldsymbol{Q}}

\newcommand{\bdelta}{\boldsymbol{\delta}}

\newcommand{\CB}[1]{{\color{black} #1}} %  Modifications mises en noir
\newcommand{\CM}[1]{{\color{black} #1}} %  Modifications

\numberwithin{equation}{section}

\title{Low-rank matrix denoising for count data using unbiased Kullback-Leibler risk estimation}

\author{ J\'{e}r\'{e}mie Bigot \& Charles Deledalle   \\
\\  Institut de Math\'ematiques de Bordeaux et CNRS  (UMR 5251)   \\ Universit\'e de Bordeaux }

\author{ J\'{e}r\'{e}mie Bigot$^{1}$  \& Charles Deledalle$^{1}$  \\
  \\   Institut de Math\'ematiques de Bordeaux et CNRS  (UMR 5251)$^{1}$ %   \\  International US Center of Excellence (IUSCE)$^{2}$
}

\date{\today}

\begin{document}
\sloppy

\maketitle

\thispagestyle{empty}

\begin{abstract}
\CM{Many statistical studies are concerned with the analysis of observations organized in a matrix form whose elements are count data. When these observations are assumed to follow a Poisson or a multinomial distribution, it is of interest to focus on the estimation of either the intensity matrix (Poisson case) or the compositional matrix (multinomial case) when it is assumed to have a low rank structure. In this setting, it is proposed to construct an estimator minimizing the regularized negative log-likelihood  by a nuclear norm penalty. Such an approach easily yields a low-rank matrix-valued estimator with positive entries  which belongs to the set of row-stochastic matrices in the multinomial case. Then, as a main contribution, a data-driven procedure is constructed to select the regularization parameter in the construction of such estimators by minimizing (approximately) unbiased estimates of the Kullback-Leibler (KL) risk in such models, which generalize Stein's unbiased risk estimation originally proposed for Gaussian data. The  evaluation of these quantities  is a delicate problem, and novel methods are introduced to obtain accurate numerical approximation of such unbiased estimates. Simulated data are used to validate this way of selecting regularizing parameters  for low-rank matrix estimation from count data.  For data following a multinomial distribution, the performances of this approach are also compared to $K$-fold cross-validation. Examples from a survey study and  metagenomics also illustrate the benefits of this methodology for real data analysis.}
\end{abstract}



\section*{Acknowledgments}  J. Bigot is a member of  Institut Universitaire de France (IUF), and this work has been carried out with financial support from the IUF. We gratefully acknowledge Anru Zhang for providing the the COMBO dataset analyzed in this paper.

% \tableofcontents

\section{Introduction}

\subsection{Motivations}

In many applications, it is of interest to estimate a signal matrix $\bX$ based on observations from multiple samples that are organized in a matrix form.  More precisely, we are interested in estimating  $\bX \in \R^{m \times k}$ in the model
\begin{equation}
\bY = \bX + \bW, \label{eq:noisemodel}
\end{equation}
where $\bX = \EE (\bY)$ and $\bW = \bY - \EE (\bY)$. In this paper, we focus on the setting where the data matrix $\bY$ at hand is made of observations that are count data modeled as being random variables sampled from a discrete probability measure such as Poisson  or  multinomial distribution.  For such distributions, a case of  particular interest is the situation when there exists many zeros in the data matrix $\bY$ due to a limited number of available observations. Such a setting makes naive estimators based on likelihood maximization   inappropriate as they result in an estimated signal matrix $\hat{\bX}$ with an excessive number of zeros. To deal with the issue of many observed zeros in the data, we shall consider  applications where $\bX$  can be assumed to have a low rank structure which is often met in practice when there exists a significant correlation between the columns or rows of $\bX$.  There exist various applications involving the observations of count data where such an assumption holds, and, in this paper, we discuss the examples of data from survey  and metagenomics studies.

\subsection{Main contributions}

In low rank matrix denoising with a noise matrix $\bW$ whose element are centered random variables with homoscedastic variances, spectral estimators which shrink or threshold the singular values of the data matrix lead to optimal estimators \cite{donoho2014,GDIEEE14}. However, for  count data sampled from  Poisson or multinomial distributions, the signal matrix $\bX$ has non-negative entries which is a property that is not satisfied by spectral estimators based on the singular value decomposition of $\bY$. Therefore, in this paper, we propose to construct an estimator $\hat{\bX} = (\hat{X}_{ij})$ taking its values in a subset
$
\Ss = \left\{ \Omega(\bZ) \; : \; \bZ \in \RR^{m \times k} \right\},
$
of $m \times k$ matrices with positive entries, parameterized by a (right) invertible mapping $\Omega: \RR^{m \times k} \mapsto  \RR_{+}^{m \times k}$ depending on the data at hand. For observations sampled from Poisson distributions,
\begin{equation}
\Omega(\bZ)_{ij} = \exp(Z_{ij}) , \label{eq:OmegaPoisson}
\end{equation}
while in the multinomial case
\begin{equation}
\Omega(\bZ)_{ij} =
n_{i} \frac{\exp Z_{ij}}{\sum_{q=1}^{k} \exp Z_{iq}},  \label{eq:OmegaMulti}
\end{equation}
for all $1 \leq i \leq m$ and $1 \leq j \leq k$, where $n_i$ is the sum of the entries of the $i$-th row of the data matrix $\bY$.  For Poisson (resp.\ multinomial) data, we shall refer to $\bX = \Omega(\bZ)$ as the intensity (resp.\ compositional) matrix. To be more precise, in the multinomial case, the compositional matrix is the row stochastic matrix $\bP$ with entries
$
\bP_{ij} = \frac{\exp Z_{ij}}{\sum_{q=1}^{k} \exp Z_{iq}},
$
and $\bX = \diag(n_1,\ldots,n_{m}) \bP$.

Then, we propose to construct an estimator $\hat{\bX}_{\lambda} = \Omega(\hat{\bZ}_{\lambda})$ through the following variational approach
\begin{equation}
\hat{\bZ}_{\lambda} \in \argmin_{\bZ \in \RR^{m \times k}}  -\log \tilde{p}(\bY;  \Omega(\bZ) ) + \lambda \|\bZ - \frac{1}{k} \bZ  \mathds{1} \mathds{1}^t\|_{*} \label{eq:variational}
\end{equation}
where $\tilde{p}(\bY; \Omega(\bZ) )$ is the (possibly normalized)
likelihood of the data in a given parametric family of discrete
distributions with parameter $\Omega(\bZ)$, $\| \cdot \|_{*}$ denotes
the nuclear norm of a matrix, $\mathds{1}$ is a $k$ dimensional column
vector of ones, \CM{the notation $u^t$ denotes the transpose of a column
vector $u$}, and $\lambda > 0$ is a regularization
parameter which controls the shrinkage of the singular values of
$\hat{\bZ}_{\lambda}$. The main motivation for using the nuclear norm
$\|\cdot\|_{*}$ as a regularization term comes from the a priori
assumption that the unknown signal matrix $\bZ = \Omega^{-1}(\bX)$ has
a low rank structure.  Incorporating such a prior knowledge is of
particular interest when the count data matrix $\bY$ has many entries
equal to zero which makes a naive estimator based on maximum
likelihood inaccurate. Note that the regularization term in
\eqref{eq:variational} could also be chosen as $ \|\bZ\|_{*}$, but we
have found that only penalizing the singular values of the
row-centered matrix $\bZ - \frac{1}{k} \bZ \mathds{1} \mathds{1}^t$
yields better results in our numerical experiments.

A key issue is obviously the choice of the shrinkage parameter $\lambda$. The main contribution of this paper is to derive an (approximately) unbiased estimator $\MUKLA(\hat{\bX}_{\lambda} )$ of the expected Kullback-Leibler risk $\MKLA(\hat{\bX}_{\lambda},\bX)$ of $\hat{\bX}_{\lambda} = \Omega(\hat{\bZ}_{\lambda})$ to choose $\lambda$ among a possible set $\Lambda$ of values (the precise definitions of  $\MUKLA$ and $\MKLA$ are given in Section \ref{sec:unbiased}). Our approach generalizes to count data the principle of Stein's unbiased risk estimation (SURE) \cite{Stein81} that has been used in  \cite{MR3105401,donoho2014} to select regularization parameters in low-rank matrix denoising from Gaussian data.

More precisely, a data-driven choice of $\lambda$ is obtained by solving the minimization problem
\begin{equation}
\min_{\lambda \in \Lambda} \MUKLA(\hat{\bX}_{\lambda} ),  \label{eq:minMUKLA}
\end{equation}
where $\MUKLA(\hat{\bX}_{\lambda})$ is a random quantity (depending on the data $\bY$) which satisfies
\begin{equation}
\argmin_{\lambda \in \Lambda} \E \left( \MUKLA(\hat{\bX}_{\lambda}) \right) = \argmin_{\lambda \in \Lambda} \MKLA(\hat{\bX}_{\lambda},\bX). \label{eq:argminMUKLA}
\end{equation}
We  focus on data sampled from Poisson and multinomial distributions.  Unbiased estimators of the Kullback-Leibler (KL) risk have already been proposed in \cite{Bigot17} in the Poisson case. In this paper, we derive novel estimators of the MKLA risk in the multinomial case which, to the best of our knowledge, has not been considered so far. A main difficulty arising from the analysis of count data lies in the  evaluation of  estimators  $\MUKLA(\hat{\bX}_{\lambda} )$ of the expected KL risk. Indeed, the numerical evaluation of such estimators requires the computation of terms involving  a notion of differentiation for functions defined on discrete domains whose exact calculation is intractable. Therefore, another contribution of this work  is to develop fast approaches for the numerical approximation of unbiased estimates of the expected KL risk  for estimators defined through the non-differentiable variational problem \eqref{eq:variational}.

\subsection{Applications involving count data with possibly many zeros}

There exits numerous applications involving the observations of count data organized in a matrix form, and we shall focus on examples where the observed data matrix may contain many zeros.

A first example is data collected from a survey. To be more precise, let us consider the problem where $n$ individuals answer a survey for $m$ different products denoted $1 \leq i \leq m$. For each product $i$, each individual is asked to choose among $k$ unique possible answers (the available choices being the same for all products). For instance, for \texttt{product i} the possible answers are: \texttt{very bad, bad, okish, good, very good}, and thus here $k=5$. For each individual, we collect the answers in a $\{0, 1\}^{m \times k}$ matrix, and we sum all these matrices over the $n$ individuals, to finally obtain a  matrix $\bY \in \{0, 1, \ldots, n\}^{m \times k}$ made of count data. The goal is then to estimate the  probability distribution of the answers for each product $i$. In the situation where the individuals do not give their answer for all products $i$, the data matrix $\bY$ may contain many zeros. Nevertheless, for survey data, it is reasonable to assume that the expectation of $\bY$ is approximately low-rank, as for example, many products may generate very close probability distribution of answers.

A second example is metagenomics sequencing  \cite{Cao17} which aims at quantifying the bacterial  abundances in biological samples for microbiome studies. This  technology allows to quantify the human microbiome by using direct DNA sequencing to obtain counts of sequencing reads of marker genes  that can be assigned to a set of bacterial taxa in the observed samples. The bacteria composition in different samples can thus be inferred from such count data. However, for various technical reasons, some rare bacterial taxa might not be measured when using metagenomics sequencing, and this results in zero read counts in the observed data matrix. A naive approach based on count normalization to estimate the taxon composition may thus lead to an excessive number of zeros. Various techniques have been proposed to deal with the issue of observing many zero counts, and we refer to \cite{Cao17}  for a recent overview. When the data matrix is made of the combination of the observed compositions from different individuals, recent studies on co-occurrence pattern \cite{Faust12} and  relationships in microbial communities  \cite{Chaffron10} suggest that searching to estimate a composition matrix  with a low-rank structure is a valid assumption  which leads to estimators with better performances than naive estimators \cite{Cao17}.

\subsection{Related literature on low-rank matrix estimation}

Low-rank matrix estimation in model \eqref{eq:noisemodel} has been extensively studied   in the setting where the additive noise matrix $\bW$ has Gaussian entries with homoscedastic variance \cite{MR3054091,MR3105401,donoho2014,MR3200641}.  However, in many situations, the  noise can be highly heteroscedastic meaning that the amount of variance in the observed data matrix $\bY$ may significantly change from entry to entry.  Examples can be found in photon imaging \cite{SalmonHDW14}, network traffic analysis \cite{Bazerque2013} or genomics for microbiome studies \cite{Cao17}. In such applications, the observations are count data that are modeled by Poisson or multinomial distributions which leads to heteroscedasticity. Motivated by the need for statistical inference from such data, the literature on statistical inference from high-dimension matrices with heteroscedastic noise has thus recently been growing \cite{Bigot17,liu2018,MAL2016,Robin19,Zhang18}.

The problem of  estimating a  low-rank matrix from Poisson data has also been  considered in \cite{Bigot17,Robin19,SalmonHDW14,Cao16}. With respect to these works, our main contributions are to propose a novel approach to chose the regularization parameter in a data-driven way using unbiased risk estimation, and to guarantee to have an estimated intensity matrix with positive entries.

 \CB{In this paper, we focus on the example of count data sets with many observed zeros, but 
our approach also share similarities with the problem of matrix completion from missing count data using nuclear norm penalization   \cite{klopp2014,CIS-295167,Cao16}}.  However,  in  matrix completion it is generally assumed that there exist observations that are missing at random, whereas, in the applications considered in this paper, the observed zeros in the data matrix are the results of under-sampling. Moreover, works in the matrix completion literature generally focus on recovering missing observations whereas our approach is focused on estimating an underlying intensity or compositional matrix.

The estimation of a compositional matrix from multinomial data  under a low rank assumption has been recently considered in \cite{Cao17} with application to metagenomics for quantifying bacterial abundances in microbiome studies. The estimator proposed in \cite{Cao17}  is also defined by regularization of the negative log-likelihood with a penalty term involving the compositional matrix itself that is constrained to have lower bounded entries by a positive constant. Therefore, the approach in \cite{Cao17}  requires the calibration of two tuning parameters that are chosen in practice by cross-validation. Our approach only requires the calibration of the parameter $\lambda$ by  minimizing an unbiased estimator of the expected KL risk. {\color{black} Moreover, for data following a multinomial distribution, our numerical results suggest that the cross-validation methodology proposed in \cite{Cao17} does not lead to a consistent estimation of the KL risk, and thus it yields to a data-driven selection of  the parameter $\lambda$ that is less interpretable.} Further details on the comparison with the work in \cite{Cao17} are given in Section \ref{sec:num} on numerical experiments.

\CM{Finally, it is natural to ask if the methodology that we propose to construct an estimator of the KL risk from either Poisson or multinomial data could be extended to other types of count data. We believe that this work could be extended to data sampled from another discrete exponential family  (e.g.\ in the binomial or negative-binomial case). However, such an extension is beyond the scope of this work, as our approach requires a statistical analysis that is specific to the expression of the KL risk from  a given discrete distribution to derive an appropriate unbiased estimator as done in Sections \ref{sec:Poisson} and \ref{sec:multinomial}. }

\subsection{Organization of the paper}

In Section \ref{sec:unbiased}, we derive the construction of unbiased estimates of the KL risk for any measurable function of the data taking its values in the space of intensity or compositional matrices. In Section \ref{sec:construction}, we detail algorithms to compute the regularized estimator defined by the variational problem \eqref{eq:variational} as well as numerical methods to select the regularization parameter $\lambda$ in a data-driven way by minimizing \eqref{eq:minMUKLA}. Numerical experiments on simulated and read data are described in Section \ref{sec:num} to illustrate the performances of our approach.

%%%%%%%%%%%%

\subsection{Publicly available source code}

For the sake of reproducible research,
a Python code available at the following address: 
  \url{https://www.charles-deledalle.fr/pages/ukla_count_data.php}
implements the proposed estimators and the experiments
carried out in this paper.

\section{Unbiased Kullback-Leibler risk estimation from count data} \label{sec:unbiased}

In this section, we discuss the construction of unbiased estimates of the Kullback-Leibler risk of any estimator $\hat{\bX} = f(\bY)$ where $f$ is a measurable function of the data matrix $\bY$. For $1 \leq i \leq n$ and $1 \leq j \leq m$, we shall denote by $\bE_{ij}$ the $(i,j)$-th element of the canonical basis of $m \times k$ matrices (namely the matrix with all entries equal to zero except the $(i,j)$-th one which is equal to one).

\subsection{The Poisson case} \label{sec:Poisson}

Let us assume that the observations are Poisson data. This corresponds to the setting where the entries $Y_{ij}$ of the data matrix $\bY$ are independent and sampled from a Poisson distribution with parameter $X_{ij} > 0$, meaning that
\begin{equation}
\P\left(Y_{ij} = y | X_{ij} \right) = e^{-X_{ij}} \frac{X_{ij}^y}{y!}, \quad \mbox{for} \quad y \in \NN. \label{eq:modPoisson}
\end{equation}
Following the terminology and notation in \cite{Bigot17}, for a given estimator $\hat{\bX} = f(\bY)$,  we define its expected KL analysis risk as
\begin{equation}
  \MKLA(\hat{\bX},\bX) =  \E \left[ \KLA(\hat{\bX},\bX) \right]
  \mathop{=}^{\text{ Poisson }}
  \sum_{i=1}^{m} \sum_{j=1}^{k}  \E \left[ \hat{X}_{ij} - X_{ij} -  X_{ij} \log \left( \frac{\hat{X}_{ij}}{X_{ij}}  \right)  \right], \label{eq:MKLAPoisson}
\end{equation}
where the above expectation is taken with respect to the distribution of $\hat{\bX} = (\hat{X}_{ij})$, and the (empirical) KL analysis risk of $\hat{\bX}$ is defined as
$$
\KLA(\hat{\bX},\bX)  = \sum_{i=1}^{m} \sum_{j=1}^{k}    \sum_{y \in \NN}  \log\left( \frac{\P\left(Y_{ij} = y |  X_{ij} \right)}{\P\left(Y_{ij} = y |  \hat{X}_{ij}  \right)}\right)  \P\left(Y_{ij} = y |  X_{ij} \right) .
$$

For Poisson data, the problem of deriving unbiased estimate of
$\MKLA(\hat{\bX},\bX)$ has been considered in
\cite{Deledalle,Bigot17}. As shown in \cite{Deledalle}, the use of
Hudson's Lemma (see \cite{hudson1978} and Lemma 2.1 in \cite{Bigot17})
allows to estimate (in an unbiased way) the expectation of the
quantity $X_{ij} \log \hat{X}_{ij}$ in equation
\eqref{eq:MKLAPoisson}, which leads to the following result
(Proposition 2.3 in \cite{Bigot17}).

\begin{prop} \label{prop:SURE-MKLA}
Let $\bY \in \NN^{m \times k}$ be a matrix whose entries are independently sampled from the Poisson distribution \eqref{eq:modPoisson}.
Let $f : \NN^{m \times k} \to \R_{+}^{m \times k}$ be a measurable mapping. Let $1 \leq i \leq m$ and $1 \leq j \leq k$, and denote by $f_{ij} : \NN^{m \times k} \to \R_{+}$ the $(i,j)$-th entry of $\hat{\bX} = f(\bY)$. Then, the quantity
\begin{equation}\label{eq:PUKLA}
\MUKLA(\hat{\bX}) \mathop{=}^{\text{ Poisson }}  \sum_{i=1}^{m} \sum_{j=1}^{k} \hat{X}_{ij} - Y_{ij} \log\left( f_{ij}(\bY - \bE_{ij}) \right),
\end{equation}
is an unbiased estimator of $\displaystyle \MKLA(\hat{\bX},\bX) +  \sum_{i=1}^{m} \sum_{j=1}^{k}   X_{ij} -  X_{ij} \log \left( X_{ij}  \right)$.
\end{prop}

Since the quantity $\sum_{i=1}^{m} \sum_{j=1}^{k}   X_{ij} -  X_{ij} \log \left( X_{ij}  \right)$ does not depend on the function $f$, it is clear that the expression \eqref{eq:PUKLA} yields  an estimator $\MUKLA(\hat{\bX})$ of the expected KL risk satisfying equality \eqref{eq:argminMUKLA}. However, the right-hand side of  \eqref{eq:PUKLA} is numerically difficult to evaluate. Indeed, a naive approach is to evaluate the mapping $f$ on the modified data matrices $\bY - \bE_{ij}$ for all $1 \leq i \leq m$ and $1 \leq j \leq k$, but this is not computationally feasible for moderate to large values of $m$ and $k$. In Section \ref{sec:construction}, we shall thus discuss fast numerical methods to approximate $\MUKLA(\hat{\bX})$ in the Poisson case.

\subsection{The  multinomial case} \label{sec:multinomial}

We now assume that the observations are multinomial data. This corresponds to the setting where the rows of the data matrix $\bY$ are independent realizations of vectors sampled from multinomial distributions with row dependent parameters.

In the case of multinomial data, each row $Y_{i}$ of $\bY$ is thus assumed to follow a multinomial distribution
with parameters $p_i = (p_{ij})_{1 \leq j \leq k}$  and $n_i$ meaning that
\begin{equation}
\PP \left(Y_{i} = (y_{ij})_{1 \leq j \leq k} | p_i \right) = \frac{n_i!}{ \prod_{j=1}^k y_{ij}!} \prod_{j=1}^k p_{ij}^{y_{ij}} \label{eq:modMulti}
\end{equation}
where $y_{ij} \in \{0, \ldots, n_i\}$ such that $\sum_{j=1}^k y_{ij} = n_i$ and $0 < p_{ij} \leq 1$ such that $\sum_{j=1}^k p_{ij} = 1$. The integer $n_i$ is the sum of the observed values for the $i$-th row. We let these numbers varying from one row to another (in the example of a survey, this corresponds to the assumption that  individuals do not answer for all products $i$). Hence, the $n_i$'s are considered to be fixed parameters throughout this section, and it is assumed that $n_i > 0$ for all $1 \leq i \leq m$. In the numerical experiments reported in Section \ref{sec:num}, we consider various examples of data where $n_i$ is small which implies that the maximum likelihood (ML) estimator of the probabilities $p_{ij}$, namely
\begin{equation}
\hat{p}^{ML}_{ij} = Y_{ij}/n_i = Y_{ij} /  \sum_{j'=1}^{k} Y_{ij'}, \label{eq:MLestimator}
\end{equation}
has poor performances. When the data matrix $\bY$ has many zero counts, a widely used approach in compositional data analysis \cite{Aitchison03} is to simply perform zero-replacement in $\bY$ by an arbitrary value $0 < z < 1$ (e.g.\ $z = 0.5$) which yields  the estimator
\begin{equation}
\hat{p}^{zr}_{ij} = \max(Y_{ij}, z) / \sum_{j'=1}^{k}  \max(Y_{ij'}, z). \label{eq:zr}
\end{equation}
In our setting, low rank assumptions on the compositional matrix
$$
\bP = (p_{ij}) \in [0, 1]^{m \times k}
$$
become necessary to improve the accuracy of the estimators $\hat{\bP }^{ML}$ or $\hat{\bP }^{zr}$. We recall that
$$
\EE[Y_{ij}] = n_i p_{ij}, \;  \Var[Y_{ij}] = n_i p_{ij} (1 - p_{ij}), \; \Cov(Y_{ij},Y_{ij'}) = - n_i p_{ij} p_{ij'} \quad \mbox{if} \quad j' \ne j,
$$
and thus the matrix $\bX = \EE(\bY)$ is such that
$$
\bX = \diag(n_1,n_2,\ldots,n_m) \bP.
$$

\subsubsection{Definition of the Kullback-Leibler analysis risk}

For a given estimator $\hat{\bP} = f(\bY)$ of $\bP$ (where $f$ is a mapping taking its values in the space of  row stochastic matrices), we denote its $(i,j)$-th entry by $\hat{p}_{ij}(\bY)$, and an estimator of $\bX$ is obviously given by
$$
\hat{\bX} = \diag(n_1,n_2,\ldots,n_m) \hat{\bP}.
$$
For clarity, we shall  sometimes write $\hat{\bP} = \hat{p}(\bY) = (\hat{p}_{ij}(\bY))$. Now, the expected KL analysis risk of $\hat{\bX}$ (or equivalently of $\hat{\bP}$)  in the multinomial case is defined as
\begin{equation}
%\MKLA(\hat{\bX},\bX) = \E \left[  \KLA(\hat{\bX},\bX)  \right]   =   \sum_{i=1}^m n_i  \E \left[ \sum_{j=1}^{k}  p_{ij}  \log \frac{{p}_{ij}}{\hat{p}_{ij}(\bY)}  \right], \label{eq:MKLAMulti}
\MKLA(\hat{\bX},\bX) = \E \left[  \KLA(\hat{\bX},\bX)  \right]   \mathop{=}^{\text{ Multinomial }}   \sum_{i=1}^m  \E \left[ \sum_{j=1}^{k}  p_{ij}  \log \frac{{p}_{ij}}{\hat{p}_{ij}(\bY)}  \right], \label{eq:MKLAMulti}
\end{equation}
where the above expectation is taken with respect to the distribution of $\hat{\bP}$,  and the (empirical) {\it normalized} KL risk of $\hat{\bX}$ is defined as
\begin{equation}
\KLA(\hat{\bX},\bX)  \mathop{=}^{\text{ Multinomial }}   \sum_{i=1}^{m} \frac{1}{n_i}    \sum_{y \in [n_i]}  \log\left( \frac{\P\left(Y_{i} = y |  p_i \right)}{\P\left(Y_{i} = y |  \hat{p}_i \right)}\right)  \P\left(Y_{i} = y |  p_i\right), \label{eq:normKLrisk}
\end{equation}
with $Y_{i}$ denoting the $i$-th row of $\bY$,  $[n_i] := \left\{ y \in \{0, \ldots, n_i\}^k \; : \; \sum_{j=1}^k y_j = n_i \right\} $. We have chosen a normalized version of the KL risk to take into account the setting where some of the $n_i$'s  take large values. Indeed, in this case, the term  $\E \left[ \sum_{j=1}^{k}  p_{ij}  \log \frac{{p}_{ij}}{\hat{p}_{ij}(\bY)}  \right]$ may dominate the value of the {\it standard} KL risk (that is without normalization), and there is little influence of rows with small values of $n_i$.

%Finally, we shall also consider the
%normalized risk
%\begin{align} \label{eq:MNKLA}
%\MNKLA(\hat{\bX},\bX) &=
%  \sum_{i=1}^m \E \left[ \sum_{j=1}^{k} p_{ij}  \log \frac{{p}_{ij}}{\hat{p}_{ij}(\bY)}   \right]
%\end{align}

\subsubsection{Unbiased estimators of the KL risk} \label{sec:unbiasedKL}

The following theorem is a key result to obtain unbiased estimators of the KL risk in the multinomial case.

\begin{thm} \label{theo:mainMulti}
  Let $Y^{(n)} \in \NN^k$ be a random vector sampled from  a multinomial distribution with parameters  $p = (p_j)_{j=1, \ldots, k}$ and $n  \in \NN_*$ meaning that
\begin{equation}
\PP \left(Y^{(n)} = y | p \right) = \frac{n!}{ \prod_{j=1}^k y_{j}!} \prod_{j=1}^k p_{j}^{y_{j}} \mbox{ for } y \in [n], \label{eq:defMulti}
\end{equation}
where $[n] := \left\{ y \in \{0, \ldots, n\}^k \; : \; \sum_{j=1}^k y_j = n \right\}$. Let $f : \RR^k \to \R$ be any measurable mapping. Then, for any $j \in {1, \ldots, k}$, one has that
\begin{equation}
  \EE\left[
    p_j f(Y^{(n)})
    \right]
  =
  \EE\left[
    \frac{Y^{(n+1)}_j}{n+1} f(Y^{(n+1)} - e_j)
    \right], \label{eq:SURE_Multi}
\end{equation}
  where $e_j$ denotes the $j$-th element of the standard orthonormal basis of $\RR^k$.
\end{thm}

\begin{proof}
  For the sake of notations, we will write the result in terms of an index
  $z \in \{1, \ldots, k\}$ instead of $j$, and we introduce the set
  $
  J^k_z = \{ 1,\ldots,z-1,z+1,\ldots, k\}.
  $
  By definition of the expectation with respect to a multinomial distribution, and using the change of variable $y_{z} \to y_{z}+1$ in a vector $y \in [n]$, we have that
  \begin{align*}
    \hspace{-1.5em}
    \EE
    \left[ p_z f(Y^{(n)}) \right]
    &=
    \sum_{y_1=0}^{+\infty}
    \ldots
    \sum_{y_z=0}^{+\infty}
    \ldots
    \sum_{y_k=0}^{+\infty}
    \left(
    p_z f(y)
    \frac{n!}{\prod_{j=1}^k y_{j}!}
    \prod_{j=1}^k p_{j}^{y_{j}}
    \right)
    {\mathds{1}}_{y \in [n]}(y)
    \\
    &=
    \sum_{y_1=0}^{+\infty}
    \ldots
    \sum_{y_z=1}^{+\infty}
    \ldots
    \sum_{y_k=0}^{+\infty}
    \left(
    p_z f(y-e_z)
    \frac{n!}{(y_{z} - 1)! \prod_{j\in J^k_z} y_{j}!}
    p_{z}^{y_{z} - 1} \prod_{j\in J^k_z} p_{j}^{y_{j}}
    \right)
    \\
    & \CM{ \times
    {\mathds{1}}_{y \in [n+1]}(y)}
    \\
    &=
    \sum_{y_1=0}^{+\infty}
    \ldots
    \sum_{y_z=1}^{+\infty}
    \ldots
    \sum_{y_k=0}^{+\infty}
    \left(
     f(y-e_z)
    \frac{y_z n!}{y_z! \prod_{j\in J^k_z} y_{j}!}
    p_{z}^{y_{z}} \prod_{j\in J^k_z} p_{j}^{y_{j}}
    \right)
    {\mathds{1}}_{y \in [n+1]}(y)
    \\
    &=
    \sum_{y_1=0}^{+\infty}
    \ldots
    \sum_{y_z=1}^{+\infty}
    \ldots
    \sum_{y_k=0}^{+\infty}
    \left(
    y_z f(y-e_z)
    \frac{n!}{\prod_{j = 1}^k y_{j}!}
    \prod_{j = 1}^k p_{j}^{y_{j}}
    \right)
    {\mathds{1}}_{y \in [n+1]}(y)
    \\
    &=
    \sum_{y_1=0}^{+\infty}
    \ldots
    \sum_{y_z=0}^{+\infty}
    \ldots
    \sum_{y_k=0}^{+\infty}
    \left(
    \frac{y_z}{n+1} f(y-e_z)
    \frac{(n+1)!}{\prod_{j = 1}^k y_{j}!}
    \prod_{j = 1}^k p_{j}^{y_{j}}
    \right)
    {\mathds{1}}_{y \in [n+1]}(y)
    \\
    &=
    \EE \left[
      \frac{Y^{(n+1)}}{n+1} f(Y^{(n+1)}-e_z)
      \right],
  \end{align*}
  which completes the proof.
\end{proof}

Now,  let us introduce the notation $\bY^{(n_1, \ldots, n_m)} = \bY$ for a data matrix whose rows are independently sampled from the multinomial distribution \eqref{eq:modMulti} with parameters $p_i$ and $n_i$ for rows $1 \leq i \leq m$. By definition of the expected KL risk \eqref{eq:MKLAMulti} and using Theorem \ref{theo:mainMulti}, we have that (using the notation Const.\ to denote the constant term $ \sum_{i=1}^m  \sum_{j=1}^{k} p_{ij}\log {p}_{ij}$ not depending on the estimator $\hat{\bX}$)
%\begin{align*}
%\MKLA(\hat{\bX},\bX) &\mathop{=}^{\text{ Multinomial }}  \sum_{i=1}^m n_i \EE\left[\sum_{j=1}^{k}
%    p_{ij}
%    \log \frac{{p}_{ij}}{\hat{p}_{ij}(\bY^{(n_1, \ldots, n_m)})}
%    \right]
%  \\
%  &
%  = - \sum_{i=1}^m n_i \EE\left[ \sum_{j=1}^{k}
%    p_{ij}
%    \log \hat{p}_{ij}(\bY^{(n_1, \ldots, n_m)})
%    \right]
%  +
%  \text{Const.}
%  \\
%  &
%  = - \sum_{i=1}^m n_i \EE\left[ \sum_{j=1}^{k}
%    \frac{Y^{(n_1, \ldots, n_i+1, \ldots, n_m)}_{ij}}{n_i+1}
%    \log \hat{p}_{ij}(\bY^{(n_1, \ldots, n_i+1, \ldots, n_m)} - \bE_{ij})
%    \right]
%  +
%  \text{Const.}
%  \\
%  &
%  =
%  -\EE\left[
%    \sum_{i=1}^m \frac{n_i}{n_i+1} \sum_{j=1}^{k}
%    Y_{ij}^{(n_1, \ldots, n_i+1, \ldots, n_m)}
%    \log \hat{p}_{ij}(\bY^{(n_1, \ldots, n_i+1, \ldots, n_m)} - \bE_{ij})
%    \right]
%  +
%  \text{Const.}
%\end{align*}
\begin{align*}
  \hspace{-1.5em}
\MKLA(\hat{\bX},\bX) &\mathop{=}^{\text{ Multinomial }}  \sum_{i=1}^m  \EE\left[\sum_{j=1}^{k}
    p_{ij}
    \log \frac{{p}_{ij}}{\hat{p}_{ij}(\bY^{(n_1, \ldots, n_m)})}
    \right]
  \\
  &
  = - \sum_{i=1}^m  \EE\left[ \sum_{j=1}^{k}
    p_{ij}
    \log \hat{p}_{ij}(\bY^{(n_1, \ldots, n_m)})
    \right]
  +
  \text{Const.}
  \\
  &
  = - \sum_{i=1}^m \EE\left[ \sum_{j=1}^{k}
    \frac{Y^{(n_1, \ldots, n_i+1, \ldots, n_m)}_{ij}}{n_i+1}
    \log \hat{p}_{ij}(\bY^{(n_1, \ldots, n_i+1, \ldots, n_m)} - \bE_{ij})
    \right]
    \\ & \CM{
  +
  \text{Const.}}
  \\
  &
  =
  -\EE\left[
    \sum_{i=1}^m \frac{1}{n_i+1} \sum_{j=1}^{k}
    Y_{ij}^{(n_1, \ldots, n_i+1, \ldots, n_m)}
    \log \hat{p}_{ij}(\bY^{(n_1, \ldots, n_i+1, \ldots, n_m)} - \bE_{ij})
    \right]
    \\ & \CM{
  +
  \text{Const.}}
\end{align*}
Therefore, \CM{in the multinomial case,} one may construct \CM{ an unbiased (up to a constant term) }\CB{estimate of the  KL risk}, which satisfies \eqref{eq:argminMUKLA}, as follows
\begin{align}
\MUKLA(\hat{\bX}) \mathop{=}
-\sum_{i=1}^m
\frac{1}{n_i+1}
\sum_{j=1}^{k}
Y_{ij}^{(n_1, \ldots, n_i+1, \ldots, n_m)}
\log \hat{p}_{ij}(\bY^{(n_1, \ldots, n_i+1, \ldots, n_m)} - \bE_{ij}) \label{eq:estimMUKLAMulti}
\end{align}
%Similarly, we can define a normalized estimator of $\MNKLA(\hat{\bX},\bX)$ by
%\begin{align}
%\MUNKLA(\hat{\bX})=
%-\sum_{i=1}^m
%\frac{1}{n_i+1}
%\sum_{j=1}^{k}
%Y_{ij}^{(n_1, \ldots, n_i+1, \ldots, n_m)}
%\log \hat{p}_{ij}(\bY^{(n_1, \ldots, n_i+1, \ldots, n_m)} - \bE_{ij}).
%\end{align}

\CB{However, $\MUKLA(\hat{\bX})$ is not truly an estimator as it cannot be computed from the data. Indeed, at first glance, it requires simulating $\bY^{(n_1, \ldots, n_i+1, \ldots, n_m)}$  and thus knowing $\bP$. Moreover, computing $\MUKLA(\hat{\bX})$ requires to solve $m$ times the variational problem \eqref{eq:variational} for each ``data matrix'' $\bY^{(n_1, \ldots, n_i+1, \ldots, n_m)}$ with $1 \leq i \leq m$.  Therefore, for moderate to large values of the number $m$ of rows, this is not feasible as the computational cost becomes prohibitive. A natural question is thus the possibility of using such a quantity  in numerical experiments for the purpose of approximating the expected KL risk in real data analysis}. Hence, we propose to rather consider the following estimator
%\begin{align}
%\widehat{\MUKLA}(\hat{\bX}) =
%-\sum_{i=1}^m
%\frac{n_i}{n_i+1}
%\sum_{j=1}^{k}
%Y_{ij}^{(n_1+1, \ldots, n_i+1, \ldots, n_m+1)}
%\log \hat{p}_{ij}(\bY^{(n_1+1, \ldots, n_i+1, \ldots, n_m+1)} - \bE_{ij})  \label{eq:estimMUKLAMultiapprox}
%\end{align}
%\begin{align}
%\widehat{\MUKLA}(\hat{\bX}) =
%-\sum_{i=1}^m
%\alpha_{i}
%\sum_{j=1}^{k}
%Y_{ij}
%\log \hat{p}_{ij}(\bY - \bE_{ij}), \quad \mbox{with} \quad \alpha_{i} = \frac{1}{n_i}  \label{eq:estimMUKLAMultiapprox}
%\end{align}
\begin{align}
\widehat{\MUKLA}(\hat{\bX}) =
-\sum_{i=1}^m
 \frac{1}{n_i}
\sum_{j=1}^{k}
Y_{ij}
\log \hat{p}_{ij}(\bY - \bE_{ij}), \label{eq:estimMUKLAMultiapprox}
\end{align}
as an approximation of the (truly) unbiased \CB{quantity} $\MUKLA(\hat{\bX})$ of the expected KL risk, by replacing in expression \eqref{eq:estimMUKLAMulti} the modified data matrices $\bY^{(n_1, \ldots, n_i+1, \ldots, n_m)}$ by the true data matrix $\bY$ and the ratio $ \frac{1}{n_i+1}$ by $\frac{1}{n_i}$  for all $1 \leq i \leq n$. Nevertheless, as in the Poisson case, it is difficult to numerically   evaluate  the right-hand side of  \eqref{eq:estimMUKLAMultiapprox} as  calculating $\log \hat{p}_{ij}(\bY - \bE_{ij})$ for all $1 \leq i \leq m$ and $1 \leq j \leq k$ is not computationally feasible.  Fast numerical methods to approximate $\widehat{\MUKLA}(\hat{\bX})$ are thus introduced  in Section \ref{sec:construction}.

\CM{From Theorem \ref{theo:mainMulti} and as confirmed by numerical experiments, it appears} that $\widehat{\MUKLA}(\hat{\bX})$, with $\hat{\bX} = f(\bY)$, is rather an approximately unbiased estimation of the expected normalized KL risk (up to constant terms) of the mapping $f$ evaluated at a data matrix that is slightly different from $\bY$. Indeed, recall that the true data matrix is
$$
\bY = \bY^{(n)} = \bY^{(n_1, \ldots, n_m)},
$$
and let us introduce (with some abuse of notation) the modified data matrix
$$
\bY^{(n-1)} = \bY^{(n_1-1, \ldots, n_m-1)}.
$$
Going back to the survey example, we may consider that the observed data (original survey) is denoted by $\bY^{(n)}$ (with  $n_i \geq 1$ for all $i$'s), and that we remove one individual  (picked arbitrarily but who has marked all products), to obtain the new data matrix $\bY^{(n-1)}$. In our numerical experiments, we shall use the data $\bY = \bY^{(n)}$ for  calculating $\widehat{\MUKLA}(\hat{\bX}^{(n)})$ for $\hat{\bX}^{(n)} := f(\bY^{(n)}) $ but its {\it expected value} has to be compared to the expected normalized KL risk (up to constant terms) of the estimator $\hat{\bX}^{(n-1)} = f(\bY^{(n-1)})$. \CM{Indeed, thanks to Equation \eqref{eq:SURE_Multi}, one has that
\begin{equation}
 \E \left[ \frac{ Y_{ij} }{n_i}
\log \hat{p}_{ij}(\bY - \bE_{ij}) \right]  =   \EE\left[p _{ij} \log  \hat{p}_{ij}(\bY^{(n-1)}),    \right] \label{eq:SURE_Multi2}
\end{equation}
and therefore, we  obtain  that  the following  relation holds
\begin{equation}
 \E \left[ \widehat{\MUKLA}(\hat{\bX}^{(n)}) \right] = \MKLA(\hat{\bX}^{(n-1)},\bX) - \sum_{i=1}^m  \sum_{j=1}^{k} p_{ij}\log {p}_{ij}. \label{eq:approxunbiased}
\end{equation}
Consequently, $\widehat{\MUKLA}(\hat{\bX}^{(n)})$ is an unbiased estimator of the expected KL risk  $\MKLA(\hat{\bX}^{(n-1)} ,\bX)$ (up to a constant term) of the estimator $\hat{\bX}^{(n-1)}$.  
Nevertheless, from our using numerical experiments, we have found that for moderate to large values of the number of observations $n$, the expected KL risks $\MKLA(\hat{\bX}^{(n-1)} ,\bX)$  and $\MKLA(\hat{\bX} ,\bX)$ (with $\hat{\bX} = \hat{\bX}^{(n)}$) are very close.  Therefore, one has that  $\widehat{\MUKLA}(\hat{\bX})$ is  (up to a constant term)  an approximately unbiased estimator of the true expected KL risk  $\MKLA(\hat{\bX} ,\bX)$. For small values of $n$, we shall use Monte Carlo simulations from the multinomial model \eqref{eq:modMulti} to numerically approximate the expected KL risks $\MKLA(\hat{\bX}^{(n-1)} ,\bX)$ and $\MKLA(\hat{\bX}^{(n)} ,\bX)$, and to confirm that Equation \eqref{eq:approxunbiased} holds. \\
}
The following proposition indicates how we may always construct the modified matrix  $\bY^{(n-1)}$ from the rows of the data matrix $\bY^{(n)}$.

\begin{prop} \label{prop:YnYnplusone}
  Let $Y^{(n)} \in \NN^k$ be a random vector sampled from the multinomial distribution \eqref{eq:defMulti} with parameters  $p = (p_j)_{j=1, \ldots, k}$ and $n  \in \NN_*$. Then, the following equality (in distribution)  holds between $Y^{(n-1)}$ and $Y^{(n)}$
  \begin{align}
    &Y^{(n-1)} =
    Y^{(n)} - e_{\bj}
  \end{align}
where the index $\bj$ is randomly picked (conditionally on $Y^{(n)}$) among elements in  $\{1,\ldots, k\}$  with probabilities
$$
\hat{p}^{ML}_j = \frac{Y^{(n)}_j}{\sum_{q=1}^k Y^{(n)}_q} = \frac{Y^{(n)}_j}{n}  \quad \mbox{for} \quad 1 \leq j \leq k.
$$
\end{prop}

\begin{proof}
  Let $\tilde{Y}^{(1),q} \in \NN^k$, $1 \leq q \leq n$, be $n$ (latent) iid multinomial random variables
  with parameters  $p = (p_j)_{j=1, \ldots, k}$. Then, we may decompose any random vector $Y^{(n)} \in \NN^k$  sampled from the multinomial distribution \eqref{eq:defMulti} as
  \begin{align}
    Y^{(n)} = \sum_{q=1}^n \tilde{Y}^{(1),q} \label{eq:sum_n}
  \end{align}
Now, for any $q' \in \{1,\ldots,n\}$ (possibly random but chosen independently from $Y^{(n)}$), we can write $Y^{(n-1)}$ as
  \begin{align}
    Y^{(n-1)} = \sum_{\substack{1 \leq q \leq n \\ q\ne q'}} \tilde{Y}^{(1),q} = Y^{(n)} - \tilde{Y}^{(1),q'} \label{eq:sum_nminusone}
  \end{align}
Therefore, we have to remove an arbitrarily chosen random variable $\tilde{Y}^{(1),q'}$ from the sum \eqref{eq:sum_n}. We could choose to pick $q'$ uniformly at random in $\{1,\ldots,n\}$. However, in practice, we clearly do not have access to the observations of $\tilde{Y}^{(1),q'}$ in the decomposition  \eqref{eq:sum_n}, and such a procedure appears to be not computationally feasible at a first glance. Nevertheless, it is clear that  for any $q'$, there exists a random integer  $q \in \{1,\ldots,n\}$ such that $Y^{(1),q'} = e_{q}$ in distribution.

Let us assume that $\bq$ is an integer that is randomly picked (conditionally on $Y^{(n)}$) among elements in  $\{1,\ldots, k\}$  with probabilities
$
\frac{Y^{(n)}_j}{n}
$
for $1 \leq j \leq k$. By conditioning with respect to the distribution of $Y^{(n)}$, it follows that\CM{, for $1 \leq j \leq k$,}
$$
\PP \left(e_{\bq} = e_j \right) = \EE \left[ \PP \left(e_{\bq} = e_j | Y^{(n)}   \right) \right] = \EE \left[ \PP \left( \bq = j | Y^{(n)}   \right) \right] = \EE \left[ \frac{Y^{(n)}_j}{n} \right] = p_j.
$$
Now, if $\bq'$ denotes an integer that is  chosen uniformly at random in $\{1,\ldots,n\}$ (and independently from $Y^{(n)}$), then, by conditioning with respect to the distribution of $\bq'$, one has that
$$
 \PP \left( Y^{(1),\bq'} = e_{j} \right)  =  \frac{1}{n} \sum_{q = 1}^{n}   \PP \left( Y^{(1),q} = e_{j} \right) = \frac{1}{n} \sum_{q = 1}^{n}  p_j = p_j  \quad \mbox{for} \quad 1 \leq j \leq k
$$
Therefore, $Y^{(1),\bq'} = e_{\bq}$ in distribution, and thus the decomposition \eqref{eq:sum_nminusone} allows to complete the proof.
\end{proof}
Hence, thanks to Proposition \ref{prop:YnYnplusone}, it is now clear that the following equality holds (in distribution)
\begin{equation}
\bY^{(n-1)} = \bY^{(n)}  -   \left[\begin{array}{c} e_{\bj_1}^{t}  \\ \vdots \\ e_{\bj_m}^{t} \end{array}\right] \label{eq:YnYnplusone}
\end{equation}
where, for each $1 \leq i \leq m$, the index $\bj_{i}$ is randomly picked (conditionally on $\bY^{(n)}$) among elements in  $\{1,\ldots, k\}$  with probabilities
$
\frac{Y^{(n)}_{ij}}{n_{i}}
$
for $1 \leq j \leq k$. Therefore, the relationship \eqref{eq:YnYnplusone} allows to easily compute the estimator $\hat{\bX}^{(n-1)} = f(\bY^{(n-1)})$ for the purpose of comparing its KL risk to $\widehat{\MUKLA}(\hat{\bX}^{(n)})$ in the numerical experiments.

{\color{black}
\subsubsection{Comparison with cross-validation} \label{section:CV}

The procedure that we propose to estimate the KL risk using $\widehat{\MUKLA}(\hat{\bX})$ can be thought of being reminiscent of the leave-one-out cross-validation method. Indeed, in the formula  \eqref{eq:estimMUKLAMultiapprox} that defines $\widehat{\MUKLA}(\hat{\bX})$, the quantity $\bY^{(i,j)} :=\bY - \bE_{ij}$ may be interpreted as a new data matrix from which one observed count has been removed from the $(i,j)$-th entry. The estimator $\widehat{\MUKLA}(\hat{\bX})$ is then defined as a weighted sum of  $m \times k$ evaluations of the estimator  $\hat{\bX}$  at the new  data matrices $\bY^{(i,j)}$ for $1 \leq i \leq m$ and $1 \leq j \leq k$.  Moreover, recalling that $\hat{p}^{ML}$ denotes the maximum likelihood estimator \eqref{eq:MLestimator}, the following equality 
\begin{equation}
\widehat{\MUKLA}(\hat{\bX}) + \sum_{i=1}^m
\sum_{j=1}^{k}  \hat{p}^{ML}_{i,j}  \log \hat{p}^{ML}_{i,j}  =
\sum_{i=1}^m \sum_{j=1}^{k}
\hat{p}^{ML}_{i,j} 
\log  \frac{\hat{p}^{ML}_{i,j}  }{\hat{p}_{ij}(\bY - \bE_{ij})} \label{eq:UKLAML}
\end{equation}
shows that this way of aggregating the risk of the estimator $\hat{\bX}$  evaluated at the modified data matrices $\bY^{(i,j)}$ with respect the ML estimator $\hat{p}^{ML}$   is different from standard $K$-fold validation  in the multinomial case as recently proposed in \cite{Cao17}. Indeed, if $\bY$ denotes the full data matrix, then the $K$-fold cross-validation (CV) procedure from  \cite{Cao17}  is as follows:

\begin{description}
\item[Step 1:]  randomly split the rows of $\bY$ into two groups of size $m_1 = \lfloor \frac{K-1}{K} m \rfloor$ and $m_2 = m - m_1$ for a total of $L$ times
\item[Step 2:] for $1 \leq \ell \leq L$, denote by $I_{\ell}$ and $I_{\ell}^c$ the row index sets of the two groups, respectively, for the $\ell$-th split. For each $i \in I_{\ell}^c$, one further randomly selects a subset $J_{i,\ell} \subset \{1,\dots,k\}$ of columns indices  with cardinality $k_1 = \lfloor \frac{K-1}{K} k \rfloor$.
\item[Step 3:] for each $\ell$-th split, the training set of indices is defined as 
$$
\Gamma_{\ell} = \{ (i,j) \; : \;  (i,j) \in I_{\ell} \times \{1,\ldots,k\} \mbox{ or } i \in I_{\ell}^c, j \in  J_{i,\ell}  \}  
$$
\CM{such that $\Gamma_{\ell} \subset \{1,\ldots,m\} \times \{1,\ldots,k\} $} corresponds to both complete and incomplete rows. Then, denote $\bY^{(\ell)}$ as the training data matrix such that $\bY^{(\ell)}_{i,j} = \bY_{i,j}$ for all $(i,j) \in \Gamma_{\ell}$ and $\bY^{(\ell)}_{i,j} = 0$ for $(i,j) \notin \Gamma_{\ell}$
\item[Step 4:] for a given estimator $\hat{\bX}$, the resulting $K$-fold CV criteria is then defined as the following prediction error (for the KL risk) on the rows corresponding to indices in $(I_{\ell}^c)_{1 \leq \ell \leq L}$
\begin{equation}
\widehat{\CV}(\hat{\bX}) := \frac{1}{L} \sum_{\ell = 1}^{L} \sum_{i \in I_{\ell}^c} \sum_{j=1}^{k}  \hat{p}^{ML}_{ij}  \log \frac{\hat{p}^{ML}_{ij}}{\hat{p}_{ij}(\bY^{(\ell)})} \label{eq:CV}
\end{equation}
with the usual convention that $u \log(u) = 0$ for $u = 0$.
\end{description}

Comparing, equalities \eqref{eq:UKLAML} and \eqref{eq:CV} it appears that our approach can be interpreted as a kind of leave-one-out cross-validation method where the value of the $(i,j)$-th entry  of the estimator $\hat{\bX}$ (normalized by $n_{i}^{-1}$) obtained from the new data matrix $\bY^{(i,j)}$ is compared to  the $(i,j)$-th entry of the maximum likelihood estimator $\hat{p}^{ML}$. Nevertheless,  for $1 \leq i \leq m$ and $1 \leq j \leq k$,  the new data matrix $\bY^{(i,j)}$ is constructed by removing only one observed count at the $(i,j)$-th entry of $\bY$, and not by setting $\bY^{(i,j)}_{i,j} = 0$ and $\bY^{(i,j)}_{i',j'} = \bY_{i',j'}$ for $(i',j') \neq (i,j)$.

It is clearly beyond the scope of this paper to determine (from a theoretical point of view) if the cross-validation criteria  \eqref{eq:CV} might lead to a consistent estimation of the KL risk, by showing for example that it is unbiased. Nevertheless, for the various numerical experiments carried out in this paper, we report results on the estimation of the KL risk using either $\widehat{\MUKLA}(\hat{\bX})$ or $\widehat{\CV}(\hat{\bX})$ to compare their performances on the analysis of simulated and real data.
}

\subsubsection{A simple class of estimators}

To illustrate the above discussion on the construction of unbiased estimator of the expected KL risk from multinomial data, let us consider a simple class of estimators $\hat{p}^{w}(\bY)$ of the row stochastic matrix $\bP$ whose $(i,j)$-th entry is defined by
\begin{align}
\hat{p}^{w}(\bY)_{ij}
&= \frac{1}{k} + w\frac{Y_{ij}^{+} - \frac{1}{k}\sum_{j'=1}^k Y_{ij'}^{+}}{\epsilon + \sum_{j'=1}^k Y_{ij'}^{+}}, \mbox{ where } Y_{ij}^{+} = \max(0,Y_{ij}),
\end{align}
that is parameterized by a threshold  $0 \leq w \leq 1$ (shrinkage parameter playing the role of regularization) that we wish to select in a data-driven way, and $\epsilon > 0$ is a fixed (but small) constant to take into account possible very low values of the sum $\sum_{j'=1}^k Y_{ij'}$. The estimator $\hat{p}^{w}(\bY) = (\hat{p}^{w}(\bY)_{ij})$ defined in this manner  is a row stochastic matrix for all data matrix including  the matrix $\bY-\bE_{ij}$ whose $(i,j)$-th entry may take the negative value $-1$. Indeed, by simple calculations, one has that
\begin{align}
\sum_{j=1}^k
\hat{p}^{w}(\bY)_{ij}
= 1 + w\frac{\sum_{j=1}^k Y_{ij}^{+}  - \sum_{j'=1}^k Y_{ij'}^{+} }{\epsilon + \sum_{j'=1}^k Y_{ij'}^{+} } = 1
\end{align}
In the setting where $n_i = 0$ or $w = 0$, then one has that $\hat{p}^{w}(Y)_{ij} = \frac{1}{k}$ which corresponds to an estimator with a large bias and a low variance.
To the contrary, when $n_i > 0$ but small and $w \approx 1$, the estimator has a low bias but a large variance. For $w=1$, this estimator essentially corresponds to the estimation \eqref{eq:zr} with zero-replacement in the data matrix $\bY$. Moreover, we can informally remark that, as $n_i \to \infty$, then
$\hat{p}^{w}(\bY)_{ij} \to \frac{1-w}{k} + w p_{ij}$
(asymptotic bias for $\omega > 0$ but no variance). Thus, this suggests that for large values of the $n_i$'s and   $w \approx 1$, we should have that $\hat{p}^{w}(Y)_{ij} \approx p_{ij}$
(asymptotically no bias and no variance). Therefore, when all the $n_i$'s are large, we expect that a value of $w$  close to $1$ should be an optimal choice, while for small values of the $n_i$'s we need to pick $0 \leq w < 1$ in a data-driven way.

Then, we obviously have
\begin{align}\label{eq:simple_minus_1}
  \hat{p}^{w}_{ij}(\bY - \bE_{ij})
  =
  \left\{\begin{array}{ll}
    \frac{1}{k} + w\frac{Y_{ij} - \frac{1}{k}\sum_{j'=1}^k Y_{ij'} }{\epsilon + \sum_{j'=1}^k Y_{ij'}  }, & \mbox{ if }  Y_{ij} = 0 \\
    \frac{1}{k} + w\frac{Y_{ij} - \frac{1}{k}\sum_{j'=1}^k Y_{ij'} + \frac{1-k}{k}}{\epsilon + \sum_{j'=1}^k Y_{ij'} - 1 }, & \mbox{ otherwise }
  \end{array}\right.
\end{align}
and thus, for this class of estimator, it is immediate to compute an estimate of the expected KL risk $\MKLA(\hat{\bX}^{(n)}_{w} ,\bX)$, where $\hat{\bX}^{(n)}_{w} =  \diag(n_1,n_2,\ldots,n_m) \hat{\bP}^{(n)}_{w}(\bY) $ and $\hat{\bP}^{(n)}_{w} = (\hat{p}^{w}(\bY)_{ij})$, using the approximation  \eqref{eq:estimMUKLAMultiapprox}.

To illustrate the construction of unbiased estimator of the expected KL risk  (as discussed in Section \ref{sec:unbiasedKL}) for this simple class of estimators, we consider simulated data sampled from the multinomial model \eqref{eq:modMulti} using the true composition $m \times k$ matrix $\bP$ whose  parametrization though the mapping $\Omega$ is displayed in  Figure \ref{fig:naive}(a) with $m=50$ and $k=50$. To generate the number of observed values for each row of the data matrix, we sample independent realizations $n_1,\ldots,n_m$ from a Poisson distribution with intensity $n_0 = 10$, and {\it those values are held fixed} when sampling a data matrix $\bY$ from the multinomial model \eqref{eq:modMulti}. An example of estimation of $\Omega^{-1}(\bP)$ using the shrinkage estimator $\hat{\bP}^{(n)}_{w}$ with $w=1$  is  displayed in Figure \ref{fig:naive}(b). Then, in Figure \ref{fig:naive}(c), we report numerical results, for different $w \in [0,1]$, on the comparison of the values of the quantity $\widehat{\MUKLA}(\hat{\bX}^{(n)}_{w} )$ to those of the normalized KL risk, as defined in \eqref{eq:normKLrisk}, of either $\hat{\bX}^{(n)}_{w} $ or
$$
\hat{\bX}^{(n-1)}_{w}  =  \diag(n_1-1,n_2-1,\ldots,n_m-1) \hat{\bP}^{(n)}_{w}(\bY^{(n-1)}),
$$
where the modified data matrix $\bY^{(n-1)}$ is given by  \eqref{eq:YnYnplusone}.  Moreover, to stress the importance of replacing in expression \eqref{eq:estimMUKLAMulti} the ratio $ \frac{1}{n_i+1}$ by $\frac{1}{n_i}$, we also report results on the numerical evaluation of the quantity
\begin{align}
\widetilde{\MUKLA}(\hat{\bX}) =
-\sum_{i=1}^m
 \frac{1}{n_i+1}
\sum_{j=1}^{k}
Y_{ij}
\log \hat{p}_{ij}(\bY - \bE_{ij}). \label{eq:estimMUKLAMultiapprox2}
\end{align}

 In Figure \ref{fig:naive}(d), we show the results of Monte Carlo simulations to approximately evaluate the expected values of the quantities  $\KLA(\hat{\bX}^{(n-1)}_{w},\bX)$, $\KLA(\hat{\bX}^{(n)}_{w},\bX)$, $\widehat{\MUKLA}(\hat{\bX}^{(n)}_{w})$ and $\widetilde{\MUKLA}(\hat{\bX}^{(n)}_{w})$.  These numerical results illustrate that the relation \eqref{eq:approxunbiased}  clearly holds, and that $\widehat{\MUKLA}(\hat{\bX}^{(n)}_{w})$ is a relevant estimator of $\MKLA(\hat{\bX}^{(n-1)}_{w},\bX)$ to choose the threshold $w$ in a data-driven way. {\color{black} Finally, in Figure \ref{fig:naive}(e-f), we report results on the evaluation of the KL risk using  $K$-fold cross-validation as described in Section \ref{section:CV}. For $K \in \{2,5,10\}$,  it appears that $\widehat{\CV}(\hat{\bX}^{(n)}_{w} )$ and its expected value are not consistent estimator of the KL risk. In particular the curves $w \mapsto \widehat{\CV}(\hat{\bX}^{(n)}_{w})$  and $w \mapsto \KLA(\hat{\bX}^{(n)}_{w},\bX)$  do not take their minimum at neighboring values of $w$. Indeed, minimizing the KL risk leads to take $\hat{w} \approx 0.4$, while minimizing the CV criterion tends to select a much larger value of $\omega$ (for $K=5$ and $K=10$ the selected value is $\hat{w} \approx 1$, and $\hat{w} \approx 0.75$ for $K=2$ ).}

\begin{figure}[htbp]
  \centering
  \subfigure[Underlying matrix $\Omega^{-1}(\bP)$]{\includegraphics[width=.3\linewidth]{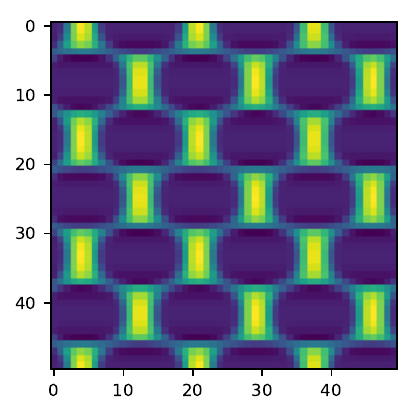}}
  \subfigure[ $\Omega^{-1}(\hat{\bP}^{(n)}_{w})$ with $w=1$]{\includegraphics[width=.3\linewidth]{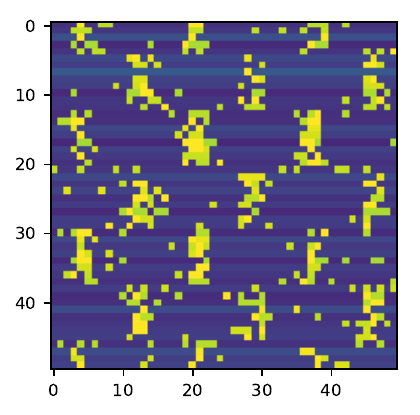}}
  \subfigure[ $\Omega^{-1}(\hat{\bP}^{(n)}_{w})$ with $w=0.4$]{\includegraphics[width=.3\linewidth]{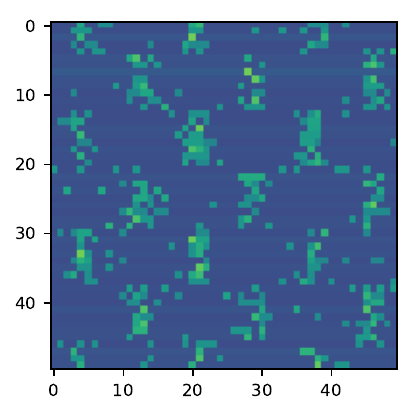}}

  \subfigure[KL risks and their estimates using the data $\bY$]{\includegraphics[width=.45\linewidth]{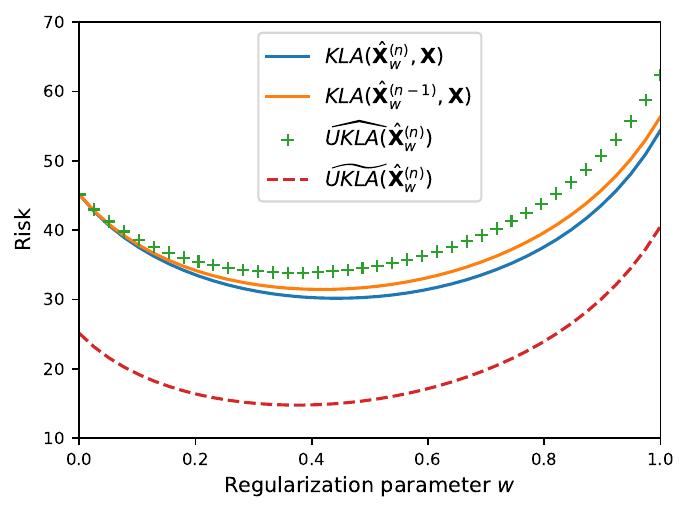}}
  \subfigure[Expected KL risks using Monte Carlo repetitions]{\includegraphics[width=.45\linewidth]{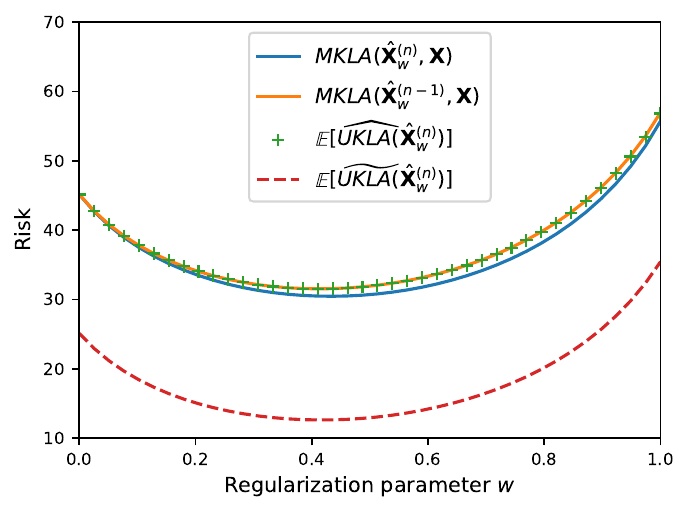}}
  
    \subfigure[{\color{black} KL risk and its estimation using $K$-fold CV}]{\includegraphics[width=.45\linewidth]{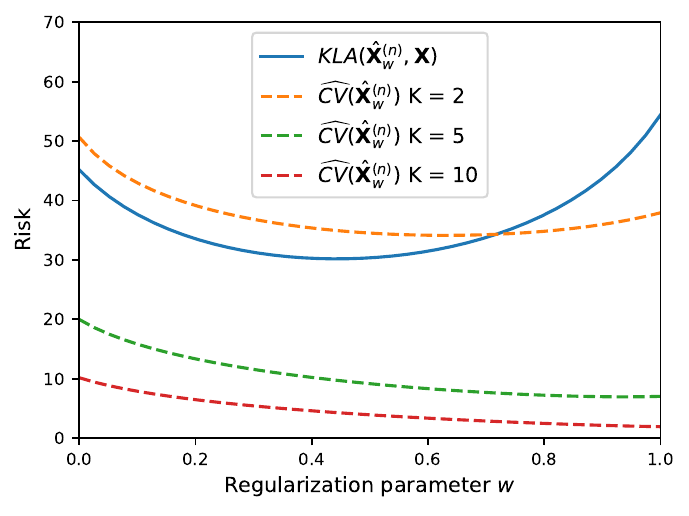}}
  \subfigure[{\color{black} Expected values (using Monte Carlo repetitions) of KL risk and estimators based on $K$-fold CV}]{\includegraphics[width=.45\linewidth]{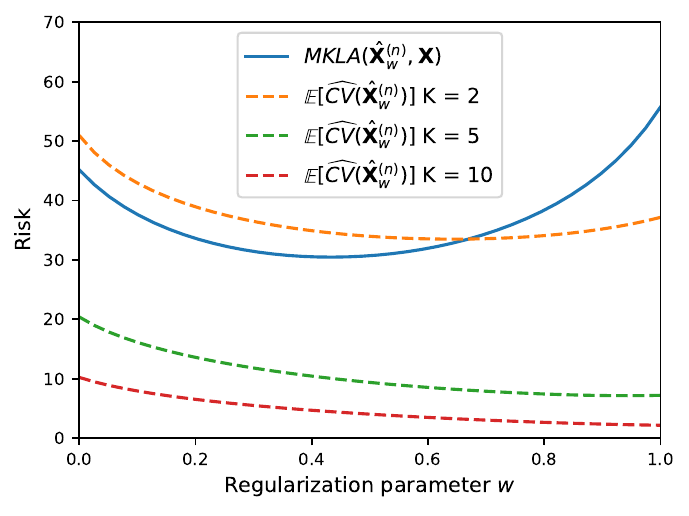}}
  
  \caption{Results on the estimation of KL risks {\color{black} using either generalized SURE or cross-validation (with $L=20$)} for the simple class of estimator $\hat{\bX}^{(n)}_{w}$ with simulated data $\bY$ sampled from  model \eqref{eq:modMulti} (with $m=50$, $k=50$ and $n_0=10$) and $w$ ranging in the interval $[0,1]$. For readability,  the constant term $ \sum_{i=1}^m  \sum_{j=1}^{k} p_{ij}\log {p}_{ij}$ has been added to all estimators of the KL risk {\color{black} based on generalized SURE. }}
  \label{fig:naive}
\end{figure}

\section{Construction of low rank estimators with their risk estimators} \label{sec:construction}

In this section, we discuss the numerical approximation of the expected KL risk of low rank estimators $\hat{\bX}_{\lambda} = \Omega(\hat{\bZ}_{\lambda}) $ where $\hat{\bZ}_{\lambda}$ is obtained by solving the variational problem \eqref{eq:variational} and $\Omega$ is the mapping defined by either \eqref{eq:OmegaPoisson} in the Poisson case or \eqref{eq:OmegaMulti} in the multinomial case.

\subsection{Optimization of the variational problem}

\subsubsection{The algorithm}\label{sec:algo}

Problem \eqref{eq:variational} can be recast as
\begin{align}
  &
  \uargmin{
    \bZ \in \RR^{m \times k}
  }
  \left\{
  E(\bZ)
  =
  F(\bZ)
  +
  G(\bZ)
  \right\}
  \qwhereq
  \choice{
    F(\bZ) =
    -\log \tilde{p}(\bY; \Omega(\bZ))
    \\
    G(\bZ) =
    \lambda \norm{\bZ - \frac{1}{k} \bZ \mathds{1} \mathds{1}^t}_*
  }
  ~.
\end{align}
Provided that $F$ is convex and differentiable and given that
$G$ is convex, this optimization
problem can be solved using the Forward-Backward algorithm
\cite{daubechies2004iterative,combettes2005signal}, {\it a.k.a.},
Iterative Soft-Thresholding Algorithm (ISTA), that reads
for $\gamma > 0$ as
\begin{align}\label{eq:ista}
  \bZ^{(t+1)} = \prox_{\gamma G}(\bZ^{(t)} - \gamma \nabla F(\bZ^{(t)}))~,
\end{align}
where $\prox_{\gamma G}$ is the proximal operator of $\gamma G : \bX \mapsto \gamma G(\bX)$ \cite{moreau1965proximite} defined
for any matrix $\bZ$ as
\begin{align}\label{eq:prox_nuclear_centered}
  \prox_{\gamma G}(\bZ) = \argmin_{\bX} \frac12 \| \bZ - \bX \|_F^2 + \gamma G(\bX)~.
\end{align}

Provided that $F$ is lower bounded and its gradient is $L$-Lipschitz, then
the Forward-Backward algorithm in eq.~\eqref{eq:ista} is guaranteed
to converge for any initialization $\bZ^{(0)}$ to a global minimum of $E$
as soon as the step size $\gamma$ is chosen in the range $(0, 2 / L)$.
ISTA converges in $O(1/t)$ on the objective $E$.
Note that if $F$ is twice differentiable, then $L$ is an upper bound of the
$\ell_2$ operator norm of its Hessian matrix (its maximum eigenvalue).

In this paper, we consider a variant of the Forward-Backward algorithm,
known as Fast ISTA (FISTA) \cite{beck2009fast},
that performs the following iterations
\begin{align}
  \bZ^{(t+1)} &= \prox_{\gamma G}\left(\bZ^{(t)} - \gamma \nabla F(\bZ^{(t)})\right)~,\\
  \rho^{(t+1)}  &= \frac{1+\sqrt{1+4 (\rho^{(t)})^2}}{2}~,\\
  \bZ^{(t+1)} &= \bZ^{(t+1)} + \frac{\rho^{(t)} - 1}{\rho^{(t+1)}} (\bZ^{(t+1)} - \bZ^{(t)})~.
\end{align}
If $F$ is also strictly convex in addition to the other mentioned requirements,
FISTA converges in $O(1/t^2)$ on the objective $E$. \CM{In the numerical experiments carried out in this paper, we shall analyze how the value of the number $T$ of iterations of the above described FISTA algorithm used to compute $\hat{\bX}_{\lambda}$ affects the value of the risk  $\KLA(\hat{\bX}_{\lambda},\bX)$ and the bias of the estimator $\widehat{\MUKLA}(\hat{\bX}_{\lambda})$.}

\subsubsection{The proximal operator of the centered nuclear norm}

The next proposition provides a closed form expression for the proximal operator
of the proposed centered nuclear norm penalty.

\begin{prop}
  Let $\sigma_i$, $u_i$ and $v_i$, $1 \leq i \leq k$, be the $i$-th singular value and
  $i$-th left and right singular vectors, respectively, of the matrix
  $\bZ - \frac{1}{k} \bZ \mathds{1} \mathds{1}^t = \sum_{i=1}^k \sigma_i u_i v_i^t$.
  The proximal operator of $\gamma G$ is given, for $\gamma > 0$, as
  \begin{align}
    \prox_{\gamma G}(\bZ)
    &=
    \frac{1}{k} \bZ \mathds{1} \mathds{1}^t
    +
    \sum_{i=1}^k (\sigma_i - \gamma \lambda)_+ u_i v_i^t~.
  \end{align}
\end{prop}
\begin{proof}
  Remark that eq.~\eqref{eq:prox_nuclear_centered} can be rewritten as
  \begin{align}
    \prox_{\gamma G}(\bZ)
    %&= \argmin_{\bX} \frac12 \| \bZ - \bX \|_F^2 +
    %\gamma \lambda \norm{\bX - \frac{1}{k} \bX \mathds{1} \mathds{1}^t}_*
    %\\
    &= \argmin_{\bX} \frac12 \| \bZ - \bX \|_F^2 +
    \gamma \lambda \norm{P_V \bX}_*
  \end{align}
  where $P_V \bX = \bX - \frac{1}{k} \bX \mathds{1} \mathds{1}^t$ is the
  orthogonal projector into the linear vector subspace $V = \enscond{\bX}{\bX \mathds{1} = 0_m}$.
  Denoting $P_{V^\bot}$ the orthogonal projector on
  the orthogonal subspace $V^\bot$ of $V$, we have
  \begin{align}\label{eq:prox_necluear_centered_pythagorus}
    \frac12 \| \bZ - \bX \|_F^2 +
    \gamma \lambda \norm{P_V \bX}_*
    =
    \frac12 \| P_V (\bZ - \bX) \|_F^2 +
    \frac12 \| P_{V^\bot} (\bZ - \bX) \|_F^2 \nonumber \\ \CM{+
    \gamma \lambda \norm{P_V \bX}_*~}.
  \end{align}
  From eq.~\eqref{eq:prox_necluear_centered_pythagorus}, it necessarily follows  that
   $\bX$  minimizing eq.~\eqref{eq:prox_nuclear_centered}
  must satisfy
  $P_{V^\bot} \bX = P_{V^\bot} \bZ$, hence $\bX \mathds{1} = \bZ \mathds{1}$.
  Using the change of variable
  $\bX \mapsto \bX' + \frac{1}{k} \bZ \mathds{1} \mathds{1}^t$,
  it follows that
  \begin{align}
    \prox_{\gamma G}(\bZ)
    &= \argmin_{\bX} \frac12 \| \bZ - \bX \|_F^2 +
    \gamma \lambda \norm{\bX - \frac{1}{k} \bZ \mathds{1} \mathds{1}^t}_*
    \\
    &= \frac{1}{k} \bZ \mathds{1} \mathds{1}^t + \argmin_{\bX'} \frac12 \| \bZ - \frac{1}{k} \bZ \mathds{1} \mathds{1}^t - \bX' \|_F^2 +
    \gamma \lambda \norm{\bX'}_*
    \\
    &= \frac{1}{k} \bZ \mathds{1} \mathds{1}^t +
    \prox_{\gamma \lambda \norm{\cdot}_*} (\bZ - \frac{1}{k} \bZ \mathds{1} \mathds{1}^t)~.
  \end{align}
  The proximal operator of the nuclear norm
  at $\bX = \sum_{i=1^k} \sigma_i u_i v_i^t$, where
  $\sigma_i$, $u_i$ and $v_i$, $1 \leq i \leq k$, are the $i$-th singular value and
  $i$-th left and right singular vectors, respectively, of the matrix $\bX$,
  is known \cite{Lewis1996,Bauschke2017} to be given by
  \begin{align}
    \prox_{\gamma \lambda \norm{ \cdot}_*} (\bX) =
    \sum_{i=1}^k (\sigma_i - \gamma \lambda)_+ u_i v_i^t~,
  \end{align}
  which concludes the proof.
\end{proof}

\subsubsection{The negative log-likelihood and its gradient in the Poisson case}

Let us first assume that the data follow the Poisson distribution \eqref{eq:modPoisson}. In this setting, the mapping $\Omega : \RR^{m \times k} \to  \RR_+^{m \times k}$ defined by \eqref{eq:OmegaPoisson} is clearly one-to-one with $\Omega^{-1}: \RR_+^{m \times k} \mapsto \RR^{m \times k}$ given by
\begin{align}
  \Omega^{-1}(\bX)_{ij} &= \log X_{ij}.
\end{align}
Hence, for any $\bZ \in \RR^{m \times k}$, the negative log-likelihood of Poisson data with parameters $\bX$
satisfies
\begin{align*}
  F(\bZ) = -\log \tilde{p}(\bY; \Omega(\bZ))
  =& \sum_{i=1}^{m} \sum_{j=1}^{k} \exp(Z_{ij}) - \sum_{i=1}^{m} \sum_{j=1}^{k}  Y_{ij} Z_{ij} +  \sum_{i=1}^{m} \sum_{j=1}^{k} \log(Y_{ij}!) \\
  =& \sum_{i=1}^{m} \sum_{j=1}^{k} \exp(Z_{ij}) - \sum_{i=1}^{m} \sum_{j=1}^{k}  Y_{ij} Z_{ij} +
  \text{Const.}
\end{align*}
where Const.~denotes terms not depending on $\bZ$.
This implies that, in the Poisson case, the variational problem \eqref{eq:variational} is equivalent to
\begin{align}
  \uargmin{
    \bZ \in \RR^{m \times k}
    }
  \sum_{i=1}^{m} \sum_{j=1}^{k} \exp(Z_{ij})
  - \sum_{i=1}^{m} \sum_{j=1}^{k}  Y_{ij} Z_{ij}
  + \lambda \|\bZ - \frac{1}{k} \bZ \mathds{1} \mathds{1}^t\|_{*}~.
\end{align}
It follows that the gradient of $F$ is given by
\begin{align*}
  (\nabla F(\bZ))_{ij}
  = \exp(Z_{ij}) - Y_{ij}
  = \Omega(\bZ)_{ij} - Y_{ij}~,
\end{align*}
and its Hessian is given, for $1 \leq i, r \leq m$ and
$1 \leq j, s \leq m$ with $(i,j) \ne (r, s)$, by
\begin{align*}
  \pdd{F(\bZ)}{Z_{ij}}
  = \exp(Z_{ij})
  = \Omega(\bZ)_{ij}
  \qandq
  \pd{^2 F(\bZ)}{Z_{ij} Z_{r,s}}
  = 0~.
\end{align*}
The $\ell_2$ operator norm of the Hessian of $F$ is thus unbounded on $\RR^{m \times k}$,
meaning that the function $F$ does not admit a Lipschitz gradient.
We will assume that the FISTA sequence satisfies, for all $t \geq 0$,
\begin{align}
  \Omega(\bZ^{(t)}) \in (0, L_{\bY})^{m \times k} \subset \RR^{m \times k}
  \qwhereq
  L_{\bY} = \max_{\substack{1 \leq i \leq m \\ 1 \leq j \leq k}} |Y_{ij}|~,
\end{align}
such that convergence can be ensured provided $\gamma$
is chosen in the range $(0, 2/ L_{\bY})$. In practice, we will choose $\gamma = 1 / L_{\bY}$.

\subsubsection{The negative log-likelihood and its gradient in the multinomial case}

Now, let us consider the  setting of multinomial data \eqref{eq:modMulti}. We introduce the set of row stochastic matrices with positive entries defined by
\begin{align}
\Ss^{m \times k}&= \enscond{\bP \in \RR^{m \times k}}{p_{ij} > 0, \bP \mathds{1} = \mathds{1}}.
\end{align}
Then, by a slight abuse of notation, we shall consider the mapping
\begin{equation}
\Omega(\bZ)_{ij} = \frac{\exp Z_{ij}}{\sum_{q=1}^{k} \exp Z_{iq}},  \label{eq:OmegaMultitrue}
\end{equation}
for all $1 \leq i \leq m$ and $1 \leq j \leq k$, instead of \eqref{eq:OmegaMulti} as we focus on the estimation of a row stochastic matrix $\bP$ in the multinomial case rather than on the expectation $\bX = \diag(n_1,n_2,\ldots,n_m) \bP$ of the data matrix $\bY$.

The mapping $\Omega : \RR^{m \times k} \to \Ss^{m \times k}$ defined by \eqref{eq:OmegaMultitrue} is not one-to-one, but it admits
a right inverse $\Omega^{-1}: \Ss^{m \times k} \mapsto \RR^{m \times k}$ given by
\begin{align}
  \Omega^{-1}(\bP)_{ij} &=
  \log P_{ij} - \frac1k \sum_{q=1}^k \log P_{iq},
  \qforq 1 \leq i \leq k.
\end{align}
Therefore, for any $\bZ \in \RR^{m \times k}$, the {\it normalized}
negative log-likelihood of multinomial data with parameters $\bP$ and $n_1,\ldots,n_m$
satisfies
\begin{align*}
  F(\bZ) =& -\log \tilde{p}(\bY; \Omega(\bZ))
  \\
  =&
  -\sum_{i=1}^m \frac1{n_i} \log(n_i!)
  + \sum_{i=1}^m \frac1{n_i} \sum_{j=1}^{k} \log(Y_{ij}!) \\
  & \CM{
  -
  \sum_{i=1}^m \frac1{n_i} \sum_{j=1}^{k} Y_{ij} \left( Z_{ij} - \log\left( \sum_{q=1}^{k} \exp Z_{iq} \right)\right)}
  \\
  =&
  -
  \sum_{i=1}^m \frac1{n_i} \sum_{j=1}^{k} Y_{ij} Z_{ij}
  +
  \sum_{i=1}^m \frac1{n_i} \sum_{j=1}^{k} Y_{ij} \log\left( \sum_{q=1}^{k} \exp Z_{iq} \right)
  +
  \text{Const.}
  %+
  %\sum_{i=1}^m y_{ik} \log\left( 1 + \sum_{q=1}^{k} \exp Z_{iq} \right) + C\\
  \\
  =&
  -
  \sum_{i=1}^m \sum_{j=1}^{k} \frac{Y_{ij}}{n_i} Z_{ij}
  +
  \sum_{i=1}^m \log\left( \sum_{q=1}^{k} \exp Z_{iq} \right) + \text{Const.}
\end{align*}
where Const.\ denotes terms not depending on $\bZ$.
Hence, in the multinomial case, the variational problem \eqref{eq:variational} is equivalent to
\begin{align}
  \uargmin{
    \bZ \in \RR^{m \times k}
    }
  -
  \sum_{i=1}^m \sum_{j=1}^{k} \frac{Y_{ij}}{n_i} Z_{ij}
  +
  \sum_{i=1}^m \log\left( \sum_{q=1}^{k} \exp Z_{iq} \right)
  + \lambda \|\bZ  - \frac{1}{k} \bZ \mathds{1} \mathds{1}^t \|_{*}
  ~.
\end{align}
It follows that the gradient of $F$ is given by
\begin{align}
  (\nabla F(\bZ))_{ij} &=
  -\frac{Y_{ij}}{n_i}
  +
  \frac{\exp \bZ_{ij}}{\sum_{q=1}^{k} \exp \bZ_{iq}}
  =
  \Omega(\bZ)_{ij} - \frac{Y_{ij}}{n_i}~,
\end{align}
and its Hessian is given, for $1 \leq i, r \leq m$ and
$1 \leq j, s \leq m$ with $(i,j) \ne (r, s)$, by
\begin{align*}
  \pdd{F(\bZ)}{Z_{ij}} =
  \Omega(\bZ)_{ij} (1 - \Omega(\bZ)_{ij})
\end{align*}
\CM{and
\begin{align*}
  \pd{^2F(\bZ)}{Z_{ij} \partial Z_{r,s}} &=
  \choice{
    -\Omega(\bZ)_{ij} \Omega(\bZ)_{rs} & \ifq r = i\\
    0 & \otherwise
  }~.
\end{align*}
}
Since $0 < \Omega(\bZ)_{ij} < 1$ and $\sum_{ij} \Omega(\bZ)_{ij} = 1$,
we have by definition of the $\ell_1$ operator norm
\begin{align}
  \norm{\pd{^2 F(\bZ)}{\bZ}}_1
  &=
  \max_{\substack{1 \leq i \leq m\\ 1 \leq j \leq k}}
  \sum_{r=1}^m \sum_{s=1}^k \left|\pd{^2F(\bZ)}{Z_{ij} \partial Z_{r,s}} \right|
  \\
  &=
  \max_{\substack{1 \leq i \leq m\\ 1 \leq j \leq k}}
  \Omega(\bZ)_{ij} (1 - \Omega(\bZ)_{ij})
  +
  \sum_{s=1}^k \Omega(\bZ)_{ij} \Omega(\bZ)_{is}
  -
  \Omega(\bZ)_{ij}^2
  \\
  &=
  \max_{\substack{1 \leq i \leq m\\ 1 \leq j \leq k}}
  2 (\Omega(\bZ)_{ij} - \Omega(\bZ)_{ij}^2)
  \leq
  \frac12~.
\end{align}
The Hessian being self-adjoint, its $\ell_1$ norm equals its $\ell_\infty$ norm.
%\begin{align}
%  \norm{\pdd{F(\bZ)}{\bZ}}_\infty = \norm{\pdd{F(\bZ)}{\bZ}}_1 \leq \frac12 ~.
%\end{align}
By Hölder inequality and definition of operator norms,
for all matrices $\bZ \in \RR^{m \times k}$ and $\bV \in \RR^{m \times k}$,
it follows that
\begin{align}\label{eq:lipschitz_multinomial_inequality}
  \frac{
    \norm{ \pd{^2 F(\bZ)}{\bZ^2} [ \bV ] }_2
  }{
    \norm{ \bV }_2
  }
  \leq
  \norm{\pdd{F(\bZ)}{\bZ}}_2 \leq
  \sqrt{\norm{\pdd{F(\bZ)}{\bZ}}_1 \norm{\pdd{F(\bZ)}{\bZ}}_\infty}
  =
  \frac12~,
\end{align}
where the $\ell_2$ vector norm of a matrix $\bV$ has to be understood
here as its Frobenius norm ($\bV$ being considered as a vector of a Hilbert space
indexed by two indices).
In particular, the inequality in the left-hand-side holds true for $\Omega(\bZ)_{ij} = \frac12$ if $1 \leq j \leq 2$ and $0$ otherwise, and
$V_{i1} = \frac1{\sqrt{2m}}$, $V_{i2} = -\frac1{\sqrt{2m}}$, and
$V_{ij} = 0$ for $3 \leq j \leq k$. Remark that $\norm{\bV}_2 = 1$ and we have
\begin{align}
  & \norm{ \pd{^2 F(\bZ)}{\bZ^2} [ \bV ] }_2^2
  =
  \sum_{i=1}^m \sum_{j=1}^k \left(\pd{^2 F(\bZ)}{\bZ^2} [ \bV ] \right)^2_{ij}
  =
  \sum_{i=1}^m \sum_{j=1}^k \left( \sum_{r,s} \pd{^2 F(\bZ)}{Z_{ij} \partial Z_{r,s}} V_{r,s} \right)^2
  \\
  &=
  \sum_{i=1}^m \sum_{j=1}^k
  \left(
  \Omega(\bZ)_{ij} V_{ij}
  -
  \sum_{s=1}^k
  \Omega(\bZ)_{ij} \Omega(\bZ)_{i,s} V_{i,s}
  \right)^2
  =
  \sum_{i=1}^m \sum_{j=1}^2
  \frac14
  \left(
  V_{ij}
  -
  \frac12
  \sum_{s = 1}^2
  V_{i,s}
  \right)^2
  \\
  &
%  \sum_{i=1}^m \sum_{j=1}^k
%  \Omega(\bZ)_{ij}^2
%  \left(
%  V_{ij}
%  -
%  \sum_{s=1}^k
%  \Omega(\bZ)_{i,s} V_{i,s}
%  \right)^2
%  \\
%  &=
  =
  \sum_{i=1}^m \sum_{j=1}^2
  \frac14
  \left(
  \frac{1}{\sqrt{2m}}
  -
  0
  \right)^2
  =
  \sum_{j=1}^2 \sum_{i=1}^m
  \frac1{8 m}
  = \frac14~.
\end{align}
From eq.~\eqref{eq:lipschitz_multinomial_inequality}, it follows that
\begin{align}
  L = \sup_{\bZ \in \RR^{m \times k}} \norm{ \pd{^2 F(\bZ)}{\bZ^2} }_2 = \frac12~.
\end{align}
As a consequence, by choosing $\gamma$ in the range $(0, 4)$,
the sequence $E(\bZ^{(t)})$ as defined by FISTA is guaranteed to converge
for any initialization $\bZ^{(0)}$ to its minimum value. In practice, we took $\gamma = 2$ in the multinomial case.

\subsection{Fast evaluation of the proposed unbiased estimators}

The unbiased estimators of the Kullback-Leibler analysis risks
for Poisson and multinomial data, introduced in \eqref{eq:PUKLA}
and \eqref{eq:estimMUKLAMulti}, respectively, require to evaluate a matrix
$\bQ \in \RR^{m \times k}$ whose elements can be expressed, for any $1 \leq i \leq m$ and $1 \leq j \leq k$, as
\begin{align}\label{eq:y_g_prod_discrete_diff}
  Q_{ij} = g_{ij}(\bY - \bE_{ij})
\end{align}
where  either $g_{ij} = \log \circ f_{ij}$ or  $g_{ij} = \log \circ \,\hat{p}_{ij}$, respectively.
In this paper, the considered spectral estimators evaluate $g_{ij}(\bY - \bE_{ij})$ with
time complexity in $\Oo(m k)$. As a consequence, since this operation must be repeated
for all $1 \leq i \leq m$ and $1 \leq j \leq k$, the overall complexity to evaluate
$\bQ$ requires $\Oo(m^2 k^2)$ operations.
Due to this prohibitive quadratic complexity to evaluate eq.~\eqref{eq:y_g_prod_discrete_diff},
we need to perform some approximations. In the next proposition we will build a
biased estimator of $\bQ$ that can be evaluated in
linear complexity, {\it i.e.} in $\Oo(m k)$.

%\begin{center}
%  \CR{Attention: les notations $t$ et $T$ sont déjà utilisées dans FISTA.
%    $i$, $j$, $k$, $n$, $m$, sont utilisées également.
%    $l$ ne l'est pas mais $L$ l'est (Lispschitz constant)...
%    peut être prendre $\ell$ et $\Ll$ alors?
%  }
%\end{center}

\begin{thm}\label{thm:g_discrete_diff_estimator}
  Let $\Ll \in \NN_*$. Assume $g$ is of class $C^{\Ll+1}$.
  Let $\bZ^{(\ell)} \in \RR^{m \times k}$, $1 \leq \ell \leq \Ll$, be a sequence of independent
  random matrices such that
  $\EE[(Z^{(\ell)}_{ij})^2] = 1$ and $\EE[Z^{(\ell)}_{ij} Z^{(\ell)}_{r,s}] = 0$ for all $1 \leq i,r \leq m$
  and $1 \leq j,s \leq k$ with $(i,r) \ne (j,s)$.
  Let us define recursively the $\ell$-th directional derivative of $g$ in directions
  $\bdelta_1, \bdelta_2, \ldots, \bdelta_{\ell} \in \RR^{m \times k}$ as
  \begin{align}
    \pd{^{\ell} g}{\bY^{\ell}}(\bY)[\bdelta_1, \ldots, \bdelta_{\ell}]
    =
    \lim_{\epsilon \to 0}
    \frac{1}{\epsilon}
    \left[
      \pd{^{\ell-1} g}{\bY^{\ell-1}}(\bY + \epsilon \bdelta_{\ell})[\bdelta_1, \ldots, \bdelta_{\ell-1}] \right. \nonumber \\
\CM{   \left.    -
      \pd{^{\ell-1} g}{\bY^{\ell-1}}(\bY)[\bdelta_1, \ldots, \bdelta_{\ell-1}]
      \right]}
  \end{align}
  with $\pd{^0 g}{\bY^0}(\bY) = g(\bY)$.
  The following random matrix $\hat{\bQ}$
  defined as
  \begin{align}\label{eq:g_discrete_diff_estimator}
    \hat{\bQ} =
    \sum_{\ell=0}^{\Ll}
    \frac{(-1)^{\ell}}{\ell!}
      %\left(\prod_{t'=1}^t Z_{ij}^{(t')} \right)
      \left(\bZ^{(1)} \odot \ldots \odot \bZ^{(\ell)}\right)
      \pd{^{\ell} g}{\bY^{\ell}}(\bY)[\bZ^{(1)}, \ldots, \bZ^{(\ell)}]
  \end{align}
  where $\odot$ denotes the Hadamard (element-wise) product,
  is an estimator of $\bQ$ defined in eq.~\eqref{eq:y_g_prod_discrete_diff}
  Moreover, its bias is given, for any $1 \leq i \leq m$ and $1 \leq j \leq k$, by
  \begin{align}
    (\EE[\hat{\bQ}] - \bQ)_{ij}
    =
    \frac{(-1)^{\Ll}}{\Ll!}\int_{0}^{1}
    \alpha^T
    \pd{^{\Ll+1} g}{Y_{ij}^{\Ll+1}}(\bY - \alpha \bE_{ij})
    \;
    d\alpha
    ~,
  \end{align}
  where the expectation is with respect to the sequence of random matrices
  $\bZ^{(\ell)}$.
\end{thm}

Before turning to the proof of Theorem \ref{thm:g_discrete_diff_estimator},
let us introduce a first Lemma.

\begin{lem}\label{lem:trace_estimator}
  Let $h : \RR^{m \times k} \to \RR^{m \times k}$ be a linear function.
  Let $\bZ \in \RR^{m \times k}$ be a random matrix such that
  $\EE[Z_{ij}^2] = 1$ and $\EE[Z_{ij} Z_{r,s}] = 0$ for all $1 \leq i,r \leq m$
  and $1 \leq j,s \leq k$ with $(i,r) \ne (j,s)$.
  Then
  \begin{align}
    \EE[Z_{ij} h(\bZ)]
    =
    h(\bE_{ij})~.
  \end{align}
\end{lem}
\begin{proof}
  Using the linearity of $h$ and the assumptions on $\bZ$,
  the proof simply reads as follow
  \begin{align}
    \EE[Z_{ij} h(\bZ)]
    &=
    \EE \left[ Z_{ij} h\left( \sum_{r,s} Z_{r,s} \bE_{r,s} \right) \right]
    =
    \sum_{r,s} \EE[Z_{ij} Z_{r,s} h(\bE_{r,s})]
    \\
    &=
    \EE[Z_{ij}^2] h(\bE_{ij})
    +
    \sum_{(r,s) \ne (i,j)} \EE[Z_{ij} Z_{r,s}] h(\bE_{r,s})
    =
    h(\bE_{ij})~.
  \end{align}
\end{proof}

We are now equipped to turn to the proof of Theorem \ref{thm:g_discrete_diff_estimator}.

\begin{proof}[Proof of Theorem \ref{thm:g_discrete_diff_estimator}]
  By Taylor expansion of order $\Ll$, we have
  \begin{align}\label{eq:g_discrete_diff}
    g(\bY - \bE_{ij})_{ij}
    =
    \sum_{\ell=0}^{\Ll}
    \frac{(-1)^{\ell}}{\ell!} \pd{^{\ell} g_{ij}}{Y_{ij}^t}(\bY)
    +
    R_{ij}
  \end{align}
  where $R_{ij}$ is the remainder given by
  \begin{align}
    R_{ij} = \frac{(-1)^{\Ll}}{\Ll!}\int_{0}^{1}
    \alpha^T
    \pd{^{\Ll+1} g_{ij}}{Y_{ij}^{\Ll+1}}(\bY - \alpha \bE_{ij})
    \;
    d\alpha~.
  \end{align}
  Note that we can rewrite eq.~\eqref{eq:g_discrete_diff} in terms of the directional derivatives
  for the directions $\bdelta_{\ell} = \bE_{ij}$ leading to
  \begin{align}
    g(\bY - \bE_{ij})_{ij}
    =
    \sum_{\ell=0}^{\Ll}
    \frac{(-1)^{\ell}}{\ell!} \left[ \pd{^{\ell} g_{ij}}{\bY^{\ell}}(\bY)[\underbrace{\bE_{ij}, \ldots, \bE_{ij}}_{\text{$t$ times}}] \right]
    +
    R_{ij}~.
  \end{align}
  Recalling that the $\ell$-th directional derivative is a $t$-linear mapping, {\it i.e.},
  linear with respect to each of its $t$ directions, we have
  by virtue of Lemma \ref{lem:trace_estimator}
  \begin{align}
    g(\bY - \bE_{ij})_{ij}
    =
    \sum_{\ell=0}^{\Ll}
    \frac{(-1)^{\ell}}{\ell!} \EE \left[
      %\left(\prod_{t'=1}^t Z_{ij}^{(t')} \right)
      \left(Z_{ij}^{(1)} \times \ldots \times Z_{ij}^{(\ell)}\right)
      \pd{^k g_{ij}}{\bY^k}(\bY)[\bZ^{(1)}, \ldots, \bZ^{(\ell)}]
      \right]
    +
    R_{ij},~
  \end{align}
  which concludes the proof of Theorem  \ref{thm:g_discrete_diff_estimator}.
\end{proof}

In practice, we will consider $\bZ^{(\ell)}$
with entries $Z^{(\ell)}_{ij} \in \{-1, +1\}^{m \times k}$ independently distributed according
to the Rademacher law: $\PP[Z^{(\ell)}_{ij} = -1] = \PP[Z^{(\ell)}_{ij} = +1] = \tfrac12$.
We suggest evaluating the $k$-th directional derivative for
any set of directions $\bdelta^{(1)}, \ldots, \bdelta^{(\ell)}$ by
relying on centered finite differences defined recursively as
\begin{align}\label{eq:recursive_finite_difference}
  \widehat{\pd{^{\ell} g}{\bY^{\ell}}}(\bY)[\bdelta_1, \ldots, \bdelta_{\ell}]
  =
  \frac{1}{2\epsilon_{\ell}}
  \left[
    \pd{^{\ell-1} g}{\bY^{\ell-1}}(\bY + \epsilon_{\ell} \bdelta_{\ell})[\bdelta_1, \ldots, \bdelta_{\ell-1}] \nonumber \right. \\
    \CM{ \left.
    -
    \pd{^{\ell-1} g}{\bY^{\ell-1}}(\bY - \epsilon_{\ell} \bdelta_{\ell})[\bdelta_1, \ldots, \bdelta_{\ell-1}]
    \right]~.}
\end{align}
Note that in practice we will choose $\epsilon_{\ell} = 0.25 \sqrt[\ell]{0.1}$
where the $\ell$-th root is considered to ensure that all elements in the summation
have an error term with a comparable order of magnitude.

%%% \begin{center}
%%%   \CR{Remark: Note that $g$ is not differentiable but is a continuous and weakly differentiable
%%%     mapping. By construction, the finite difference of any order, are thus continuous
%%%     and weakly differentiable as well. Though $g$ is not differentiable, $\EE[g(\bY)]$
%%%     is $C^\infty$. I believe we could prove that
%%%     $$
%%%     \lim_{\epsilon_t \to 0}
%%%     \EE \left[  \widehat{\pd{^{\ell} g}{\bY^t}}(\bY)[\bdelta_1, \ldots, \bdelta_t] \right]
%%%     =
%%%     \widehat{\pd{^t \EE[g]}{\bY^t}}(\bY)[\bdelta_1, \ldots, \bdelta_t]
%%%     $$
%%%     That would be a strong result to justify our choice.
%%%     In the same vein as what I did in \cite{deledalle2014stein}.
%%%   }
%%% \end{center}

Thanks to the use of finite differences, evaluating the $\Ll$ partial derivatives
involved in the definition of $\hat{\bQ}$ requires evaluating
$g$ at $S = \sum_{\ell=0}^{\Ll} 2^{\ell} = 2^{\Ll+1} - 1$ different
locations around $\bY$. Since $\Ll$ will be chosen independently of $m$ and $k$,
this shows that the estimator $\hat{\bQ}$ defined in Theorem \ref{thm:g_discrete_diff_estimator}
can be evaluated in linear time $\Oo(m k)$ (assuming that evaluating $g$ requires $\Oo(m k)$ operations as well).
In practice, we will choose $1 \leq \Ll \leq 6 \ll m k$, hence $3 \leq S \leq 127$,
which undeniably shows the practical advantage of using $\hat{\bQ}$ as a proxy for $\bQ$.

The functions $g$ induced by the spectral estimators considered in this paper
are not of class $C^{\Ll+1}$. They are, in fact, not even $C^1$ since the proximal operator of the nuclear norm
is only differentiable almost everywhere. Nevertheless, we observed in practice
that relying on finite differences, with $\epsilon > 0$ big enough,
has a smoothing effect that subsequently leads to a relevant estimation of $\bQ$,
even though the higher order partial derivatives may not exist
(this is consistent with observations made in \cite{deledalle2014stein}
in the case of the Stein Unbiased Risk Estimator for Gaussian distributed data).

Note that a similar methodology to evaluate $\bQ$ was proposed in \cite{Deledalle} and \cite{Bigot17}
except only a linear approximation of $g$ was considered (\ie $\Ll = 1$).
Satisfying results were obtained since $g$ was closed enough
to its first order approximation. In this paper, we consider multinomial data
for which the quantity $\bP$ to be estimated lies on the manifold
of row-stochastic positive matrices $\Ss^{m \times k}$. This manifold has clearly a non-linear structure
poorly approximated by the set of its tangent subspaces (all the more for small $k$).
As a consequence, $g$ cannot be well captured by its linear expansion,
and, as supported by our numerical experiments,
considering higher order approximations increases considerably the
quality of $\widehat{\MUKLA}(\hat{\bX})$ as an estimator of $\MKLA(\hat{\bX},\bX)$
in this context.

\section{Numerical experiments and applications} \label{sec:num}

In this section, we report the results of various numerical experiments on simulated and real data that shed some lights on the performances of our approach. For readability, in the multinomial case (resp.\ Poisson case) the constant term $ \sum_{i=1}^m  \sum_{j=1}^{k} p_{ij}\log {p}_{ij}$  (resp.\  $\sum_{i=1}^{m} \sum_{j=1}^{k}    X_{ij} \log \left( X_{ij}  \right) -X_{ij} $)  has been added to all estimators of the KL risk. Finally, in all the Figures, we have chosen to visualize true and estimated compositional (or intensity)  matrices through their parametrization by the mapping $\Omega$, that is to display $\Omega^{-1}(\bP)$ or $\Omega^{-1}(\hat{\bP})$ instead of $\bP$ or $\hat{\bP}$.

\subsection{Comparison with a related low-rank approach in the multinomial case}
\label{sec:comp_low_rank_approaches}

In the case of multinomial data, the problem of estimating the compositional matrix $\bP$ under a low rank structure assumption has been recently considered in \cite{Cao17} using a nuclear norm regularized maximum likelihood estimator constrained to belong to a bounded simplex space. More precisely, the numerical approach considered in \cite{Cao17} amounts to compute an estimator defined as
\begin{equation}
\hat{\bP}_{\mu,\alpha} = \argmin_{\bP \in \Ss^{m \times k}(\alpha) }   -\frac{1}{n}\log p(\bY; \bP) + \mu \|\bP\|_{*}  \label{eq:estCao17}
\end{equation}
where \CM{$\log p(\bY; \bP) =   \sum_{i=1}^m \sum_{j=1}^{k} Y_{ij} \log(P_{ij})$ is the  log-likelihood (up to terms not depending on the $p_{ij}$'s) for multinomial data with parameters $\bP$},  $\mu > 0$ is the usual regularization parameter, $n = \sum_{i=1}^{m} n_{i}$, and
$$
\Ss^{m \times k}(\alpha) = \enscond{\bP \in \RR^{m \times k}}{\alpha \leq p_{ij} \leq 1, \bP \mathds{1} = \mathds{1}}
$$
is the subspace of row-stochastic matrices whose elements have entries lower bounded  by $\alpha$, where $\alpha > 0$ is tuning parameter.
Hence, the methodology followed in \cite{Cao17} differs from our approach by the introduction of a tuning parameter $\alpha$ to guarantee the construction of an estimator with positive entries, and the regularization of the compositional matrix $\bP$ itself though the penalty $ \|\bP\|_{*}$ instead of penalizing its re-parametrization  $\Omega^{-1}(\bP)$ as proposed in this paper.

The estimator  $\hat{\bX}_{\lambda}$ is computed using the FISTA algorithm described in Section \ref{sec:algo}. The computation of $\hat{\bP}_{\mu,\alpha} $ is based on an optimization algorithm that uses a generalized accelerated proximal gradient method described in \cite[Section 3.2]{Cao17}. An implementation in R is available from \url{https://github.com/yuanpeicao/composition-estimate}.  From these codes (using the default values for the algorithm to compute $\hat{\bP}_{\mu,\alpha}$), it appears that a recommended data-driven choice for $\alpha$ is
\begin{equation}
\alpha =  \alpha(s) := \min(1 , s \min \enscond{Y_{ij}}{ Y_{ij} > 0} ) \mbox{ with } s = 10^{-2}. \label{eq:alpha}
\end{equation}
Alternatively, cross-validation can be used to select $\alpha$ as detailed in \cite[Section 3.4]{Cao17}.

Using Monte Carlo simulations, we now report the results of numerical experiments on the comparison of the expected KL  risks
$\MKLA(\hat{\bX}_{\lambda},\bX)$ and $\MKLA(\hat{\bX}_{\mu,s},\bX)$, where
$$
\hat{\bX}_{\lambda} = \diag(n_1,n_2,\ldots,n_m) \hat{\bP}_{\lambda} \mbox{ and } \hat{\bX}_{\mu,s} = \diag(n_1,n_2,\ldots,n_m)\hat{\bP}_{\mu,s},
$$
where $ \hat{\bP}_{\lambda} = \Omega(\hat{\bZ}_{\lambda} )$ and $\hat{\bP}_{\mu,s} = \hat{\bP}_{\mu,\alpha(s)} $.
To this end, we consider the following choices for the low-rank compositional matrix $\bP$: \\

\noindent {\bf (Case 1)} the entries of $\bP$ are
\begin{equation}
P_{ij} = \frac{1}{10k} + \frac{9}{10} \frac{A_{ij}}{\sum_j A_{ij}}
\mbox{ with } A_{ij} =  \exp(10 \cos(\frac{i}{k} 6 \pi) \sin(\frac{j}{k} 6 \pi)) \label{eq:modcase1}
\end{equation}
for $1 \leq i \leq m$ and $1 \leq j \leq k$, \\

\noindent {\bf (Case 2)}  $\bP$ is chosen according to the simulation study described in \cite[Section 5]{Cao17}. This corresponds to
\begin{equation}
  P_{ij} = W_{ij} / \sum_{\ell=1}^{k} W_{i \ell}, \label{eq:modcase2}
\end{equation}
where $W = U V^T$ with $U \in \R_+^{m \times r}$ whose entries are the absolute values of iid standard Gaussian variables, and $V \in \R^{m \times k}$ is a random matrix with correlated entries (having small variances) chosen to mimic the typical behavior of compositional data arising from metagenomics (we fix the rank $r=20$). With small probability, this procedure may produce non-positive values and the generating process is repeated until all the entries of $\bP$ are positive (for further details we refer to \cite[Section 5]{Cao17}).

The resulting compositional matrices $\bP$ are displayed in Figure \ref{fig:case1} and Figure \ref{fig:case2} through the visualization of their parametrization by the mapping $\Omega$. Then, as proposed in
\cite{Cao17}, we generate count data from a Poisson-multinomial model
which consists in first generating independent realizations
$n_1,\ldots,n_m$ from a Poisson distribution with intensity $n_0$, and
then in sampling $\bY$ from the multinomial model \eqref{eq:modMulti}
conditionally on these realizations. Monte-Carlo simulations can then
be used to estimate the expected risk $\KLA$ of each estimator
over different grids of values for the regularization parameters
$\lambda$ and $\mu$ (note that the $n_i$'s are random variables in this simulation setup). The results are displayed in Figure
\ref{fig:case1} and Figure
\ref{fig:case2}  by choosing the scaling parameter $s \in \{10^{-4},10^{-3}\}$ in the calibration \eqref{eq:alpha} for
$\alpha$. It appears that the lowest risks are obtained for the estimators $\hat{\bX}_{\lambda}$ and $ \hat{\bX}_{\mu,s}$  with $s=10^{-3}$.
The main advantage of our approach is its dependence on only one regularization parameter $\lambda$, whereas the methodology in  \cite{Cao17} requires the tuning of the two regularization parameters $\mu$ and $\alpha$.  In particular, the choice of $\alpha$ has a great importance.

%We observed that finding data-driven values for $\mu$ and $\alpha$ following the cross-validation strategy proposed in  \cite{Cao17} results in more involved computations than finding a data-driven regularization parameter $\lambda^{\ast}$ minimizing $\lambda \mapsto \widehat{\MUKLA}(\hat{\bX}_{\lambda})$.  Therefore, in the rest of this section, we shall  concentrate on the evaluation of numerical methods to compute such a value $\lambda^{\ast}$ and their applications to real data analysis.

% Case 1
\begin{figure}[htbp]
  \centering

  \subfigure[$n_0=1000$, $\lambda^\ast = 1.45$, $\mu^\ast = 0.025$ for $s = 10^{-3}$]{\includegraphics[width=.45\linewidth]{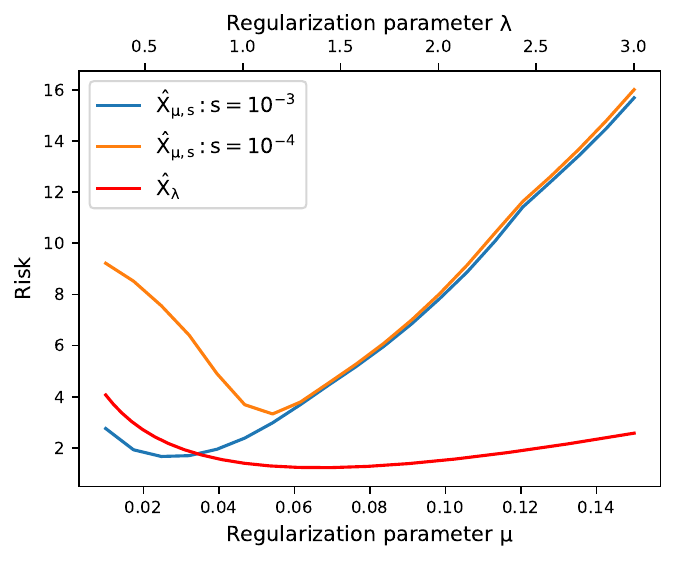}}
  \subfigure[$n_0=100$, $\lambda^\ast = 1.64$, $\mu^\ast = 0.076$ for $s = 10^{-3}$]{\includegraphics[width=.45\linewidth]{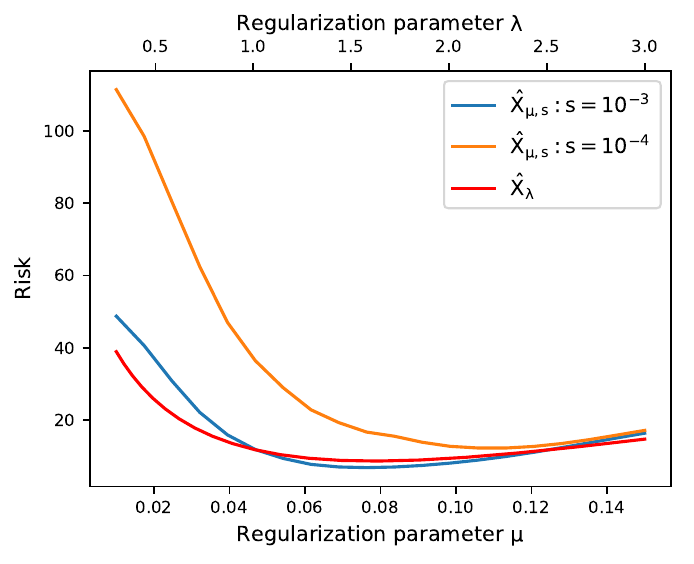}}

  \subfigure[Underlying matrix $\Omega^{-1}(\bP)$]{\includegraphics[width=.32\linewidth]{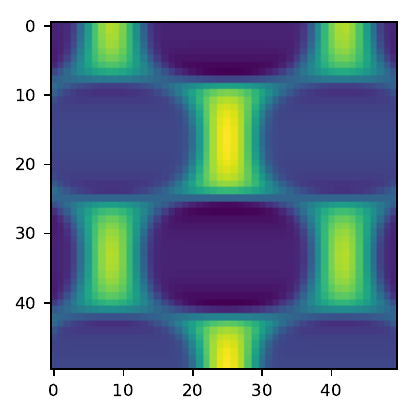}}
  \subfigure[ $\Omega^{-1}(\hat{\bP}_{\lambda^\ast})$]{\includegraphics[width=.32\linewidth]{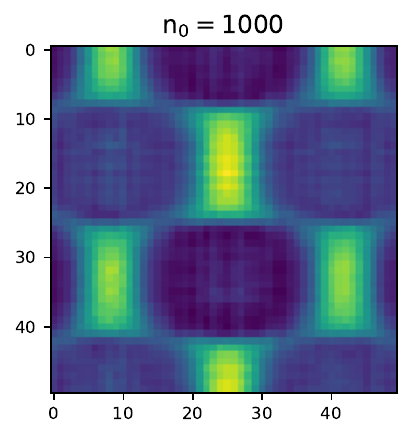}}
  \subfigure[ $\Omega^{-1}(\hat{\bP}_{\mu^\ast,s})$ with $s = 10^{-3}$]{\includegraphics[width=.32\linewidth]{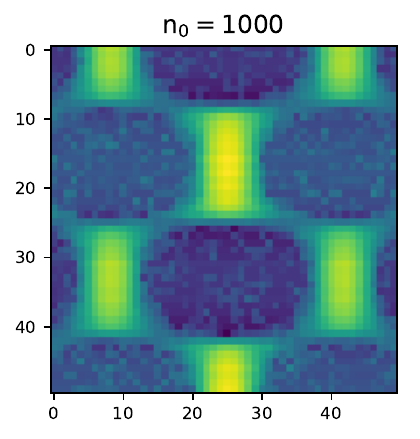}}

  \hspace{.32\linewidth}
  \subfigure[ $\Omega^{-1}(\hat{\bP}_{\lambda^\ast})$]{\includegraphics[width=.32\linewidth]{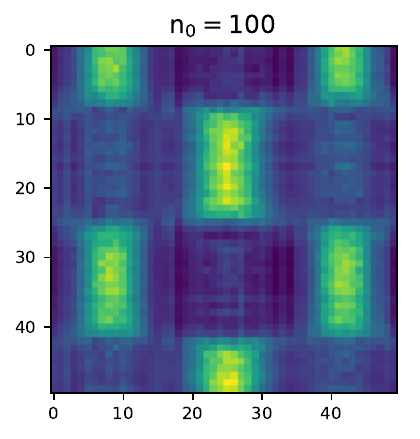}}
  \subfigure[ $\Omega^{-1}(\hat{\bP}_{\mu^\ast,s})$ with $s = 10^{-3}$]{\includegraphics[width=.32\linewidth]{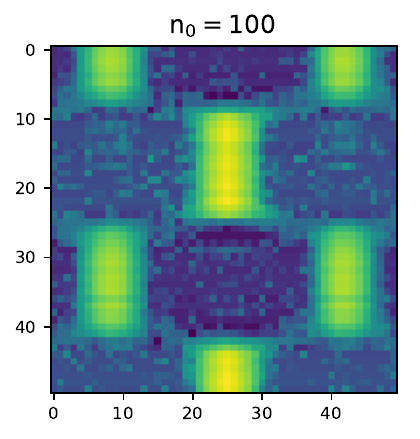}}

  \caption{{\bf (Case 1)} (a-b) Expected KL risk (by Monte Carlo simulations) $\KLA(\hat{\bX}_{\lambda},\bX)$ and $\KLA(\hat{\bX}_{\mu,s},\bX)$ as functions of $\lambda$ and $\mu$ (with $m=200$ and $k=100$) for different mean value $n_0$ of counts per line. (c-g) The upper-left $50 \times 50$ sub-matrices of the parametrization by $\Omega$ of the
    composition matrix $\bP$  , and the  denoising results (from one realization $\bY$) obtained using the best parameters $\lambda^\ast$ and $\mu^\ast$ minimizing the expected KL risk for $n_0 = 1000$ and $n_0 = 100$.  }
  \label{fig:case1}
\end{figure}

% Case 2
\begin{figure}[htbp]
  \centering

  \subfigure[$n_0=1000$, $\lambda^\ast = 1.42$, $\mu^\ast = 0.025$ for $s = 10^{-3}$]{\includegraphics[width=.45\linewidth]{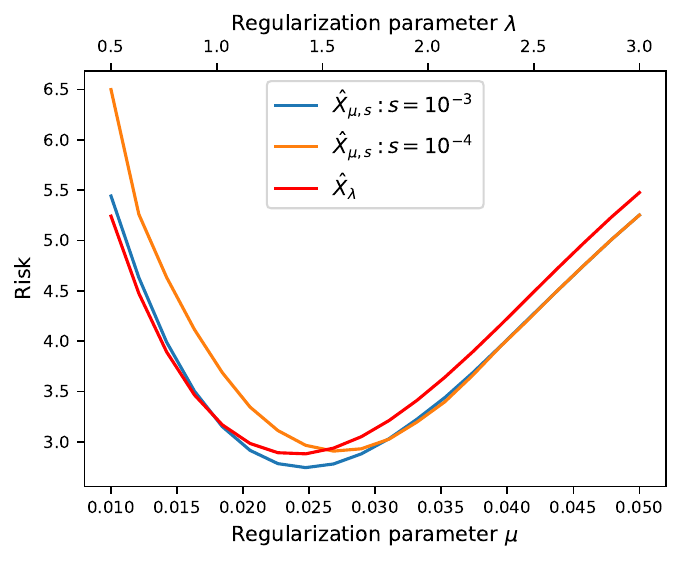}}
  \subfigure[$n_0=100$, $\lambda^\ast = 2.13$, $\mu^\ast = 0.115$ for $s = 10^{-3}$]{\includegraphics[width=.45\linewidth]{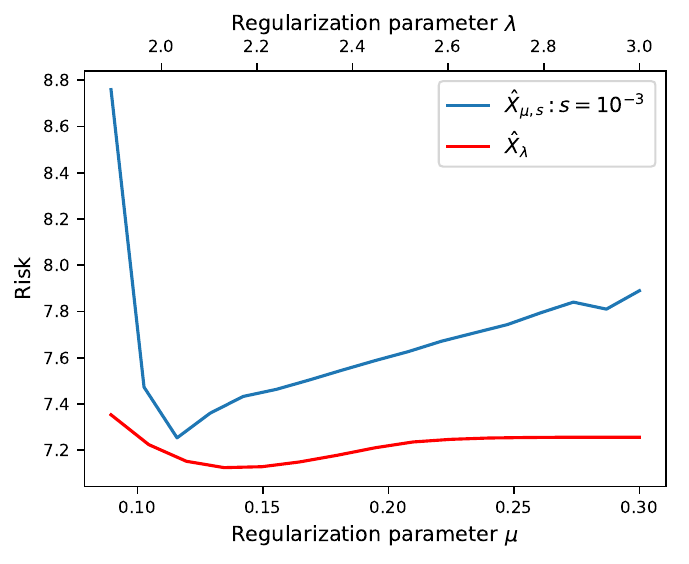}}

  \subfigure[Underlying matrix $\Omega^{-1}(\bP)$]{\includegraphics[width=.32\linewidth]{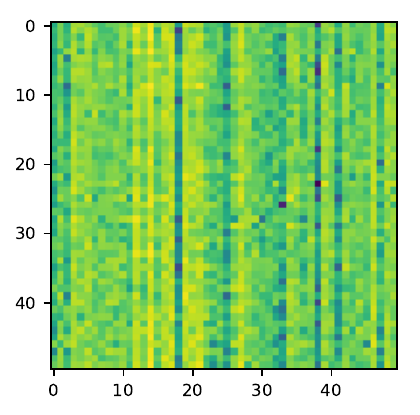}}
  \subfigure[ $\Omega^{-1}(\hat{\bP}_{\lambda^\ast})$]{\includegraphics[width=.32\linewidth]{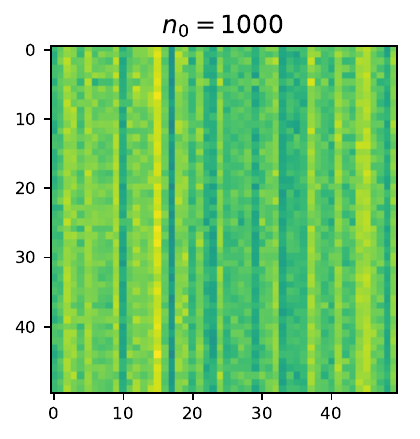}}
  \subfigure[ $\Omega^{-1}(\hat{\bP}_{\mu^\ast,s})$ with $s = 10^{-3}$]{\includegraphics[width=.32\linewidth]{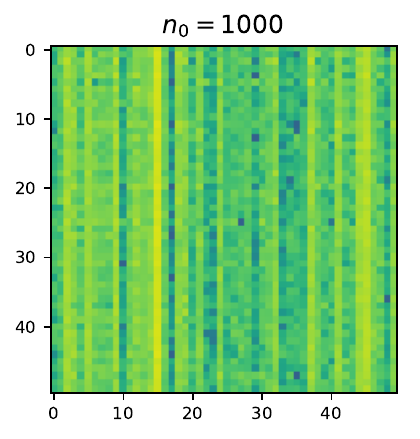}}

  \hspace{.32\linewidth}
  \subfigure[ $\Omega^{-1}(\hat{\bP}_{\lambda^\ast})$]{\includegraphics[width=.32\linewidth]{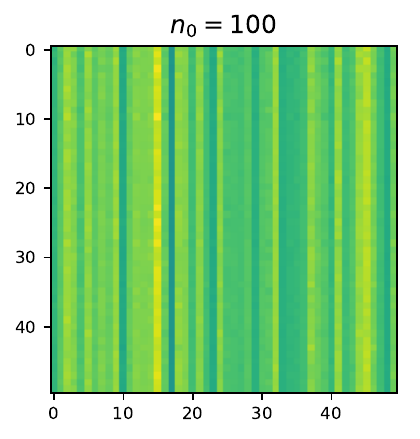}}
  \subfigure[ $\Omega^{-1}(\hat{\bP}_{\mu^\ast,s})$ with $s = 10^{-3}$]{\includegraphics[width=.32\linewidth]{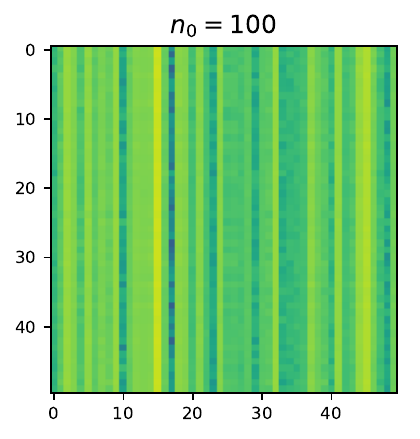}}

  \caption{{\bf (Case 2)} (a-b) Expected KL risk (by Monte Carlo simulations) $\KLA(\hat{\bX}_{\lambda},\bX)$ and $\KLA(\hat{\bX}_{\mu,s},\bX)$ as functions of $\lambda$ and $\mu$ (with $m=200$ and $k=100$) for different mean value $n_0$ of counts per line. (c-g) The upper-left $50 \times 50$ sub-matrices of the parametrization by $\Omega$ of the composition matrix $\bP$ sampled from model \eqref{eq:modcase2}, and the  denoising results (from one realization $\bY$) obtained using the best parameters $\lambda^\ast$ and $\mu^\ast$ minimizing the expected KL risk for $n_0 = 1000$ and $n_0 = 100$.  }
  \label{fig:case2}
\end{figure}

\subsection{Quality of estimation of the expected KL risk}
\label{sec:quality_est_MUKLA}

In this section we evaluate the quality of our approximation of $\widehat{\MUKLA}(\hat{\bX})$
based on Taylor expansion of different orders, Rademacher Monte-Carlo simulations,
and finite difference approximations.

\subsubsection{A Poisson example}

We consider the low-rank sinusoidal matrix $\bX$ defined as \CM{$$X_{ij} = 5 \cos(\frac{i}{k} 6 \pi) \sin(\frac{j}{k} 6 \pi).$$}
In this simulations $m=200$, $k=100$.
We sample a single realization $\bY$ from the Poisson model \eqref{eq:modPoisson},
and we evaluated the risk $\KLA(\hat{\bX}_{\lambda},\bX)$ and its estimations $\widehat{\MUKLA}(\hat{\bX}_{\lambda})$ for the
proposed low-rank variational estimators
over a grid of values for the regularization parameters $\lambda$.
We considered the approximations
based on Taylor expansion that we evaluated up to order $4$.
We considered a single realization of the Rademacher matrices $\bZ^{(\ell)}$
and finite differences as defined in eq.~\eqref{eq:recursive_finite_difference}.

The results are displayed on Figures
\ref{fig:results_poisson_synthetic_sinusoidal}.
We observe that the curve $\lambda \mapsto \KLA(\hat{\bX}_{\lambda},\bX)$ and their estimates $\lambda \mapsto \widehat{\MUKLA}(\hat{\bX}_{\lambda})$ are
closed to each others, as well as their minimizers. We also remark that the
Taylor approximations requires about 2 orders to converge more closely to
the underlying $\KLA$ risk.

\begin{figure}[htbp]
  \centering
  \subfigure[]{\includegraphics[width=.6\linewidth]{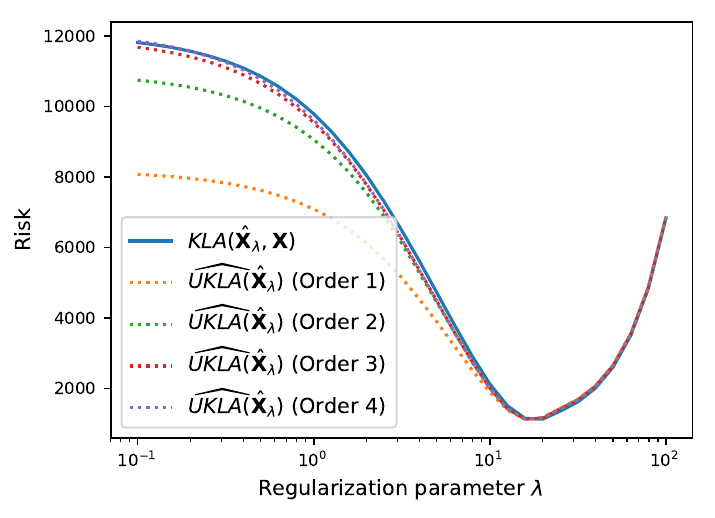}}%
  \subfigure[$\Omega^{-1}(\bP)$ and $\Omega^{-1}(\hat{\bX}_{\lambda})$]{\includegraphics[width=.4\linewidth]{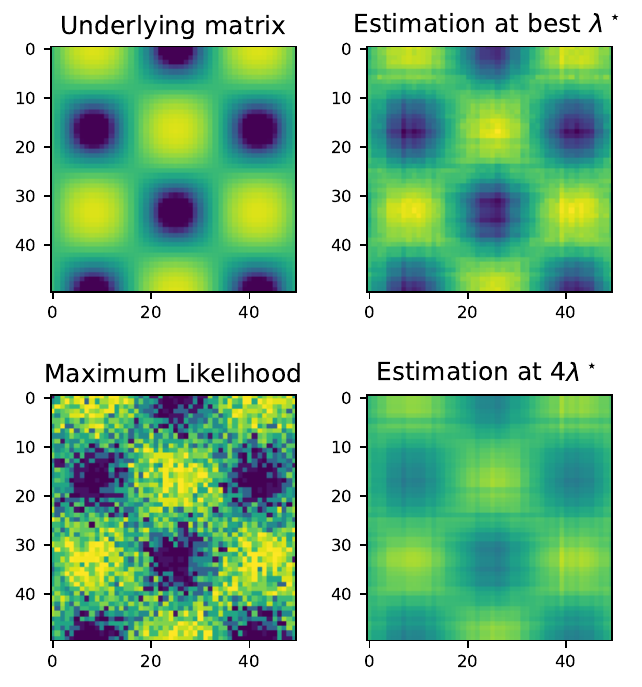}}
  \caption{
    Results of the low-rank variational estimator on Poisson simulated data
    from the $200 \times 100$ from the sinusoidal data.
    (a) Evaluation of $\KLA$ and $\MUKLA$ with respect to the regularization parameter
    $\lambda$ for different  orders in the Taylor expansion.
    (b) The upper-left $50 \times 50$ sub-matrices $\Omega^{-1}(\bP)$,
    the results obtained using the optimal parameter $\lambda^{\ast}$,
     $\lambda = 0$ (Maximum Likelihood), and $4 \lambda^{\ast}$, respectively from top to bottom,
    left to right.
  }
  \label{fig:results_poisson_synthetic_sinusoidal}
\end{figure}

\subsubsection{Multinomial examples}
We consider the two settings ({\bf Case 1} and {\bf Case 2})  of simulated multinomial data
described in Section \ref{sec:comp_low_rank_approaches}.  For each case,
we consider both the simple and the low rank variational estimators.
In these simulations $m=200$, $k=100$ and we again consider $n_1,\ldots,n_m$
sampled from a Poisson distribution with intensity $n_0=400$.
For each scenario, we sample a single realization $\bY$ from the multinomial model \eqref{eq:modMulti}
conditionally on these realizations $n_i$, and we evaluated
the risk $\KLA(\hat{\bX},\bX)$ and its estimation $\widehat{\MUKLA}(\hat{\bX})$ for the two proposed estimators $\hat{\bX} = \hat{\bX}_{w}$ and $\hat{\bX} = \hat{\bX}_{\lambda}$
over different grids of values for the regularization parameters
$w$ and $\lambda$.
For the simple estimator, we considered
the exact value $\widehat{\MUKLA}(\hat{\bX}_{w})$ based on eq.~\eqref{eq:simple_minus_1}, as well as the Rademacher Monte-Carlo approximations
based on Taylor expansion, as given in eq.~\eqref{eq:g_discrete_diff_estimator}, up to order $6$.
For the low rank variational estimator, we considered only the approximations of $\widehat{\MUKLA}(\hat{\bX}_{\lambda})$
based on Taylor expansion that we evaluated up to order $3$.
In both cases, we considered a single realization of the Rademacher matrices $\bZ^{(\ell)}$
and finite differences as defined in eq.~\eqref{eq:recursive_finite_difference}.

The results are displayed on Figures
\ref{fig:results_synthetic_sinusoidal_simple}, \ref{fig:results_synthetic_sinusoidal},
\ref{fig:results_synthetic_composition_simple} and \ref{fig:results_synthetic_composition}.
For the simple estimator, we observe that the curves $\KLA(\hat{\bX}_{w},\bX)$ and the exact estimate $\widehat{\MUKLA}(\hat{\bX}_{w})$ are
closed to each others, as well as their minimizers. Due to the highly
non linear behavior of the simple estimator, we also remark that the
Taylor approximations requires 6 orders to converge more closely to the exact $\widehat{\MUKLA}(\hat{\bX}_{w})$
estimator, hence, fitting the underlying $\KLA(\hat{\bX}_{w},\bX)$ risk.
In contrast, the low rank variational estimator seems to have a smoother
structure and only an order of 2 seems sufficient to
get a satisfying estimate of the true $\KLA(\hat{\bX}_{\lambda},\bX)$ risk by $\widehat{\MUKLA}(\hat{\bX}_{\lambda})$. {\color{black} In Figures
\ref{fig:results_synthetic_sinusoidal_simple}, \ref{fig:results_synthetic_sinusoidal},
\ref{fig:results_synthetic_composition_simple} and \ref{fig:results_synthetic_composition}, we have also reported the values of the CV criterion \eqref{eq:CV}. It can be seen that this criterion does not lead to a consistent estimation of the KL risk. Moreover, the values of the regularization parameters (either $w$ or $\lambda$) selected by minimizing the CV criterion are generally much different from those obtained by minimizing the KL risk or our criteria based on generalized SURE.}

\begin{figure}[htbp]
  \centering
  \subfigure[]{\includegraphics[width=.6\linewidth]{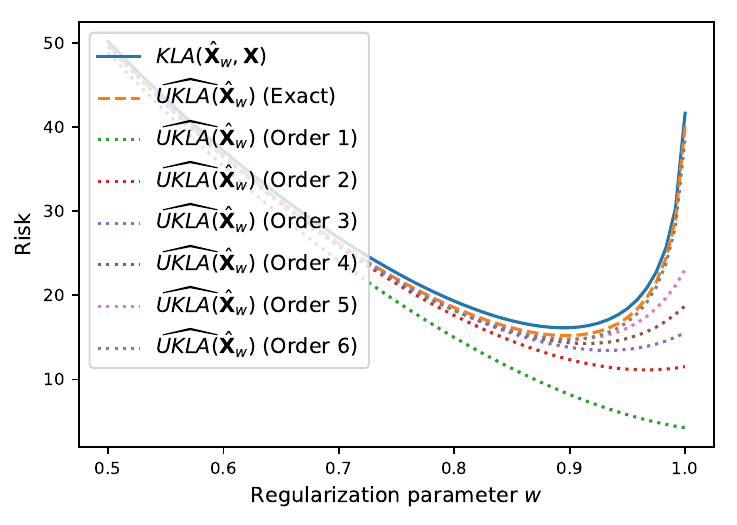}}%
  \subfigure[$\Omega^{-1}(\bP)$ and $\Omega^{-1}(\hat{\bP}_{w})$]{\includegraphics[width=.4\linewidth]{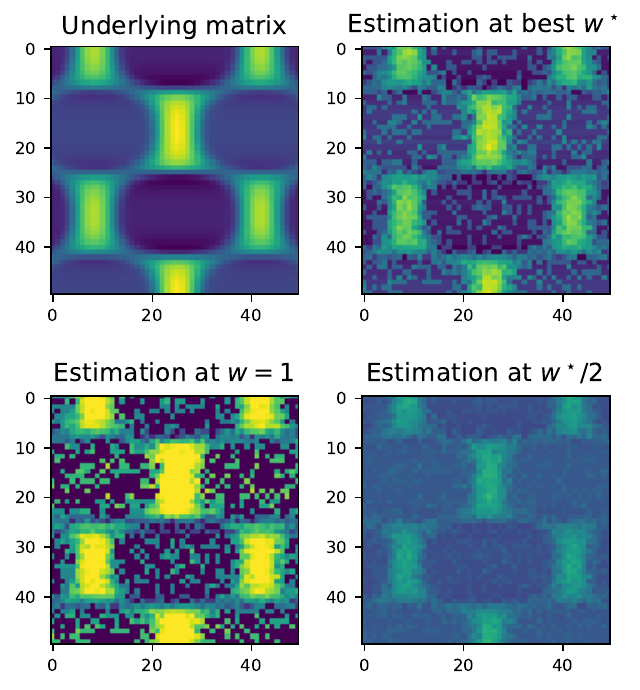}}
    \subfigure[]{\includegraphics[width=.5\linewidth]{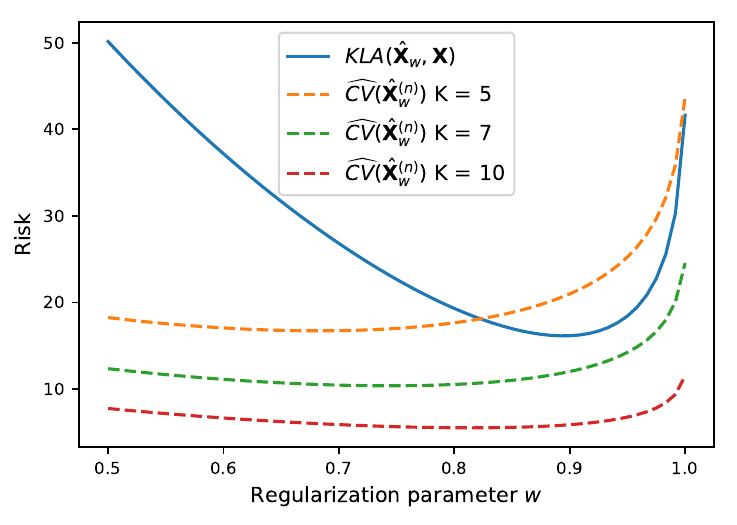}}
  \caption{
    {\bf (Case 1)}. Results of the simple estimator on multinomial simulated data
    with $n_0 = 400$ sampled from the $200 \times 100$ compositional matrix data $\bP$ \eqref{eq:modcase1}.
    (a) Evaluation of $\KLA$ and $\MUKLA$ with respect to the regularization parameter
    $w$ for different  orders in the Taylor expansion.
    (b) The upper-left $50 \times 50$ sub-matrices $\Omega^{-1}(\bP)$,
    the results obtained using the optimal parameter $w^{\ast}$,
    $w=1$ (ML estimation with zero-replacement), and $w^{\ast}/2$, respectively from top to bottom,
    left to right. {\color{black} (c) KL risk and its estimation using $K$-fold CV.}
  }
  \label{fig:results_synthetic_sinusoidal_simple}
\end{figure}

\begin{figure}[htbp]
  \centering
  \subfigure[]{\includegraphics[width=.6\linewidth]{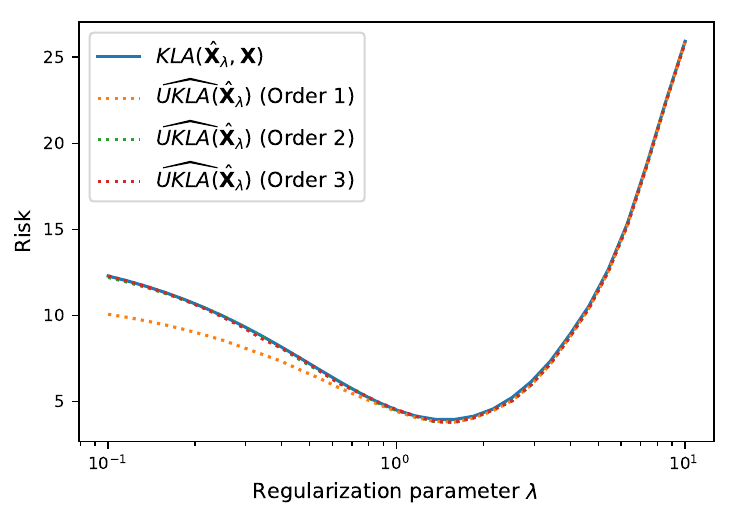}}%
  \subfigure[$\Omega^{-1}(\bP)$ and $\Omega^{-1}(\hat{\bP}_{\lambda})$]{\includegraphics[width=.4\linewidth]{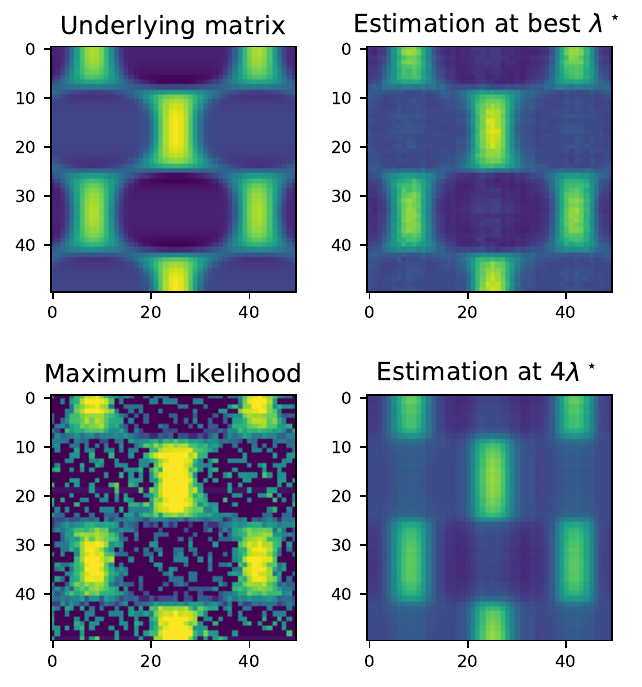}}
   \subfigure[]{\includegraphics[width=.5\linewidth]{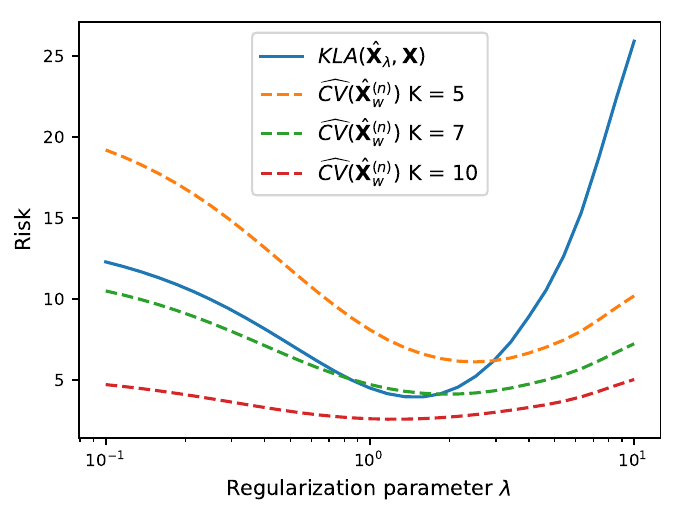}}
  \caption{
    {\bf (Case 1)}.  Results of the low-rank variational estimator on multinomial simulated data
    with $n_0 = 400$ sampled from the $200 \times 100$ compositional matrix data $\bP$ \eqref{eq:modcase1}.
    (a) Evaluation of $\KLA$ and $\MUKLA$ with respect to the regularization parameter
    $\lambda$ for different  orders in the Taylor expansion.
    (b) The upper-left $50 \times 50$ sub-matrices of the
    underlying matrix $\Omega^{-1}(\bP)$,
    the results obtained using the optimal parameter $\lambda^{\ast}$,
    $\lambda = 0$ (Maximum Likelihood), and $4 \lambda^{\ast}$, respectively from top to bottom,
    left to right. {\color{black} (c) KL risk and its estimation using $K$-fold CV.}
  }
  \label{fig:results_synthetic_sinusoidal}
\end{figure}

\begin{figure}[htbp]
  \centering
  \subfigure[]{\includegraphics[width=.6\linewidth]{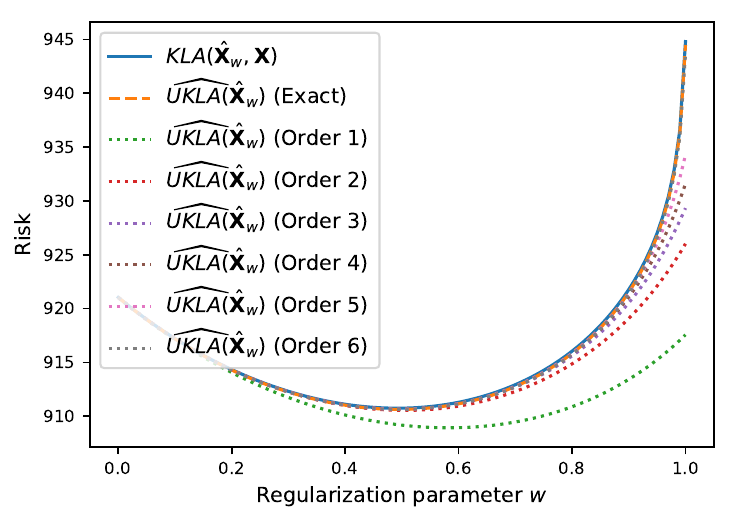}}%
  \subfigure[$\Omega^{-1}(\bP)$ and $\Omega^{-1}(\hat{\bP}_{w})$]{\includegraphics[width=.4\linewidth]{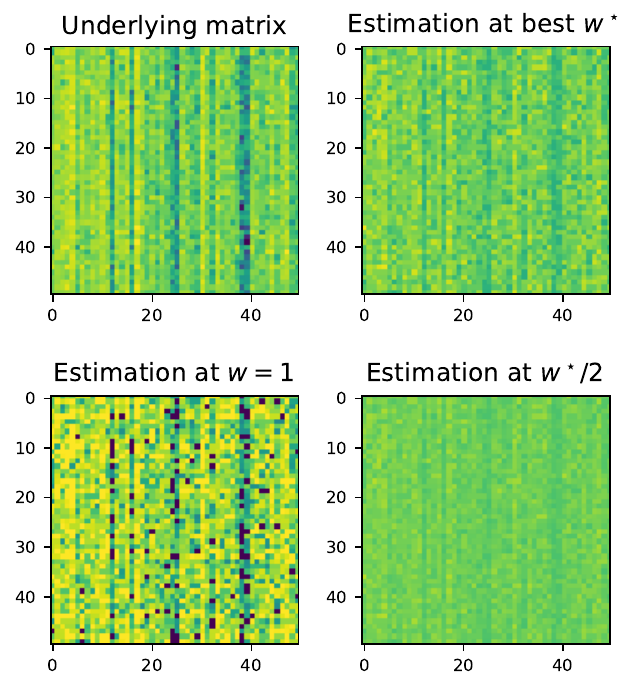}}
   \subfigure[]{\includegraphics[width=.5\linewidth]{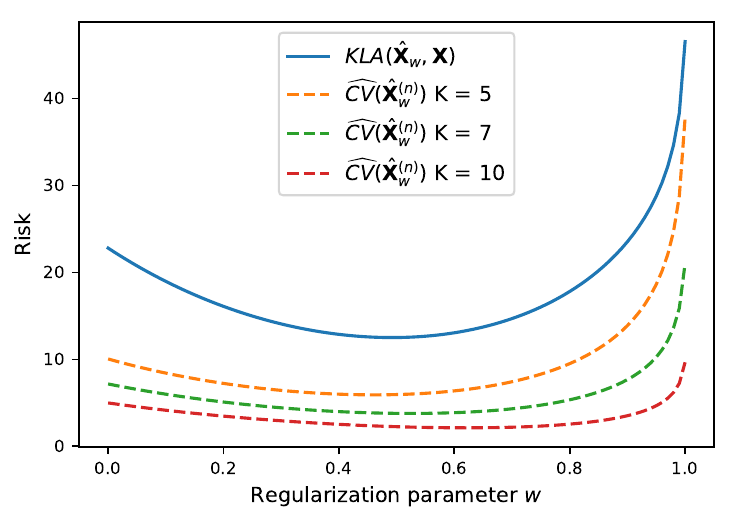}}
  \caption{
    {\bf (Case 2)}. Results of the simple estimator on multinomial simulated data
    with $n_0 = 400$ sampled from the $200 \times 100$ compositional matrix data $\bP$ \eqref{eq:modcase2}.
    (a) Evaluation of $\KLA$ and $\MUKLA$ with respect to the regularization parameter
    $w$ for different  orders in the Taylor expansion.
    (b) The upper-left $50 \times 50$ sub-matrices of the
    underlying matrix $\Omega^{-1}(\bP)$,
    the results obtained using the optimal parameter $w^{\ast}$,
    $w = 1$ (ML estimation with zero replacement), and $w^{\ast}/2$, respectively from top to bottom,
    left to right.  {\color{black} (c) KL risk and its estimation using $K$-fold CV.}
  }
  \label{fig:results_synthetic_composition_simple}
\end{figure}

\begin{figure}[htbp]
  \centering
  \subfigure[]{\includegraphics[width=.6\linewidth]{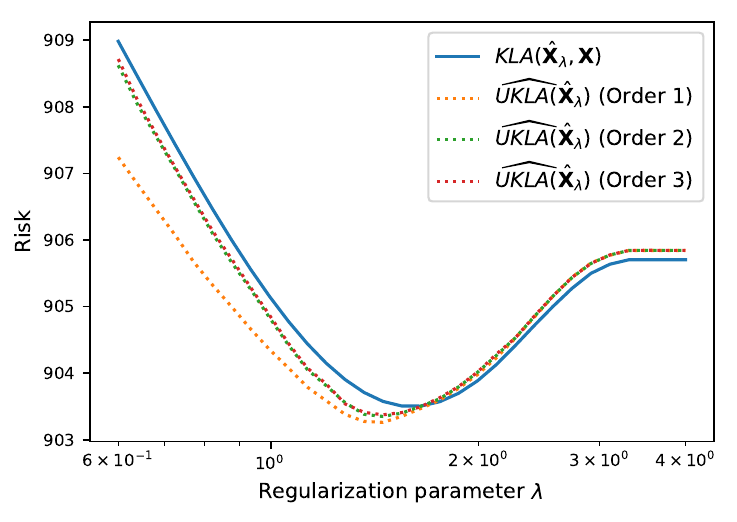}}%
  \subfigure[$\Omega^{-1}(\bP)$ and $\Omega^{-1}(\hat{\bP}_{\lambda})$]{\includegraphics[width=.4\linewidth]{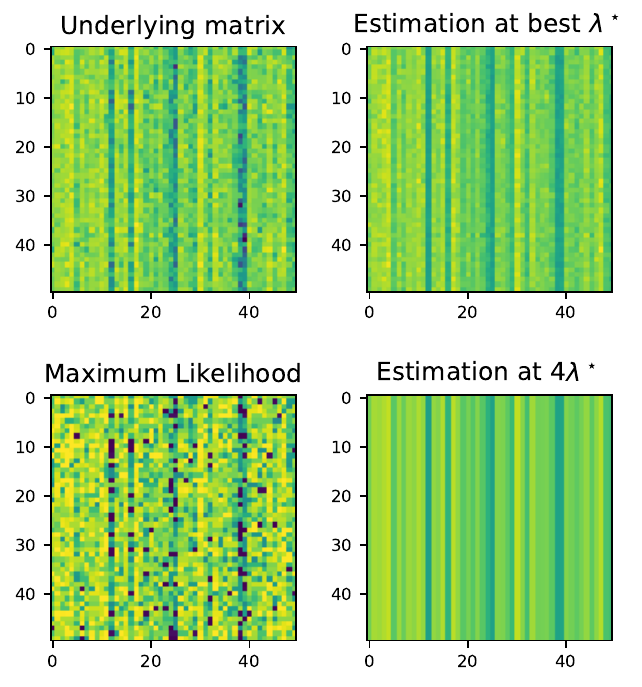}}
  \subfigure[]{\includegraphics[width=.5\linewidth]{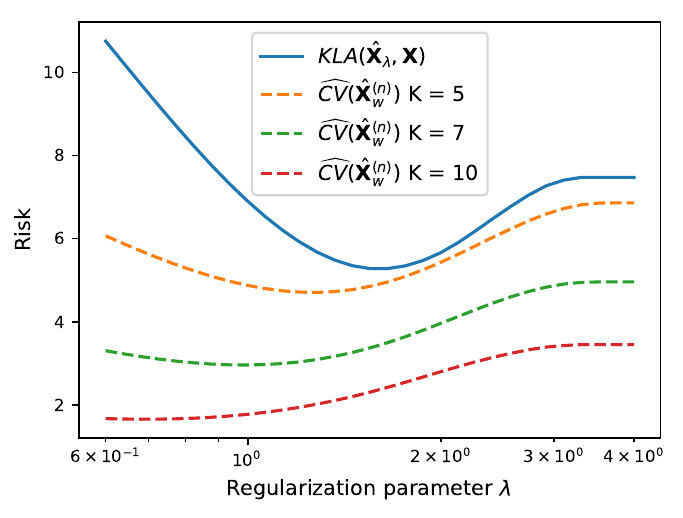}}
  \caption{
     {\bf (Case 2)} Results of the low-rank variational estimator on multinomial simulated data
    with $n_0 = 400$ sampled from the $200 \times 100$ compositional matrix data $\bP$ \eqref{eq:modcase2}.
    (a) Evaluation of $\KLA$ and $\MUKLA$ with respect to the regularization parameter
    $\lambda$ for different  orders in the Taylor expansion.
    (b) The upper-left $50 \times 50$ sub-matrices of the
    underlying matrix $\Omega^{-1}(\bP)$,
    the results obtained using the optimal parameter $\lambda^{\ast}$,
    $\lambda=0$ (Maximum Likelihood), and $4 \lambda^{\ast}$, respectively from top to bottom,
    left to right.  {\color{black} (c) KL risk and its estimation using $K$-fold CV.}
  }
  \label{fig:results_synthetic_composition}
\end{figure}

\CM{
\subsubsection{Influence of the number $T$ of iterations  of the FISTA algorithm}

To conclude these numerical experiments with simulated data, we discuss  the influence of the number $T$ of iterations  of the FISTA algorithm used to compute  the low rank variational estimator $\hat{\bX}_{\lambda}$. To this end, we consider the first setting ({\bf Case 1})  of simulated multinomial data described in Section \ref{sec:comp_low_rank_approaches} with $(m,k) \in \{ (100,50) ; (200,100) ; (400,200) \}$ and $n_0 = 400$. In this setting, we  analyze how the choice of $T$ affects the value of the risk  $\KLA(\hat{\bX}_{\lambda},\bX)$ and the bias of the estimator $\widehat{\MUKLA}(\hat{\bX}_{\lambda})$  with respect to the regularization parameter $\lambda$. The approximation of  $\widehat{\MUKLA}(\hat{\bX}_{\lambda})$ is based on Taylor expansion that we  again evaluated up to order $3$ with a single realization of the Rademacher matrices $\bZ^{(\ell)}$ and finite differences as defined in eq.~\eqref{eq:recursive_finite_difference}.
The results are reported in Figures \ref{fig:results_synthetic_sinusoidal_T_m100_k50}, \ref{fig:results_synthetic_sinusoidal_T_m200_k100} and \ref{fig:results_synthetic_sinusoidal_T_m400_k200} for $T \in \{5,10,50,100\}$. For all values of  $(m,k)$, we observe than increasing $T$ allows to decrease the value of the risk  $\KLA(\hat{\bX}_{\lambda},\bX)$, and that the choice of $T$ also affects the value of $\lambda$ minimizing  $\KLA(\hat{\bX}_{\lambda},\bX)$ and  $\widehat{\MUKLA}(\hat{\bX}_{\lambda})$ (for any value of the order of differentiation). Moreover, we have observed that choosing $T > 100$ does not yield a significantly smallest value of the risk $\KLA(\hat{\bX}_{\lambda},\bX)$ and the estimator $\widehat{\MUKLA}(\hat{\bX}_{\lambda})$. Therefore, for $T > 100$, the value of $\lambda$ minimizing $\widehat{\MUKLA}(\hat{\bX}_{\lambda})$  (for any value of the order of differentiation) remains the same.  Hence, from these numerical results, it appears that choosing $T = 100$ is typically a sufficient large number of  iterations to ensure that  $\widehat{\MUKLA}(\hat{\bX}_{\lambda})$ has a small bias in the sense that it is numerically close to the true value of  $\KLA(\hat{\bX}_{\lambda},\bX)$. Finally, we have not found any significant influence of the choice of $(m,k)$ on the performances of the algorithm apart from the fact that  taking larger values of $m$ and $k$  increases its computational cost.
}

\begin{figure}[htbp]
  \centering
  \subfigure[$T=5$]{\includegraphics[width=.48\linewidth]{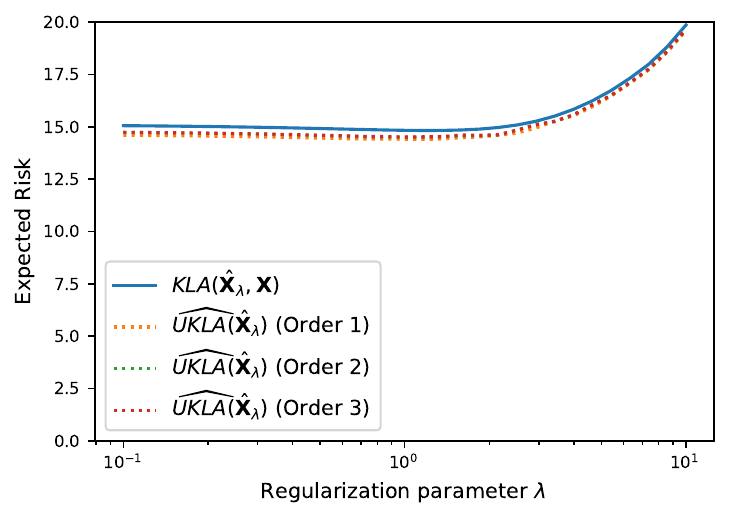}}
  \subfigure[$T=10$]{\includegraphics[width=.48\linewidth]{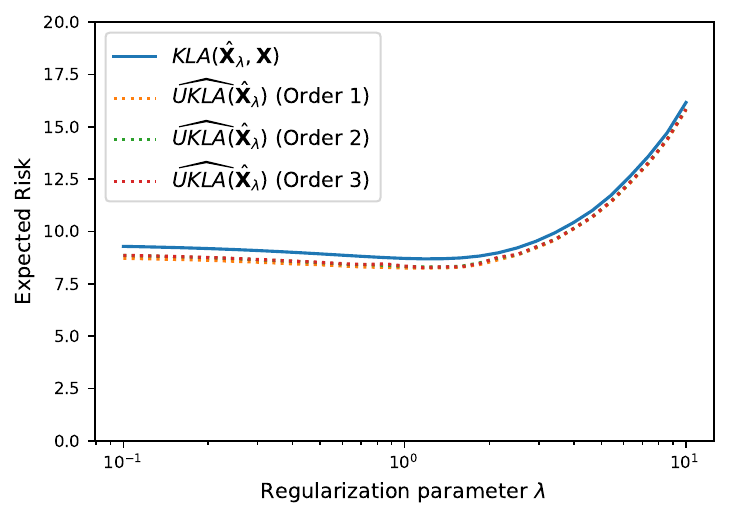}}
  \subfigure[$T=50$]{\includegraphics[width=.48\linewidth]{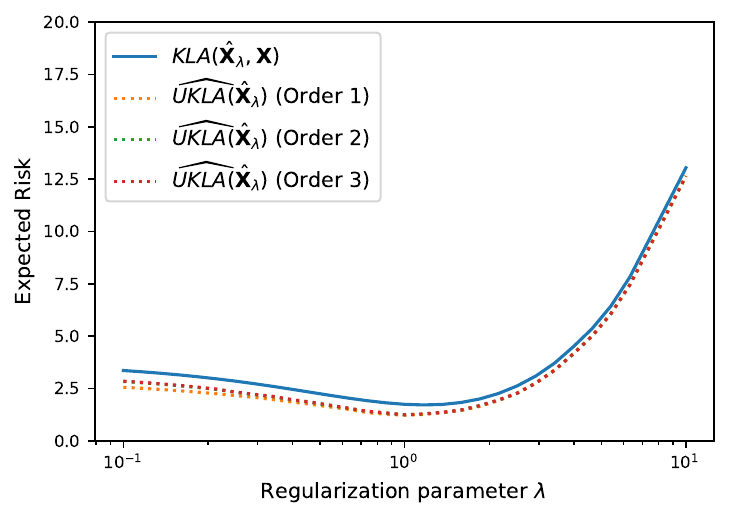}}
  \subfigure[$T=100$]{\includegraphics[width=.48\linewidth]{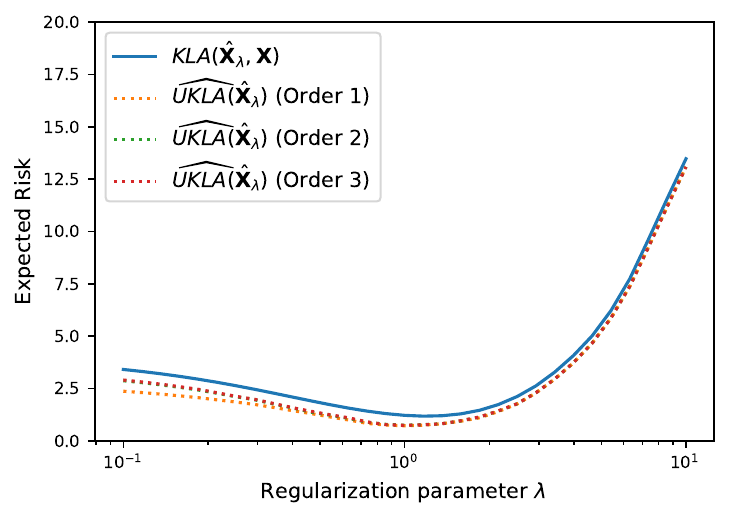}}
  \caption{
    \CM{{\bf (Case 1) - $(m,k) = (100,50)$}.  Results of the low-rank variational estimator on multinomial simulated data with $n_0 = 400$ sampled from the $100 \times 50$ compositional matrix data $\bP$ \eqref{eq:modcase1}.
 Evaluation of $\KLA$ and $\MUKLA$ with respect to the regularization parameter
    $\lambda$ for different orders in the Taylor expansion, and various values of the number $T$ of iterations  of the FISTA algorithm.}}
  \label{fig:results_synthetic_sinusoidal_T_m100_k50}
\end{figure}

\begin{figure}[htbp]
  \centering
  \subfigure[$T=5$]{\includegraphics[width=.48\linewidth]{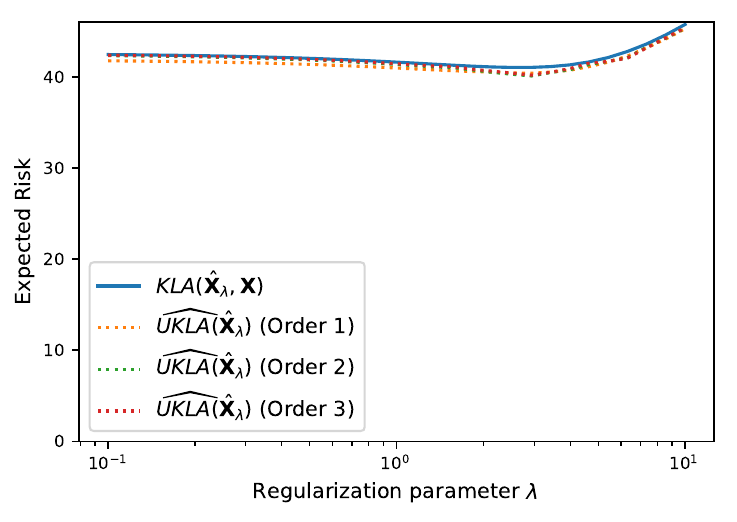}}
  \subfigure[$T=10$]{\includegraphics[width=.48\linewidth]{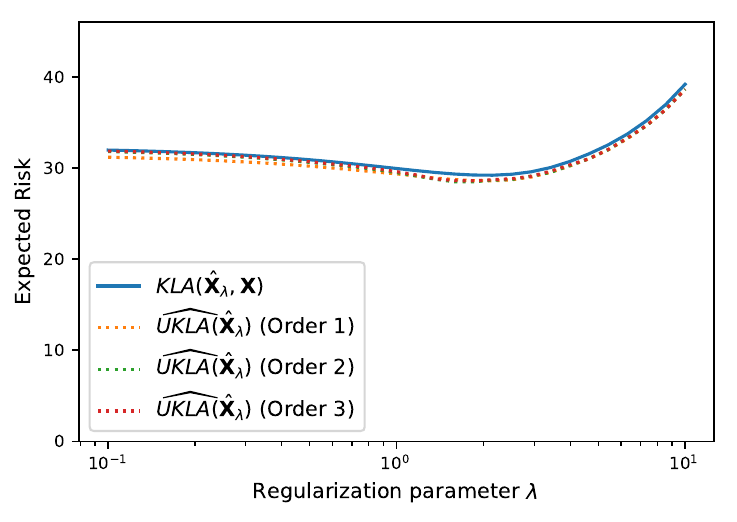}}
  \subfigure[$T=50$]{\includegraphics[width=.48\linewidth]{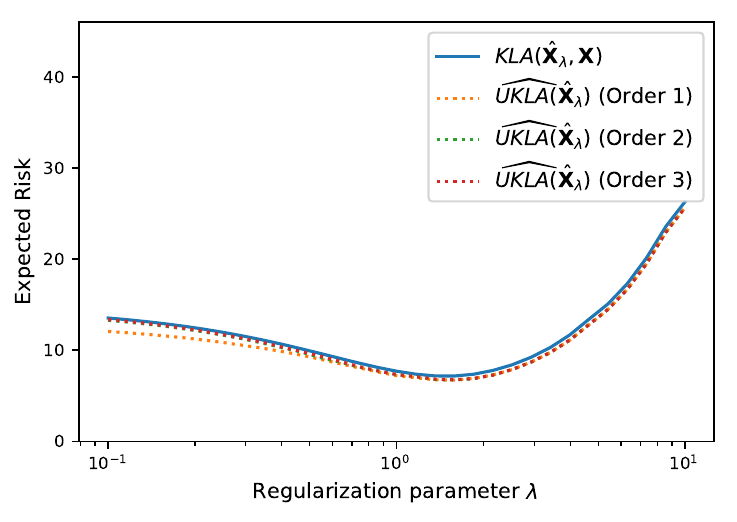}}
  \subfigure[$T=100$]{\includegraphics[width=.48\linewidth]{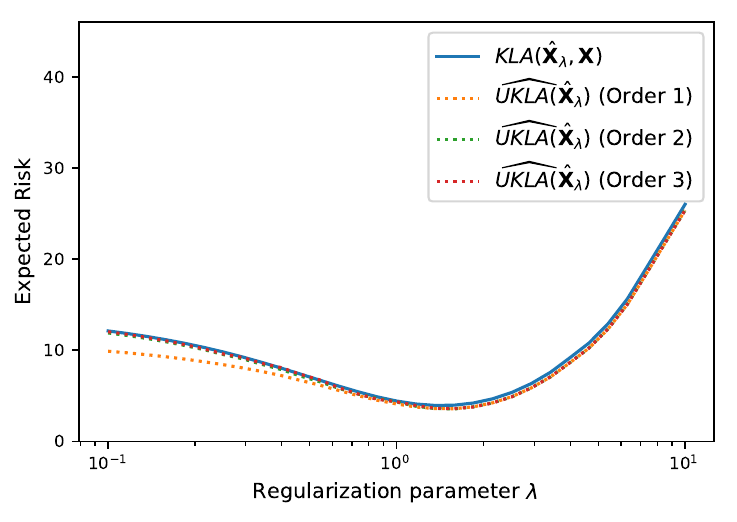}}
  \caption{
    \CM{{\bf (Case 1) - $(m,k) = (200,100)$}.  Results of the low-rank variational estimator on multinomial simulated data with $n_0 = 400$ sampled from the $200 \times 100$ compositional matrix data $\bP$ \eqref{eq:modcase1}.
 Evaluation of $\KLA$ and $\MUKLA$ with respect to the regularization parameter
    $\lambda$ for different orders in the Taylor expansion, and various values of the number $T$ of iterations  of the FISTA algorithm.}}
  \label{fig:results_synthetic_sinusoidal_T_m200_k100}
\end{figure}

\begin{figure}[htbp]
  \centering
  \subfigure[$T=5$]{\includegraphics[width=.48\linewidth]{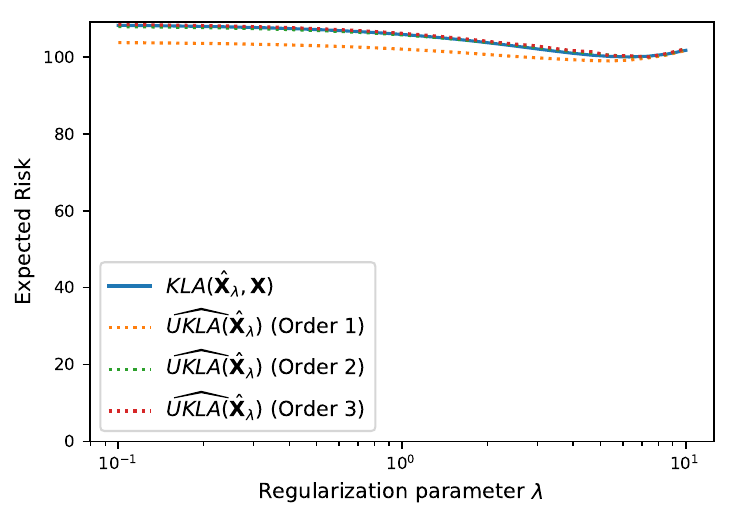}}
  \subfigure[$T=10$]{\includegraphics[width=.48\linewidth]{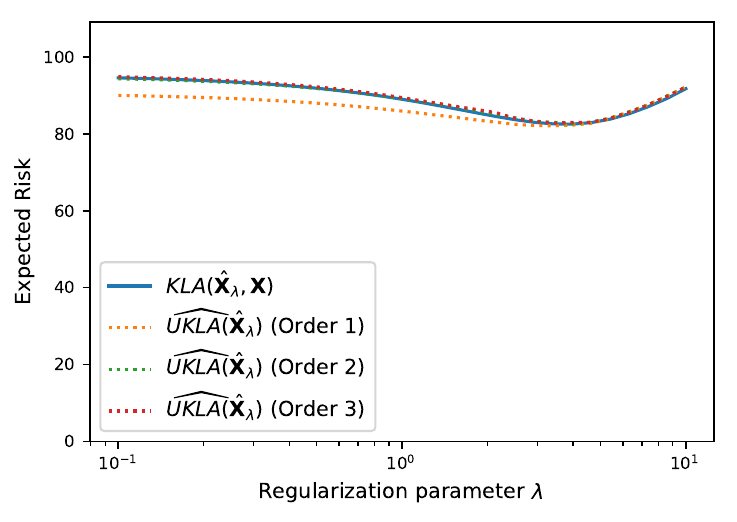}}
  \subfigure[$T=50$]{\includegraphics[width=.48\linewidth]{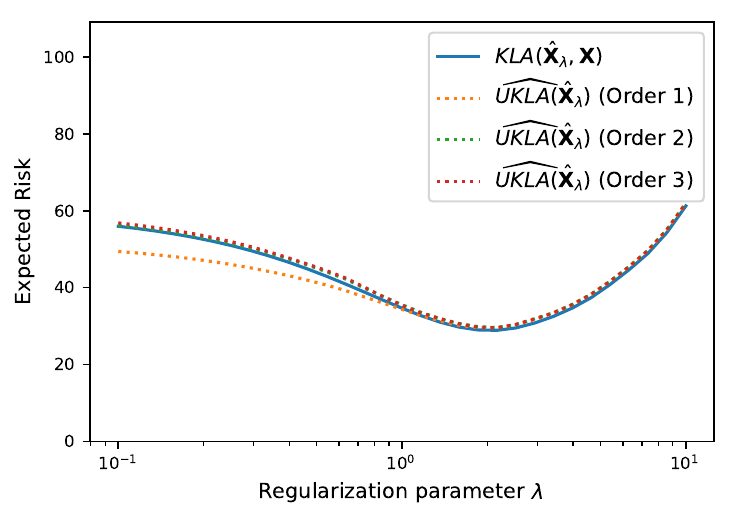}}
  \subfigure[$T=100$]{\includegraphics[width=.48\linewidth]{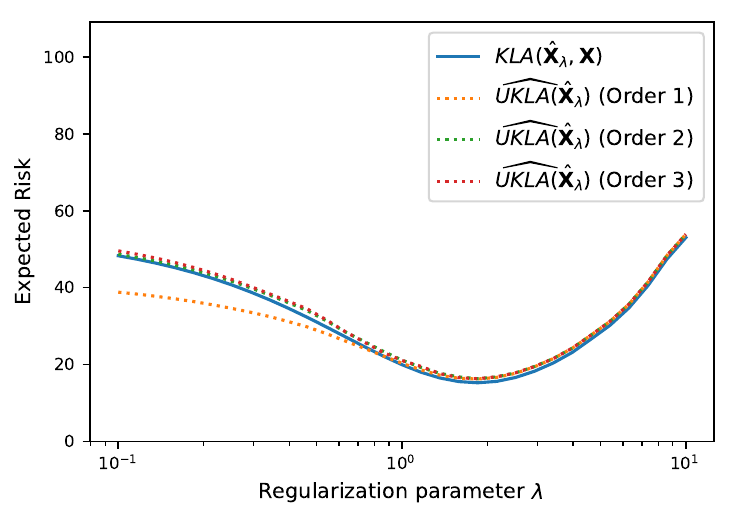}}
  \caption{
    \CM{{\bf (Case 1) - $(m,k) = (400,200)$}.  Results of the low-rank variational estimator on multinomial simulated data with $n_0 = 400$ sampled from the $400 \times 200$ compositional matrix data $\bP$ \eqref{eq:modcase1}.
 Evaluation of $\KLA$ and $\MUKLA$ with respect to the regularization parameter
    $\lambda$ for different orders in the Taylor expansion, and various values of the number $T$ of iterations  of the FISTA algorithm.}}
  \label{fig:results_synthetic_sinusoidal_T_m400_k200}
\end{figure}

\subsection{Analysis of real data}

\subsubsection{Tripadvisor's hotel reviews}

\newcommand{\up}[1]{\hspace{-5em}\makebox[5em][r]{$\color{green}#1\uparrow\quad$}}
\newcommand{\down}[1]{\hspace{-5em}\makebox[5em][r]{$\color{red}#1\downarrow\quad$}}

\begin{table}[htbp]
  \caption{Thirty most frequent words, ordered by decreasing frequencies, used to described the Tripadvisor's hotel review data set according to the Maximum Likelihood estimator.}
  \label{tab:tripadvisor_most_freq_ML}
  \vspace{.5em}
  \centering
  \begin{tabular}{ccp{2em}ccp{2em}cc}
    \cline{1-2} \cline{4-5} \cline{7-8}
    frequency (\%) & word && frequency (\%) & word && frequency (\%) & word\\
    \cline{1-2} \cline{4-5} \cline{7-8}
    3.86 &      hotel && 1.03 &      locat && 0.70 &       walk\\
    3.41 &       room && 0.97 &       nice && 0.67 &      place\\
    2.17 &       veri && 0.94 &       time && 0.66 &       also\\
    2.11 &        not && 0.91 &        day && 0.66 &      other\\
    2.03 &       stay && 0.87 &       just && 0.65 &       make\\
    1.60 &      great && 0.78 &     servic && 0.64 &       well\\
    1.38 &       good && 0.75 &      clean && 0.64 &  breakfast\\
    1.36 &        get && 0.74 &      beach && 0.62 &       food\\
    1.13 &      staff && 0.72 &       onli && 0.62 &     friend\\
    1.03 &      night && 0.71 &    restaur && 0.61 &       pool\\
    \cline{1-2} \cline{4-5} \cline{7-8}
  \end{tabular}
  \vspace{1em}
  \centering
  \caption{Thirty most frequent words, ordered by decreasing frequencies, used to described the Tripadvisor's hotel review data set according to our low rank variational estimator.
    Green and red arrows indicate changes of ranking compared to the Maximum Likelihood estimator.
  }
  \label{tab:tripadvisor_most_freq_LR}
  \vspace{.5em}
  \begin{tabular}{ccp{2em}ccp{2em}cc}
    \cline{1-2} \cline{4-5} \cline{7-8}
    frequency (\%) & word && frequency (\%) & word && frequency (\%) & word\\
    \cline{1-2} \cline{4-5} \cline{7-8}
    3.86 &      hotel && \down{-1}1.02 &      night && \down{-1}0.69 &    restaur \\
    3.40 &       room && 0.96 &       nice && 0.66 &      place \\
    2.16 &       veri && 0.92 &       time && 0.64 &       also \\
    2.10 &        not && 0.89 &        day && 0.64 &      other \\
    2.03 &       stay && 0.85 &       just && 0.63 &       make \\
    1.59 &      great && 0.76 &     servic && \up{+1}0.62 &  breakfast \\
    1.37 &       good && \up{+1}0.76 &      beach && \down{-1}0.62 &       well \\
    1.34 &        get && \down{-1}0.73 &      clean && \up{+2}0.61 &       pool \\
    1.12 &      staff && 0.70 &       onli && \up{+3}0.61 &     resort \\
    \up{+1}1.02 &      locat && \up{+1}0.69 &       walk && \down{-1}0.60 &     friend \\
    \cline{1-2} \cline{4-5} \cline{7-8}
  \end{tabular}
\end{table}
 
\begin{table}[htbp]
  \caption{Thirty most co-occurrent words, ordered by decreasing cosine, used to described
    the Tripadvisor's hotel review data set according to the Maximum Likelihood estimator.}
  \label{tab:tripadvisor_most_corr_ML}
  \centering
  \vspace{.5em}
  \begin{tabular}{cccp{2em}cccc}
    \cline{1-3} \cline{5-7}
    cosine & word1 & word2 && correlation & word1 & word2\\
    \cline{1-3} \cline{5-7}
    0.83 &      front &       desk && 0.63 &  recommend &       help \\
    0.79 &      train &    station && 0.63 &    distanc &       walk \\
    0.77 &       pool &       swim && 0.63 &  entertain &      vacat \\
    0.76 &     flight &    airport && 0.63 &        tip &      vacat \\
    0.75 &      vacat &     resort && 0.62 &     resort &  entertain \\
    0.72 &      beach &      ocean && 0.61 &      staff &       help \\
    0.71 &        kid &      child && 0.61 &       call &       tell \\
    0.70 &     resort &       food && 0.61 &     resort &      beach \\
    0.70 &       food &      lunch && 0.61 &       desk &       call \\
    0.68 &      drink &       food && 0.61 &       help &     friend \\
    0.67 &     ground &     resort && 0.61 &       love &     wonder \\
    0.66 &      vacat &      beach && 0.60 &      peopl &      vacat \\
    0.66 &       read &     review && 0.60 &  entertain &        tip \\
    0.65 &        bad &        not && 0.60 &       food &      vacat \\
    0.65 &      lunch &     resort && 0.60 &        kid &     famili \\
    \cline{1-3} \cline{5-7}
  \end{tabular}
 %\vspace{1em}
 
 \end{table}
 
 \begin{table}[htbp]
  
  \caption{Thirty most co-occurrent words, ordered by decreasing cosine, used to described
    the Tripadvisor's hotel review data set according to our low rank variational estimator.
    Green and red arrows indicate changes of ranking compared to the Maximum Likelihood estimator.}
  \label{tab:tripadvisor_most_corr_LR}
  
%  \vspace{.5em}
  \begin{tabular}{cccp{2em}cccc}
    \cline{1-3} \cline{5-7}
    cosine & word1 & word2 && correlation & word1 & word2\\
    \cline{1-3} \cline{5-7}
    \up{+27}0.91 &  entertain &        tip && \up{+168}0.83 &      bring &       show \\
    \down{-1}0.90 &      front &       desk && \down{-15}0.83 &      train &    station \\
    \up{+6}0.88 &      lunch &       food && \up{+131}0.82 &  entertain &      bring \\
    \up{+15}0.87 &      vacat &        tip && \down{-7}0.81 &     resort &     ground \\
    \down{-2}0.87 &       swim &       pool && \down{-6}0.81 &     review &       read \\
    \up{+64}0.87 &        tip &      bring && \up{+27}0.81 &      vacat &      chair \\
    0.87 &        kid &      child && \up{+164}0.80 &      lunch &      bring \\
    \up{+10}0.87 &      vacat &  entertain && \up{+11}0.80 &      vacat &       week \\
    \up{+51}0.87 &  entertain &       show && \up{+20}0.80 &  entertain &      drink \\
    \up{+21}0.86 &      bring &      vacat && \up{+60}0.80 &  entertain &       week \\
    \up{+87}0.85 &        tip &       show && \up{+76}0.79 &       show &      vacat \\
    \up{+3}0.85 &     resort &      lunch && \down{-17}0.79 &       food &      drink \\
    \down{-8}0.84 &      vacat &     resort && \up{+21}0.79 &       meal &       food \\
    \up{+20}0.84 &      vacat &      lunch && \up{+110}0.79 &        tip &      lunch \\
    \down{-11}0.83 &     flight &    airport && \up{+32}0.79 &      fresh &      fruit \\
    \cline{1-3} \cline{5-7}
  \end{tabular}
\end{table}

\begin{figure}[htbp]
  \centering
  \subfigure[]{\includegraphics[width=.6\linewidth]{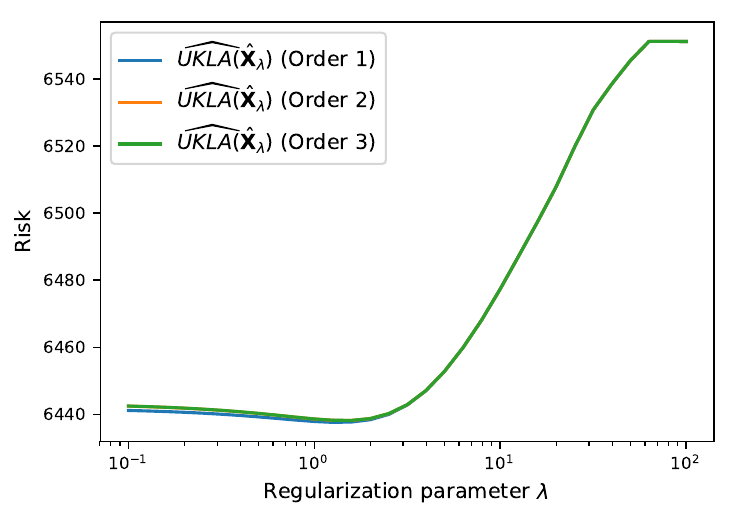}}%
  \subfigure[$\Omega^{-1}(\hat{\bP}_{\lambda})$]{\includegraphics[width=.4\linewidth]{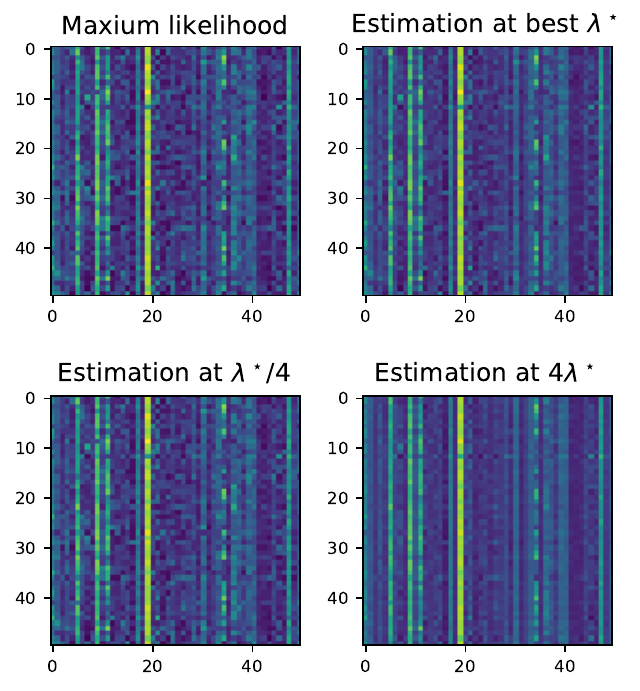}}
   \subfigure[]{\includegraphics[width=.6\linewidth]{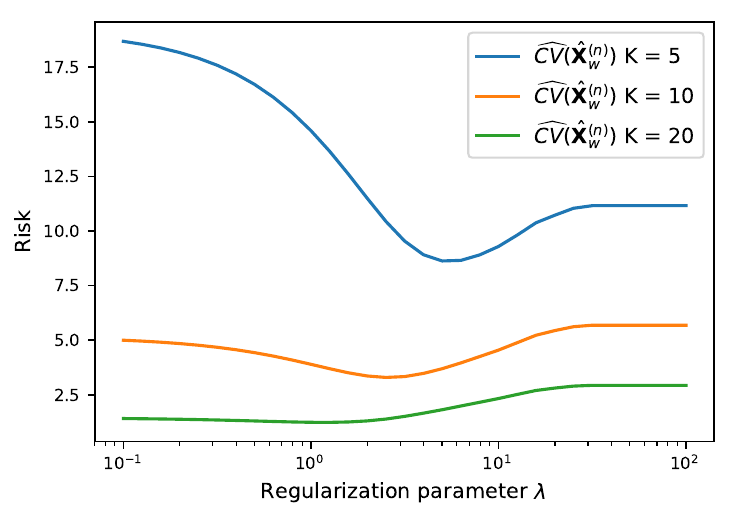}}
  \caption{
    Results of the low-rank variational estimator on Tripadvisor's hotel review data.
    (a) Evaluation of $\MUKLA$ with respect to the regularization parameter
    $\lambda$ for different  orders in the Taylor expansion.
    (b) The upper-left $50 \times 50$ sub-matrices obtained
    by the Maximum Likelihood estimator ($\lambda =0$),
    the results obtained using the optimal parameter $\lambda^{\ast}$,
    $\lambda^{\ast}/4$ and $4 \lambda^{\ast}$, respectively from top to bottom,
    left to right.  {\color{black} (c) Estimation of the KL risk using $K$-fold CV.}
  }
  \label{fig:results_tripadvisor}
\end{figure}

We first consider a survey study related to the reviews by clients of $m$ hotel  where
for each hotel $1 \leq i \leq m$ we count the number of occurrences
of each words (among $k$ words) and $n_i$ is the number of
words for the review of the $i$-th hotel.
We considered the TripAdvisor Data Set described
in \cite{wang2010latent,wang2011latent} that is freely available
at \url{http://times.cs.uiuc.edu/~wang296/Data/}.
This dataset contains reviews for 1,760 hotels.

We extracted all English nouns, verbs (except auxiliaries and modals),
adjectives and adverbs present in these reviews.
We used \texttt{WordNet} lemmatizer algorithm \cite{miller1995wordnet}
to convert words such as ``better'' into ``good'', and ``eating'' into ``eat''.
We kept only words that were at least three characters long and we
used \texttt{Snowball} stemmer algorithm \cite{snowball} to replace words
such as ``beautiful'' and ``beauty'' into their prefix ``beauti''.
%This led to 17,993 distinct such words.
Next, we kept only words used at least 10,000 times, and hotels described
by at least 2,000 words, leading to a dictionary of $k=339$ distinct
words describing $m=1,223$ hotels. We finally build the matrix $\bY$ containing
the number of occurrence of each of the $k$ words for each of the $m$ hotels.
The total number of counts is $\sum_{i=1}^m \sum_{j=1}^k \bY_{ij} = 11\;027\;432$.

Given the occurrence matrix $\bY$ we estimated $\bP$ using the proposed
low-rank variational estimator for different values of regularization
parameter $\lambda$. For each such parameter $\lambda$, we evaluated
our risk estimate $\widehat{\MUKLA}(\hat{\bX}_{\lambda})$ based on Taylor expansion up to order 3,
Rademacher Monte-Carlo estimation, and finite difference approximation.
Results are given on Figure \ref{fig:results_tripadvisor}.
Consistently with the simulations performed in Section \ref{sec:quality_est_MUKLA},
we observe that using an order of 2 in the Taylor expansion is sufficient
for the computation of $\widehat{\MUKLA}(\hat{\bX}_{\lambda})$ (increasing the order does not change the estimation).
As expected, the curves of $\lambda \mapsto \widehat{\MUKLA}(\hat{\bX}_{\lambda})$ indicate that the optimal value
for the parameter $\lambda$ should neither be chosen too small nor too large,
and is $\lambda^\star \approx 1.5$. {\color{black} Figure \ref{fig:results_tripadvisor}(c) shows that   selecting  $\lambda$ by minimizing the CV criterion leads to the choice of a larger value of this regularization parameter, and thus to a smoother estimation than the one obtained by minimizing our criteria based on generalized SURE.}

In order to evaluate on these data the quality of an estimator $\hat{\bP}$ of
$\bP$, we will measure the
frequencies of each word of index $1 \leq j \leq k$ as
\begin{align}
  &\hat{f}_j = \frac{\sum_{i=1}^m n_i \hat{P}_{ij}}{\sum_{i=1}^k n_i \sum_{j'=1}^k \hat{P}_{i,j'}}~.
\end{align}
and the co-occurrence of pairs of words of
indices $1 \leq j_1, j_2 \leq k$ measured by the (centered) cosines of columns of $\hat{\bP}$ and given as
\begin{align}
  \hat{c}_{j_1,j_2} = \frac{
    \sum_{i=1}^m
    (\hat{P}_{i,j_1} - \frac1m \sum_{i'=1}^m \hat{P}_{i',j_1})
    (\hat{P}_{i,j_2} - \frac1m \sum_{i'=1}^m \hat{P}_{i',j_2})
  }{
    \sqrt{
    \sum_{i=1}^m
    (\hat{P}_{i,j_1} - \frac1m \sum_{i'=1}^m \hat{P}_{i',j_1})^2
    }
    \sqrt{
    \sum_{i=1}^m
    (\hat{P}_{i,j_2} - \frac1m \sum_{i'=1}^m \hat{P}_{i',j_2})^2
    }
  }~.
\end{align}
We compare these statistics for two estimators of $\bP$: the Maximum Likelihood
$\hat{\bP}^{\rm ML} = \diag(n_1, \ldots, n_m)^{-1} \bY$,
and the low rank variational estimator
$\hat{\bP}_{\lambda^\star}$.
Table \ref{tab:tripadvisor_most_freq_ML} and \ref{tab:tripadvisor_most_freq_LR}
show respectively the list of the thirty most frequent words according to
both the Maximum Likelihood estimator and the low rank variational estimator.
Table \ref{tab:tripadvisor_most_corr_ML} and \ref{tab:tripadvisor_most_corr_LR}
show respectively the list of the thirty most co-occurrent pairs of words according to
both the Maximum Likelihood estimator and the low rank variational estimator.
Regarding frequency analysis, unsurprisingly ``hotel'' and ``room'' appears as
the most frequent words for both estimators, and only subtle differences seem to exist.
Regarding the co-occurrence analysis, unsurprisingly (``front'', ``desk'') and
(``train'', ``station'') appears highly co-occurrent for both estimators, but
we observe that correlations are significantly reinforced for the low rank
variational estimator, as well as co-occurrence for words such as ``bring'', ``tip'',
``entertain'' and ``show''.

\subsubsection{Metagenomics data}

We propose to apply our approach to the analysis of metagenomics data
on a Cross-sectional study Of diet and stool MicroBiOme (COMBO) composition
\cite{Wu105} that have been studied in \cite{Cao17}. In this study,
DNAs from stool samples of $m=98$ healthy volunteers were analyzed by
metagenomics sequencing which yields an average of 9265 reads per
sample (with a standard deviation of 386) and led to identifying
$k=87$ bacterial genera presented in at least one sample. As argued in
\cite{Cao17} the resulting count data matrix has many zeros which are
likely due to under-sampling. The analysis in \cite{Cao17} shows that  supposing that the true
composition matrix is approximately low rank is a reasonable assumption.

Our approach by  regularized maximum likelihood estimation is then applied to the resulting
count data matrix $\bY$ ($k = 87$ bacterial genera over $m = 98$ samples) to estimate the unknown composition matrix $\bP$  using different values of the regularization
parameter $\lambda$. For each value of $\lambda$, we evaluated
the unbiased risk estimate $\widehat{\MUKLA}(\hat{\bX}_{\lambda})$ based on a Taylor expansion {\color{black} up to order 6},
Rademacher Monte-Carlo estimation, and finite difference approximation.
The results are displayed on Figure \ref{fig:results_BMI}.
We have observed that {\color{black} using an order of 1} in the Taylor expansion is sufficient
for the approximation of $\widehat{\MUKLA}(\hat{\bX}_{\lambda})$ as increasing the order does not change the results.

It can be seen that  the curve $\lambda \mapsto \widehat{\MUKLA}(\hat{\bX}_{\lambda})$ is increasing, which indicates that the optimal value for the regularization parameter  is $\lambda^\star = 0$. Hence, this suggests that the best approach for this dataset is to do ML estimation without any regularization. % Using the cross-validation approach from \cite{Cao17} to select the hyperparameters $\mu$ and $\alpha$ for the estimator $\hat{\bP}_{\mu,\alpha}$ defined in \eqref{eq:estCao17}, we also found that the best estimator is given by ML with no regularization.
This (surprising)  result can be explained by the fact that, in each row of the data matrix $\bY$,  a few columns contain very large counts which suggest that, for each row $i$ of the underlying composition matrix $\bP$, a few entries  $p_{ij}$ have  large values  while all the others are close to zero. In the setting of modeling such data as being sampled from a multinomial distribution, the un-regularized ML approach thus yields an estimator with the smallest KL risk.  {\color{black} Surprisingly,  Figure \ref{fig:results_BMI}(c,d) show that   selecting  $\lambda$ by minimizing $\lambda \mapsto \widehat{\CV}(\hat{\bX}_{\lambda}) $  leads to choose larger values of this regularization parameter, and thus cross-validation yields an estimator that is much smoother than ML estimation as displayed in  Figure \ref{fig:results_BMI}(b).}

Finally, we highlight the potential benefits of low-rank regularization for this data set as follows by proceeding as in \cite[Section 6]{Cao17}. We display in Figure \ref{fig:results_BMI_boxplot} the boxplots of the values (in logarithmic scale) of each column  of the estimator $\hat{\bX}_{\lambda}$ for different $\lambda \in \{0,4.64,100\}$ by distinguishing, for each $1 \leq j \leq k$, the rows $i$ of $\hat{\bX}_{\lambda}$ that belong either to the set $\Omega_{j}$ or $\Omega_{j}^{c}$ where counts are positive or zero in the data matrix that is
$$
\Omega_{j} = \{i \; : \; \bY_{ij} > 0 \} \mbox{ and } \Omega_{j}^{c} = \{i \; : \; \bY_{ij} = 0 \}.
$$
For  $\lambda \in \{4.64,100\}$ and each $1 \leq  j \leq k$, the values of $\hat{\bX}_{\lambda}$ in $\Omega_{j}^{c}$ are shrank towards those in $\Omega_{j}$ showing that taking increasing values of $\lambda > 0$ allows to perform zero-replacement in a data-driven manner.

\begin{figure}[htbp]
  \centering
  \subfigure[]{\includegraphics[width=.6\linewidth]{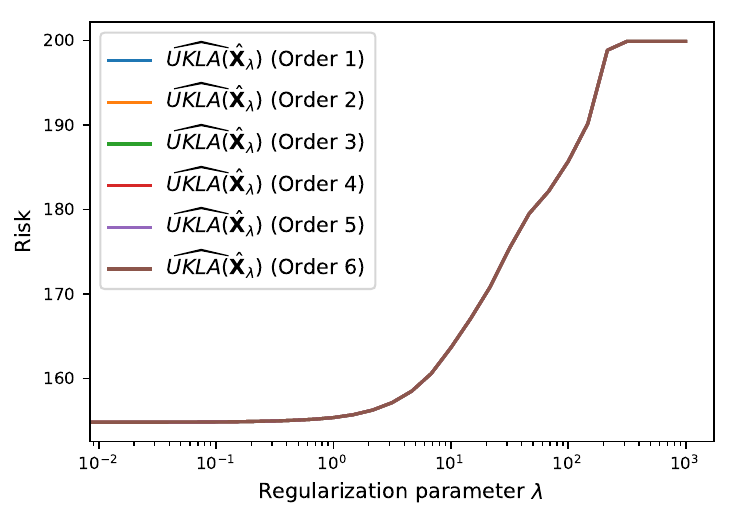}}%
  \subfigure[$\Omega^{-1}(\hat{\bP}_{\lambda})$]{\includegraphics[width=.4\linewidth]{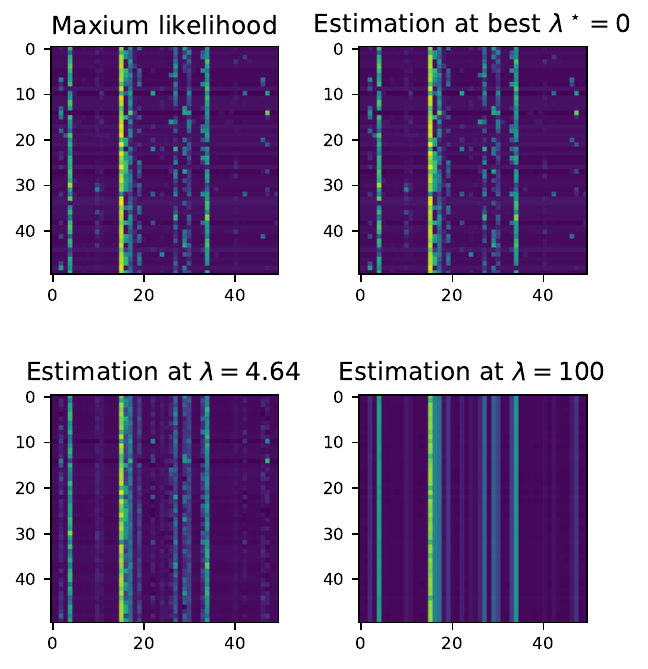}}
    \subfigure[]{\includegraphics[width=.45\linewidth]{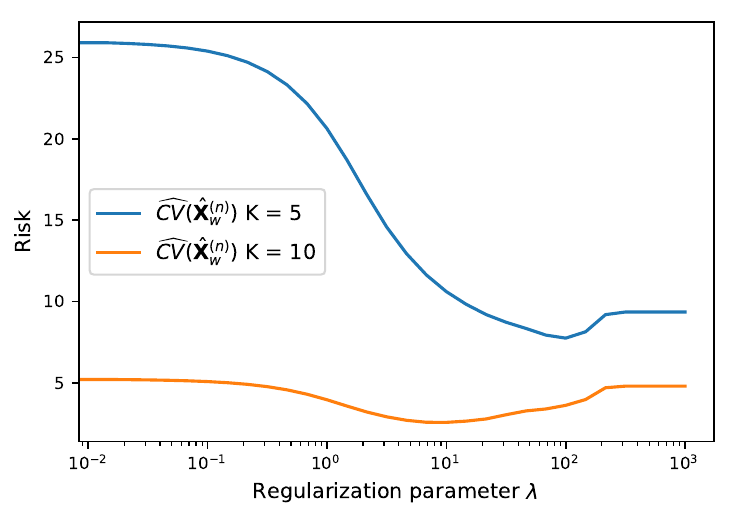}}%
      \subfigure[]{\includegraphics[width=.45\linewidth]{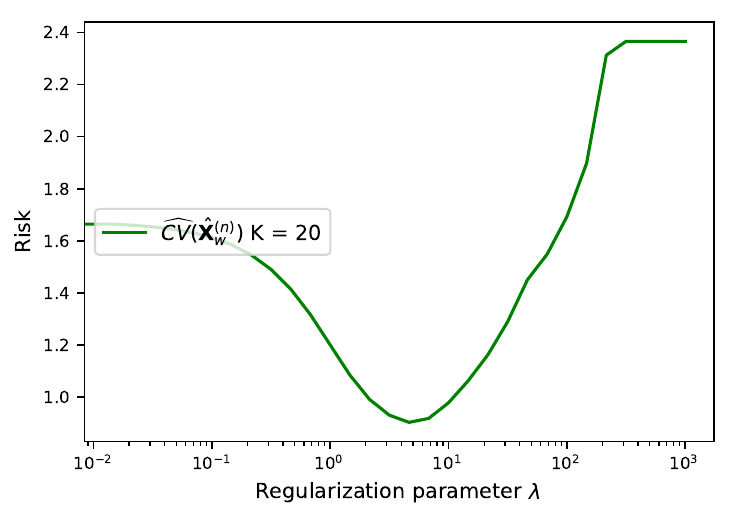}}%
  \caption{
    Results of the low-rank variational estimator on the COMBO dataset.
    (a) Evaluation of $\MUKLA$ with respect to the regularization parameter
    $\lambda$ for different  orders in the Taylor expansion.
    (b) The upper-left $50 \times 50$ sub-matrices obtained
    by the Maximum Likelihood estimator,
    the results obtained using the optimal parameter $\lambda^\ast = 0$,
    $\lambda = 5.4$ and $\lambda = 10$, respectively from top to bottom,
    left to right. {\color{black} (c,d) Estimation of the KL risk using $K$-fold cross-validation. For $K \in \{5,10,20\}$ the minimum of the CV criterion is reached for $\lambda \in \{10^2,10,4.64\}$ respectively.} The estimator  $\hat{\bX}_{\lambda}$ has been computed using the FISTA algorithm with  $T=1000$ iterations. 
  }
  \label{fig:results_BMI}
\end{figure}

\begin{figure}[htbp]
  \centering
  \subfigure[$\lambda = 0$]{\includegraphics[width=1\linewidth]{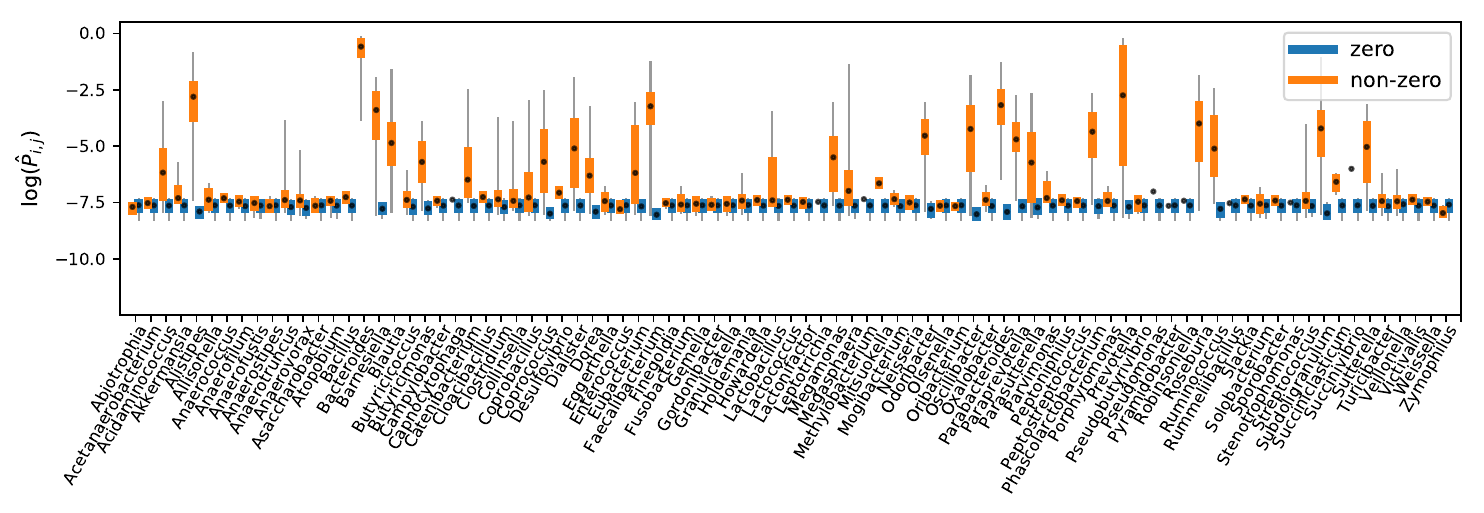}}
   \subfigure[$\lambda = 4.64$]{\includegraphics[width=1\linewidth]{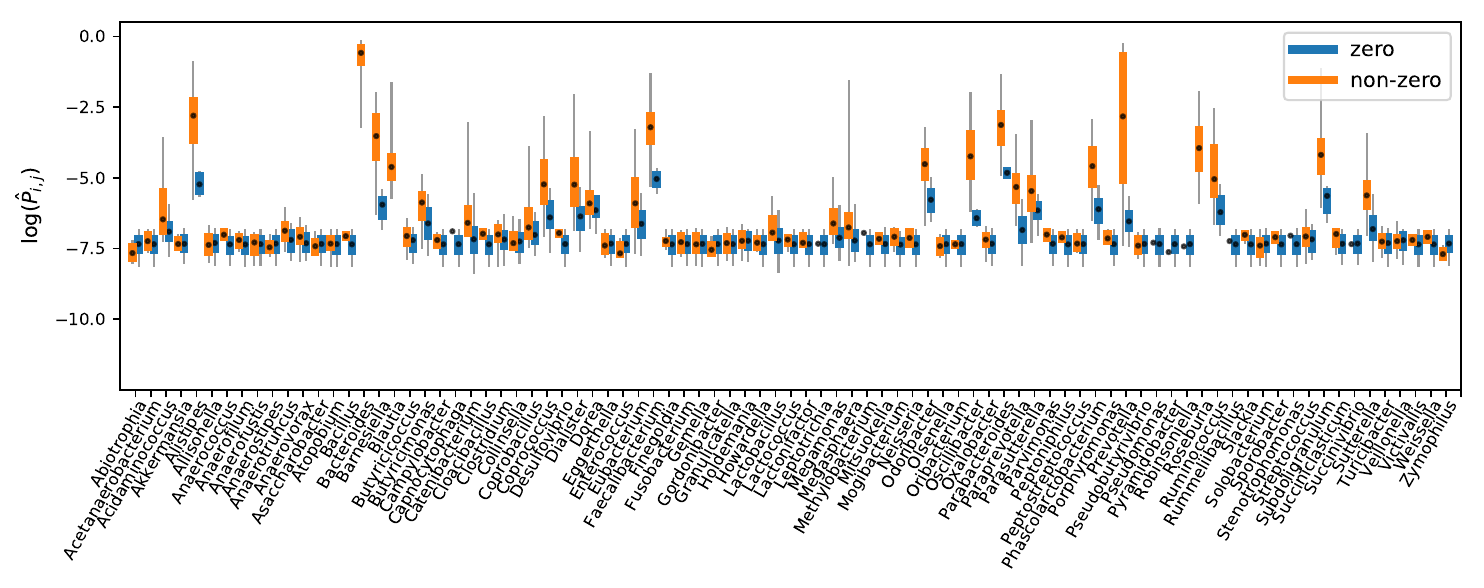}}
    \subfigure[$\lambda = 100$]{\includegraphics[width=1\linewidth]{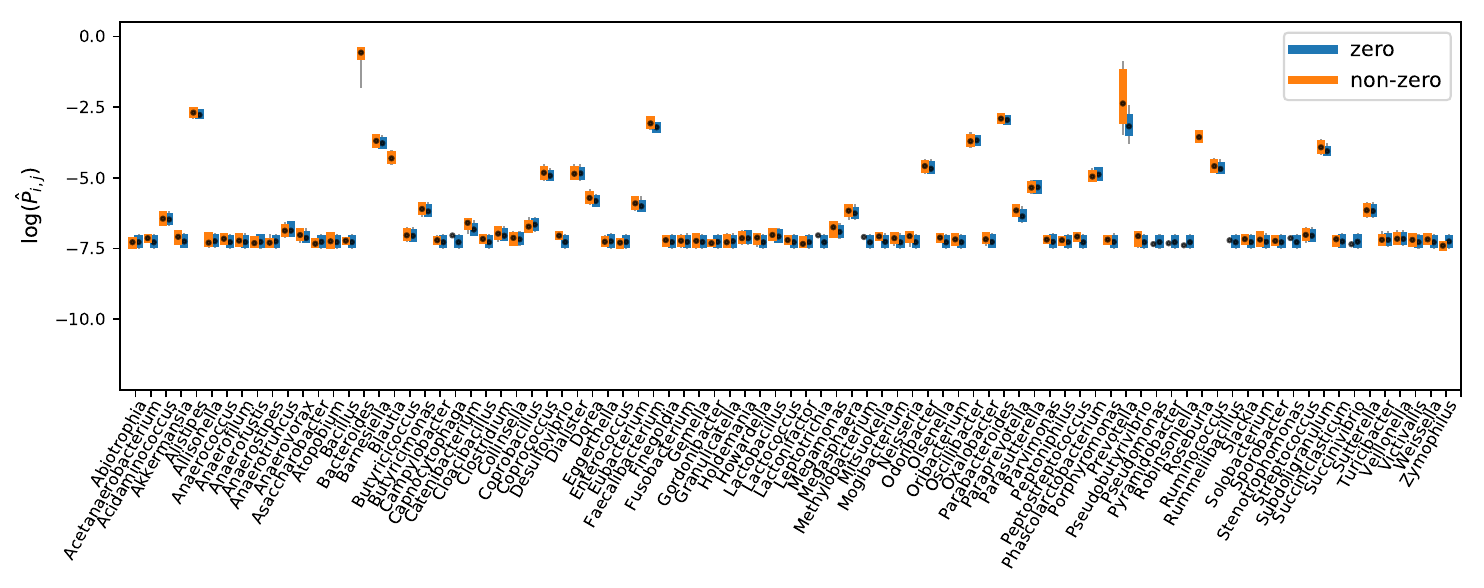}}
  \caption{
    Effects of low-rank variational estimation on the COMBO dataset. Boxplots of the values (in logarithmic scale) of the entries of $\hat{\bX}_{\lambda}$ for each bacteria $j$ within the sets $\Omega_j$ and $\Omega_j^{c}$ of  non-zero and zero observations for different values of the regularization parameter $\lambda$.
  }
  \label{fig:results_BMI_boxplot}
\end{figure}

\bibliographystyle{alpha}
\bibliography{Discrete_GSURE_LowRank_final}

\end{document}